\documentclass{amsart}
\usepackage{amssymb}
\usepackage[all,cmtip]{xy}
\input diagxy
\usepackage[linktocpage=true,backref]{hyperref}
\usepackage{todonotes}

\let\oldtocsection=\tocsection

\let\oldtocsubsection=\tocsubsection

\renewcommand{\tocsection}[2]{\hspace{0em}\oldtocsection{#1}{#2}}
\renewcommand{\tocsubsection}[2]{\hspace{2em}\oldtocsubsection{#1}{#2}}

\newtheorem{theorem}{Theorem}[section]
\newtheorem{proposition}[theorem]{Proposition}
\newtheorem{corollary}[theorem]{Corollary}
\newtheorem{lemma}[theorem]{Lemma}

\theoremstyle{definition}

\newtheorem{definition}[theorem]{Definition}
\newtheorem{example}[theorem]{Example}

\numberwithin{equation}{section}

\theoremstyle{remark}
\newtheorem{remark}[theorem]{Remark}

\DeclareMathOperator*{\colim}{colim}
\DeclareMathOperator{\coker}{coker}

\newcommand{\Dir}{\mathsf{Dir}^a}
\newcommand{\Dirk}{\mathsf{Dir}^a_{\kappa}}

\newcommand{\Invk}{\mathsf{Inv}^a_\kappa}
\newcommand{\elTate}{\mathsf{Tate}^{el}}
\newcommand{\elTatek}{\mathsf{Tate}^{el}_{\kappa}}
\newcommand{\elTatec}{\mathsf{Tate}^{el}_{\aleph_0}}
\newcommand{\nelTate}{\mathsf{n}\text{-}\mathsf{Tate}^{el}}
\newcommand{\nelTatek}{\mathsf{n}\text{-}\mathsf{Tate}^{el}_\kappa}
\newcommand{\nelTatec}{\mathsf{n}\text{-}\mathsf{Tate}^{el}_{\aleph_0}}
\newcommand{\Tate}{\mathsf{Tate}}
\newcommand{\Tatek}{\mathsf{Tate}_{\kappa}}
\newcommand{\Tatec}{\mathsf{Tate}_{\aleph_0}}
\newcommand{\nTate}{\mathsf{n}\text{-}\mathsf{Tate}}
\newcommand{\nTatek}{\mathsf{n}\text{-}\mathsf{Tate}_\kappa}
\newcommand{\nTatec}{\mathsf{n}\text{-}\mathsf{Tate}_{\aleph_0}}
\newcommand{\elDrate}{\mathsf{Tate}^{Dr,el}}
\newcommand{\Drate}{\mathsf{Tate}^{Dr}}

\newcommand{\Dratec}{\mathsf{Tate}_{\aleph_0}^{Dr}}

\newcommand{\Ind}{\mathsf{Ind}^a}
\newcommand{\Indk}{\mathsf{Ind}^a_{\kappa}}
\newcommand{\Indc}{\mathsf{Ind}^a_{\aleph_0}}
\newcommand{\FM}{\mathcal{FM}}
\newcommand{\FMk}{\mathcal{FM}_\kappa}
\newcommand{\Pro}{\mathsf{Pro}^a}
\newcommand{\Prok}{\mathsf{Pro}^a_{\kappa}}
\newcommand{\Proc}{\mathsf{Pro}^a_{\aleph_0}}

\newcommand{\ic}{ic}
\newcommand{\Mod}{\mathsf{Mod}}
\newcommand{\lex}{\mathsf{Lex}}
\newcommand{\ab}{\mathsf{Ab}}
\newcommand{\Gr}{Gr}
\newcommand{\Vect}{\mathsf{Vect}}

\newcommand{\Frac}{Frac}

\newcommand{\Calk}{\mathsf{Calk}}
\newcommand{\Calkk}{\mathsf{Calk}_\kappa}
\newcommand{\Calkc}{\mathsf{Calk}_{\aleph_0}}
\newcommand{\End}{\mathsf{End}}
\newcommand{\Mat}{\mathsf{Mat}}
\newcommand{\Fun}{\mathsf{Fun}}
\newcommand{\into}{\hookrightarrow}
\newcommand{\onto}{\twoheadrightarrow}

\newcommand{\op}{op}
\newcommand{\Ab}{\mathbb{A}}
\DeclareMathOperator{\QCoh}{QCoh}
\DeclareMathOperator{\Coh}{Coh}

\newcommand{\Cc}{\mathcal{C}}
\newcommand{\Dc}{\mathcal{D}}
\newcommand{\Ec}{\mathcal{E}}
\newcommand{\Fc}{\mathcal{F}}
\newcommand{\Oc}{\mathcal{O}}

\newcommand{\Nb}{\mathbb{N}}

\DeclareMathOperator{\Spec}{Spec}

\begin{document}
\title{Tate Objects in Exact Categories}
\author{Oliver Braunling, Michael Groechenig, and Jesse Wolfson}
\dedicatory{\rm With an appendix by \textsc{Jan \v S\v tov\'\i\v cek} and \textsc{Jan Trlifaj}}

\address{Department of Mathematics, Universit\"{a}t Duisberg-Essen}
\email{oliver.braeunling@uni.due.de}
\address{Department of Mathematics, Imperial College London}
\email{m.groechenig@imperial.ac.uk}
\address{Department of Mathematics, Northwestern University}
\email{wolfson@math.northwestern.edu}

\address{Department of Algebra, Faculty of Mathematics and Physics, Charles University\\
Sokolovsk\'{a} 83, 186 75 Praha~8, Czech Republic}
\email{stovicek@karlin.mff.cuni.cz}

\address{Department of Algebra, Faculty of Mathematics and Physics, Charles University\\
Sokolovsk\'{a} 83, 186 75 Praha~8, Czech Republic}
\email{trlifaj@karlin.mff.cuni.cz}

\begin{abstract}
    We study elementary Tate objects in an exact category. We characterize the category of elementary Tate objects as the smallest sub-category of admissible Ind-Pro objects which contains the categories of admissible Ind-objects and admissible Pro-objects, and which is closed under extensions.  We compare Beilinson's approach to Tate modules to Drinfeld's. We establish several properties of the Sato Grassmannian of an elementary Tate object in an idempotent complete exact category (e.g. it is a directed poset). We conclude with a brief treatment of $n$-Tate modules and $n$-dimensional ad\`{e}les.

    An appendix due to J. \v S\v tov\'\i\v cek and J. Trlifaj identifies the category of flat Mittag-Leffler modules with the idempotent completion of the category of admissible Ind-objects in the category of finitely generated projective modules.
\end{abstract}

\thanks{O.B.\ was supported by DFG SFB/TR 45 ``Periods, moduli spaces and arithmetic of algebraic varieties'' and Alexander von Humboldt Foundation. J.W.\ was partially supported by an NSF Graduate Research Fellowship under Grant No.\ DGE-0824162, and by an NSF Research Training Group in the Mathematical Sciences under Grant No.\ DMS-0636646. He was a guest of Kyoji Saito at IPMU while this paper was being completed. J.\v{S} and J.T. were supported by GA\v{C}R P201/12/G028.}


\maketitle
\tableofcontents

\section{Introduction}
In the article \emph{Residues of Differentials on Curves} \cite{Tat:68} J.\ Tate developed a new understanding of the classical theory of residues. Given a differential $fdg$ on a curve $X$ defined over a field $k$, the residue at a point $x$ can be defined as the trace of a suitable operator assigned to $f$ and $g$ acting on the infinite-dimensional vector space $\widehat{F}_x \cong k_x((t))$. This approach to residues immediately implies the independence of local coordinates, and it allowed Tate to give an intrinsic proof of the \emph{sum-of-residues theorem}. The work of Parshin \cite{Par:76}, Arbarello--de Concini--Kac \cite{AdK:87} and Beilinson \cite{Bei:80} brought Tate's techniques to new heights.

A conceptual approach to the infinite-dimensional vector space $k((t))$ is provided by Lefschetz's theory of \emph{locally linearly compact} vector spaces \cite[Chapter II.6]{Lef:42}. A topological vector space $U$ over a discrete field $k$ is said to be \emph{discrete} if it has the discrete topology. The \emph{topological dual} $U^{\vee}$ of a discrete vector space is called a \emph{linearly compact} vector space. A \emph{locally linearly compact} vector space $W$ can be written as an extension
$$0 \to U^{\vee} \to W \to V \to 0$$
of a discrete vector space $V$ by a \emph{linearly compact} vector space $U^{\vee}$.

For the example of formal Laurent series one endows $k((t))$ with the finest linear topology such that $t^n \rightarrow 0$ for $n \rightarrow \infty$. The aforementioned extension is induced by the direct sum decomposition $k((t)) = k[[t]] \oplus k((t))/k[[t]]$.

While the theory of locally linearly compact vector spaces may be sufficient for the purpose of (equal characteristic) algebraic geometry, arithmetic considerations necessitate an analogous treatment of \emph{locally linearly compact} abelian groups. Indeed, the short exact sequence
$$0 \to \mathbb{Z}_p \to \mathbb{Q}_p \to \mathbb{Q}_p/\mathbb{Z}_p \to 0$$
certainly realizes the $p$-adic numbers $\mathbb{Q}_p$ as an extension of the \emph{discrete} abelian group $\mathbb{Q}_p/\mathbb{Z}_p$ by the \emph{compact} abelian group $\mathbb{Z}_p$.

Moreover, Parshin and Beilinson's theory of ad\`eles (\cite{Par:76},\cite{Bei:80}) and Drinfeld's theory of infinite-dimensional vector bundles \cite{Dri:06} suggest that the constructions above should be \emph{iterated} and \emph{studied in families}.

For iteration, there are many examples of the objects we wish to describe, for instance, the vector space $k((t_1))\cdots((t_n))$. The problem is to find the correct notion of morphisms between these objects, i.e. to formulate an appropriate categorical framework. While a ``semi-topological'' approach suffices for applications to residues, as in \cite{Yek:92}, it has been known for some time (e.g. \cite{Kat:00}), that the notion of topology is insufficient to fully describe ad\`{e}les and higher local fields above dimension 2. One solution, proposed by Beilinson \cite{Bei:87} and, more recently, Kato \cite{Kat:00}, is to recast locally linear compact objects in terms of formal limits and colimits. This allows one to define a notion of \emph{Tate objects} in an arbitrary category, and, by recursion, to consider $n$-Tate objects.

In this article, we develop the properties of Ind, Pro and Tate objects in arbitrary exact categories. If one views an exact category $\Cc$ as a non-commutative analogue of the category of finite dimensional vector bundles on a scheme, then one can think of Ind, Pro and Tate objects in $\Cc$ as modeling families of discrete, linearly compact, and locally linearly compact objects respectively.

\subsubsection*{Ind-Objects}
We begin with the category $\Ind(\Cc)$, whose objects are certain formal colimits in $\Cc$ (Definition \ref{defi:ind}). We establish its main properties in Section \ref{sec:ind}, and then introduce two related categories: the category $\FM(\Cc)$ whose objects are direct summands of objects in $\Ind(\Cc)$, and a full sub-category $P(\Cc)\subset\FM(\Cc)$.

One can understand these categories as generalizations of familiar categories of modules. Denote by $P_f(R)$ the category of finitely generated projective (left) modules over a ring, by $P(R)$ the category of all projective (left) modules, and by $\FM(R)$ the category of flat Mittag-Leffler (left) modules.
\begin{theorem}[\ref{cor:projRgeneral}, \ref{p:more}]
    Let $R$ be a ring.
    \begin{enumerate}
        \item The categories $P(R)$ and $P(P_f(R))$ are equivalent.
        \item (\v S\v tov\'\i\v cek, Trlifaj) The categories $\FM(R)$ and $\FM(P_f(R))$ are equivalent.
    \end{enumerate}
\end{theorem}

\subsubsection*{Pro-Objects}
We now turn to the category $\Pro(\Cc)$, whose objects are certain formal limits in $\Cc$. The properties of $\Pro(\Cc)$ follow by formal duality from those of $\Ind(\Cc)$, and we summarize them in Theorem \ref{thm:prox}. We also introduce two related categories: $\FM^\vee(\Cc)$ and a full sub-category $P^\vee(\Cc)\subset\FM^\vee(\Cc)$.

Denote by $P^\vee(R^\circ)$ the category of topological duals of discrete projective (right) $R$-modules, and by $\FM(R^\circ)$ the category of discrete flat Mittag-Leffler (right) modules.
\begin{theorem}[\ref{cor:projdualinproic}, \ref{prop:inddualpro}]
    Let $R$ be a ring.
    \begin{enumerate}
        \item The categories $P^\vee(R^\circ)$ and $P^\vee(P_f(R))$ are equivalent.
        \item The duality $\hom(-,R)\colon P_f(R)^{\op}\to^\simeq P_f(R)$ extends to a pair of exact duality equivalences
            \begin{equation*}
                \hom(-,R)\colon\FM(R^\circ)^{\op}\to^\simeq\FM^\vee(P_f(R))\colon\hom(-,R).
            \end{equation*}
    \end{enumerate}
\end{theorem}
The duality between $\FM(R^\circ)$ and $\FM^\vee(P_f(R))$ stands in contrast to the pathologies of topological duality for flat Mittag-Leffler modules, such as the failure to preserve exact sequences. The categorical approach avoids these, and it ensures that the properties of $\FM^\vee(\Cc)$ follow formally from the properties we establish for $\FM(\Cc)$.

\subsubsection*{Tate Objects}
The category $\elTate(\Cc)$ of \emph{elementary Tate objects} was introduced by Beilinson \cite{Bei:87} following Tate's treatment of residues. While the following characterization of $\elTate(\Cc)$ should not be surprising to experts, we are unable to find a previous instance of it in print.
\begin{theorem}[\ref{thm:eltatechar}]
    Let $\Cc$ be an exact category. The category $\elTate(\Cc)$ is the smallest full sub-category of $\Ind(\Pro(\Cc))$ which
    \begin{enumerate}
        \item contains the sub-category $\Ind(\Cc)\subset\Ind(\Pro(\Cc))$,
        \item contains the sub-category $\Pro(\Cc)\subset\Ind(\Pro(\Cc))$, and
        \item is closed under extensions.
    \end{enumerate}
\end{theorem}
Theorem \ref{thm:eltatechar} provides a basic tool for producing Tate objects in practice, and forms the basis for our approach to Beilinson--Parshin ad\`{e}les (Section \ref{sec:bp}). It also indicates how the theory of Tate objects should be generalized to homological settings. While we do not pursue this here, this generalization allows for applications to the study of perfect complexes and the algebraic $K$-theory of schemes.\footnote{D. Clausen has also informed us that he uses this generalization in on-going work on Artin reciprocity.}

We now introduce the category $\Tate(\Cc)$, whose objects are direct summands of elementary Tate objects. Drinfeld \cite{Dri:06} recently proposed a full sub-category $\Drate(R)$ of the category of topological $R$-modules as a category of Tate modules.\footnote{A module in $\Drate(R)$ is a topological direct summand of a module $P\oplus Q^\vee$, where $P$ and $Q$ are discrete projective modules and $Q^\vee$ is the topological dual.}${}^,~$\footnote{As an unfortunate consequence of Tate's mathematical creativity, there are also several other and very different notions of ``Tate module''.}

\begin{theorem}[\ref{thm:drincomp}]
    Let $R$ be a ring. There exists a fully faithful embedding
    \begin{equation*}
        \Drate(R)\into \Tate(P_f(R)).
    \end{equation*}
\end{theorem}

\subsubsection*{Sato Grassmannians}
The linear algebra of Tate spaces admits a strong analogy to that of Hilbert space. A key object in this analogy is the \emph{Sato Grassmannian}, which parametrizes certain sub-spaces of a Tate space called \emph{lattices}; a basic example is the sub-space $k[[t]]\subset k((t))$. The two key features of the Sato Grassmannian are that a) the quotient of a lattice by a sub-lattice is finite dimensional, and b) given any two lattices $L_0$ and $L_1$ in a Tate space, there exists a common sub-lattice $N$ contained in $L_0$ and $L_1$, and also a common enveloping lattice containing them both.

Sato Grassmannians admit a natural generalization to \emph{elementary Tate objects}. The first feature holds in general (Proposition \ref{prop:latcomm}). For the second, we show the following.
\begin{theorem}[\ref{thm:grdir}]
    Let $\Cc$ be an idempotent complete exact category. The poset underlying the Sato Grassmannian $\Gr(V)$ of an elementary Tate object $V$ in $\Cc$ is directed and co-directed.
\end{theorem}
We view this as the most important theorem of this paper. For a Tate vector space $V$, a key insight of \cite{SeW:85} is that a pair of lattices $L_0$ and $L_1$ of $V$ can be viewed as an analogue of a Fredholm operator on Hilbert space. The \emph{index space} of this ``operator'' is the finite dimensional $\mathbb{Z}/2$-graded vector space
\begin{equation*}
    L_0/(L_0\cap L_1)\oplus L_1/(L_0\cap L_1).
\end{equation*}
Similarly, a pair of lattices $L_0,~L_1$ in a Tate module $V\in\elTate(P_f(R))$ can be thought of as a family of Fredholm operators parametrized by $\Spec(R)$. The theorem guarantees that a common sub-lattice $N$ of $L_0$ and $L_1$ exists, and one can therefore define the \emph{index bundle} of $L_0$ and $L_1$ to be the finitely generated $\mathbb{Z}/2$-graded projective module
\begin{equation*}
    L_0/N\oplus L_1/N.
\end{equation*}
This definition recalls Atiyah's construction of the index of a continuous family of Fredholm operators \cite[Appendix A]{Ati:67}. We develop this analogy further in \cite{BGW:13}, where we show that the assignment of an index bundle to a pair of lattices extends, independent of the choice of sub-lattice, to a natural map from $\Gr(V)\times\Gr(V)$ to the algebraic $K$-theory space of $R$.

\subsubsection*{n-Tate Objects}
For any exact category $\Cc$, the category $\Tate(\Cc)$ is also an exact category. We can therefore define $n\text{-}\Tate(\Cc)$ to be the category of Tate objects in $(n-1)\text{-}\Tate(\Cc)$. Our earlier results imply that the categories $n\text{-}\Tate(\Cc)$ satisfy the expected properties for all $n$ (see Theorem \ref{thm:ntatex}). We conclude by exhibiting a principle example of interest.\footnote{That the ad\`{e}les form an $n$-Tate object should not be surprising to experts, but we have been unable to find it in the literature.}

\begin{theorem}[\ref{thm:adeles}]
    Let $X$ be an $n$-dimensional Noetherian scheme. Denote by $\Coh_0(X)\subset\Coh(X)$ the full sub-category of sheaves having 0-dimensional support. The $n$-dimensional ad\`{e}les give an exact functor
    \begin{equation*}
        \begin{xy}
            \morphism<1000,0>[\Coh(X)`\nelTate(\Coh_0(X));\Ab^n_X(-)]
        \end{xy}.
    \end{equation*}
\end{theorem}

\subsubsection*{History}
The categories $\Ind(\Cc)$, $\Pro(\Cc)$ and $\Tate(\Cc)$ play a great role in various applications, notably in algebraic $K$-theory (e.g.\ \cite{Sch:04}, \cite{Sai:13}, \cite{BGW:13}), in integrable hierarchies of differential equations (e.g.\ \cite{SaS:83}, \cite{Sat:87}), in chiral algebras \cite{BeD:04}, in Drinfeld's study of infinite-dimensional vector bundles in algebraic geometry \cite{Dri:06}, in Parshin and Beilinson's theory of multidimensional ad\`{e}les of schemes (\cite{Par:76},\cite{Bei:80}), in reciprocity laws (e.g.\ \cite{OsZ:11},\cite{OsZ:13},\cite{BGW:14}), in de Rham epsilon factors \cite{BBE:02}, and in the representation theory of double loop groups (e.g.\ \cite{ArK:10}, \cite{FrZ:12}).\footnote{Yekutieli \cite{Yek:06} has also introduced a related construction of ``Dir-Inv-modules''. The category of countable admissible Ind-Pro objects in the category of $R$-modules is equivalent to the full sub-category of \emph{complete} Dir-Inv $R$-modules.}${}^,~$\footnote{Tate objects in categories of coherent sheaves have also appeared in \cite[3.1]{KOZ:09} in the context of a two-dimensional Krichever correspondence; in \emph{loc. cit.} they are referred to as ``ind-pro coherent sheaves.''} A non-linear analogue of $\elTate(\Cc)$, due to Kato \cite{Kat:00}, plays a key role in the study of higher local fields, of algebraic groups over them (e.g.\ \cite{GaK:04}), and of formal loop spaces of schemes (e.g.\ \cite{KaV:04}).

The notion of \emph{Tate objects} in an arbitrary exact category was introduced by Beilinson \cite{Bei:87} and has also been recently studied by Previdi in \cite{Pre:11}; if we restrict to countable indexing diagrams, our ``elementary Tate objects'' coincide with their approach (Proposition \ref{prop:eltatec=bei}). Moving beyond countable diagrams allows us to treat examples such as the ad\`{e}les of curves over fields of uncountable cardinality.

Sato Grassmannians were introduced by Sato and Sato in their study of integrable hierarchies of differential equations (e.g.\ \cite{Sat:87}). Segal and Wilson \cite{SeW:85} introduced an analogous Grassmannian for polarized Hilbert spaces and this has since played a major role in the study of loop groups (e.g.\ \cite{PrS:86}). Sato Grassmannians for elementary Tate objects in exact categories $\Cc$ have been studied recently by Previdi \cite{Pre:12}. In order to ensure their good behavior, Previdi introduced two properties for exact categories: ``partially abelian'' and ``AIC plus AIC$^{\op}$''. The first notion turns out to be unnecessarily strong and, as was explained to us by T. B\"{u}hler, is equivalent to being abelian. The second condition, according to B\"{u}hler, is equivalent to Schneider's notion of ``quasi-abelian'' \cite{Sch:99} and Rump's ``almost abelian'' \cite{Rum:01}. Unfortunately, many exact categories of interest fail to satisfy ``AIC plus AIC$^{op}$''. A basic example is the category of vector bundles over the real line $\mathbb{R}$: the intersection of the admissible monics
\begin{align*}
    \begin{xy}
        \morphism(-1000,0)<500,0>[\mathbb{R}\times\mathbb{R}`\mathbb{R}\times\mathbb{R}^2;]
        \morphism(-1000,-250)/|->/<500,0>[(x,t)`(x,(t,0));]
        \place(0,0)[\text{and}]
        \morphism(500,0)<500,0>[\mathbb{R}\times\mathbb{R}`\mathbb{R}\times\mathbb{R}^2;]
        \morphism(500,-250)/|->/<500,0>[(x,t)`(x,(t,xt));]
    \end{xy}
\end{align*}
is not a vector bundle.

The Beilinson--Parshin ad\`{e}les were introduced by Parshin \cite{Par:76}, for 2-dimensional schemes, and by Beilinson \cite{Bei:80}, in the general case. From the time of their introduction, it appears to have been known, or at least strongly expected, that the $n$-dimensional ad\`{e}les should have the structure of an $n$-Tate object, but we are unable to find this in print. A closely related, more recent treatment of the $n$-dimensional ad\`{e}les appears in \cite{Osi:07}.

\subsubsection*{How to Read this Paper}
Ind, Pro and Tate objects arise widely in practice, and our hope is that this paper will be a useful reference.

The core results in this paper are summarized, for $n\text{-}\Tate(\Cc)$, in Theorem \ref{thm:ntatex}. A similar summary of results for $\Pro(\Cc)$ appears in Theorem \ref{thm:prox}, and the analogous results for $\Ind(\Cc)$ are developed in Section \ref{sec:ind}.

Section \ref{sec:ind} contains the bulk of the technical work in this paper, and develops the essential properties of the categories $\Ind(\Cc)$, $\FM(\Cc)$, and $P(\Cc)$. Much of this falls under the heading of ``things work as expected'', and the length of the section is a function of recording the proofs. With these proofs in hand, the analogous results for $\Pro(\Cc)$, $\FM^\vee(\Cc)$, and $P^\vee(\Cc)$ follow immediately, and Section \ref{sec:tate} extends these results to $\elTate(\Cc)$ and $\Tate(\Cc)$.

In Section \ref{sec:applat}, we recall Sato Grassmannians and establish Theorem \ref{thm:grdir}. The results of Sections \ref{sec:ind}-\ref{sec:applat} extend naturally to the setting of $n$-Tate objects, and in Section \ref{sec:ntate} we present these properties, and treat the Beilinson--Parshin ad\`{e}les.

Appendix \ref{sec:buhler} repeats a proof due to T. B\"{u}hler which shows that the present approach to left s-filtering sub-categories is equivalent to Schlichting's.

Appendix \ref{sec:stovicektrlifaj}, due to J. \v S\v tov\'\i\v cek and J. Trlifaj, identifies the category $\FM(R)$ of flat Mittag-Leffler modules with the category $\FM(P_f(R))$ and discusses a few relevant properties of these categories.

\subsection*{Acknowledgements}
We are very grateful to J. \v S\v tov\'\i\v cek and J. Trlifaj for helpful explanations, for simplifying the original example in Section \ref{sec:notic}, and for providing Appendix \ref{sec:stovicektrlifaj}. We are very grateful to T. B\"{u}hler for his generous and detailed comments on an earlier draft, and for helpful explanations of several aspects of exact categories. We thank V. Drinfeld, E. Getzler and A. Yekutieli for helpful conversations, B. Keller and M. Morrow for helpful correspondence, and X. Zhu for alerting us to a confusion in an earlier draft. We would also like to thank T. Hausel for supporting a visit of the first and the third author to EPF Lausanne, where part of this work was carried out.

\section{Preliminaries}\label{sec:prel}
\subsection{Cardinal Arithmetic}
We recall the following standard lemma.
\begin{lemma}
    Let $\kappa$ be an infinite cardinal. Let $K$ be a set of cardinality $\kappa$. The disjoint union $\coprod_{n\in\mathbb{N}}K^n$ of all finite tuples of elements of $K$ has cardinality $\kappa$. In particular, the disjoint union of all finite subsets of $K$ has cardinality at most $\kappa$.
\end{lemma}
\begin{proof}
    The proof is a standard induction, using the facts that $|K\times K|=\kappa$ and $|K\coprod K|=\kappa$.
\end{proof}

\subsection{Exact Categories}
Exact categories provide a general framework for linear algebra. We refer the reader to B\"{u}hler's survey \cite{Buh:10} for a full treatment.

\begin{definition}
    Let $\Cc$ be an additive category. A \emph{kernel-cokernel pair} is a sequence
    \begin{equation*}
        X\into Y\onto Z
    \end{equation*}
    such that $X\into Y$ is the kernel of $Y\onto Z$, and $Y\onto Z$ is the cokernel of $X\into Y$.

    An \emph{exact category} is an additive category $\Cc$ equipped with a class $\Ec\Cc$ of kernel-cokernel pairs. An \emph{exact sequence} is a kernel-cokernel pair in $\Ec\Cc$. An \emph{admissible monic} is a map $X\into Y$ which serves as the kernel in an exact sequence; an \emph{admissible epic} is $Y\onto Z$ is a map which serves as a cokernel in an exact sequence. We require that
    \begin{enumerate}
        \item for all $X\in \Cc$, the identity $1_X$ is both an admissible monic and an admissible epic,
        \item the class of admissible monics and the class of admissible epics are closed under composition,
        \item the pushout of an admissible monic along an arbitrary morphism exists and is an admissible monic, and the pullback of an admissible epic along an arbitrary morphism exists and is an admissible epic.
    \end{enumerate}
    A functor $F\colon\Cc\to\Dc$ between exact categories is \emph{exact} if $F(\Ec\Cc)\subset \Ec\Dc$. A fully faithful embedding $F\colon\Cc\to\Dc$ is \emph{fully exact} if every exact sequence in $\Dc$ of the form $F(X)\into F(Y)\onto F(Z)$ is the image of an exact sequence in $\Cc$.
\end{definition}

Every additive category $\Cc$ defines an exact category where the class $\Ec\Cc$ consists of all split exact sequences. The category $P_f(R)$ of finitely generated projective modules over a ring $R$ provides a motivating example. For another source of examples, every abelian category $\Cc$ defines an exact category in which $\Ec\Cc$ is the class of all kernel-cokernel pairs.

\subsubsection{Idempotent Completeness}
We recall two conditions on exact categories: idempotent completeness and weak idempotent completeness. In practice, the former is both more important and better behaved than the latter. The category $F_f(R)$ of finitely generated free modules over a ring $R$ provides an example of an exact category which is not idempotent complete. The category $P_f(R)$ of finitely generated projective $R$-modules provides an example of one which is.

\begin{definition}
    An exact category $\Cc$ is \emph{weakly idempotent complete} if every retract has a kernel. Explicitly, we require that any map $r\colon X\to Y$ for which there exists a right inverse $s\colon Y\to X$ admits a kernel in $\Cc$.
\end{definition}

\begin{remark}
    This condition is actually self-dual. For any additive category $\Cc$, all retracts have kernels in $\Cc$ if and only if all retracts have kernels in $\Cc^{\op}$. See \cite[Lemma 7.1]{Buh:10}.
\end{remark}

\begin{definition}
    An exact category $\Cc$ is \emph{idempotent complete} if, for every $p\colon X\to X$ such that $p^2=p$, there exists an isomorphism $X\cong Y\oplus Z$ which takes $p$ to the projection $0\oplus 1_Z$.
\end{definition}

\begin{example}\cite[Exercise 7.11]{Buh:10}
    Let $R=\mathbb{Q}\times\mathbb{Q}$. The category of finitely generated free $R$-modules is weakly idempotent complete, but not idempotent complete.\footnote{This example relies on the difference, for disconnected spaces, between a module being free and it being free on each component.}
\end{example}

\begin{definition}
    Let $\Cc$ be a category. Define the \emph{idempotent completion} $\Cc^{\ic}$ of $\Cc$ to be the category whose objects are pairs $(X,p)$, with $p\colon X\to X$ an idempotent in $\Cc$. Morphisms $(X,p)\to (Y,q)$ in $\Cc^{\ic}$ correspond to morphisms $g\colon X\to Y$ in $\Cc$ such that $qgp=g$; composition is induced by composition in $\Cc$.
\end{definition}

\begin{example}
    Let $F_f(R)$ and $P_f(R)$ denote the categories of finitely generated free and projective $R$-modules. The idempotent completion $F_f(R)^{\ic}$ is equivalent to $P_f(R)$.
\end{example}

The assignment $X\mapsto(X,1)$ defines a fully faithful embedding $\Cc\into\Cc^{\ic}$. We do not distinguish between $\Cc$ and its essential image under this embedding.

\begin{proposition}\cite[Proposition 6.10]{Buh:10}
    Let $\Cc$ be an exact category. Let $\Cc^{\ic}$ be the idempotent completion of $\Cc$. Define $\Ec(\Cc^{\ic})$ to consist of sequences which are direct summands of exact sequences in $\Cc$. Explicitly, a sequence
    \begin{equation*}
        (X_0,p_1)\to (X_1,p_1)\to(X_2,p_2)
    \end{equation*}
    is exact in $\Cc^{\ic}$ if there exists a sequence
    \begin{equation*}
        (X_0',p_0')\to (X_1',p_1')\to (X_2',p_2')
    \end{equation*}
    such that, for all $i$, $(X_i,p_i)\oplus(X_i',p_i')$ is isomorphic to an object $Y_i$ in $\Cc$, and such that the sequence
    \begin{equation*}
        Y_0\to Y_1\to Y_2
    \end{equation*}
    is exact in $\Cc$. The category $\Cc^{\ic}$ is an idempotent complete exact category. The embedding $\Cc\into\Cc^{\ic}$ is fully exact. This embedding is 2-universal in the category of exact functors $\Cc\to\Dc$ with $\Dc$ idempotent complete.
\end{proposition}

\subsubsection{Exact, Full Sub-Categories}
We recall here three conditions which may hold for exact, full sub-categories $\Cc\subset\Dc$. As explained to us by T. B\"{u}hler, if a sub-category satisfies all three of these conditions, then it is ``left s-filtering'' in the sense of Schlichting \cite[Definition 1.5]{Sch:04}. Left s-filtering sub-categories play many of the same roles for exact categories as Serre sub-categories play for abelian categories.

\begin{definition}
    A full sub-category $\Cc$ of an exact category $\Dc$ is \emph{closed under extensions} if, for every exact sequence
    \begin{equation*}
        X\into F\onto Z
    \end{equation*}
    with $X$ and $Z$ in $\Cc$, we have $F\in\Cc$ as well.
\end{definition}

The following lemma is a simple exercise in the definitions.
\begin{lemma}\label{lemma:extclosed}
    Let $\Dc$ be an exact category. Let $\Cc\subset\Dc$ be closed under extensions. Define a sequence in $\Cc$ to be exact if it is exact in $\Dc$. This endows $\Cc$ with the structure of an exact category.
\end{lemma}

\begin{definition}
    An exact, full sub-category $\Cc\subset\Dc$ is \emph{left special} if, for every admissible epic $F\onto X$ in $\Dc$ with $X\in\Cc$, there exists a commutative diagram in $\Dc$
    \begin{equation*}
        \begin{xy}
            \square/^{ (}->`>`>`^{ (}->/[Z`Y`G`F;```]
            \square(500,0)/->>`>`=`->>/[Y`X`F`X;```]
        \end{xy}
    \end{equation*}
    in which the top row is an exact sequence in $\Cc$ and the bottom row is an exact sequence in $\Dc$. We say $\Cc$ is \emph{right special} if $\Cc^{\op}$ is \emph{left special} in $\Dc^{\op}$.
\end{definition}

\begin{remark}
    We adapt the terminology ``left special'' from Schlichting \cite[Definition 1.5]{Sch:04}. Right special sub-categories were previously studied by Keller as sub-categories satisfying ``Condition C2'' \cite[Section 12.1]{Kel:96}.
\end{remark}

\begin{lemma}\label{lemma:lsextclosed}
    Left special sub-categories are closed under extensions.
\end{lemma}
\begin{proof}
    Let $\Cc\subset\Dc$ be a left special sub-category. Let
    \begin{equation*}
        X\into F\onto Z
    \end{equation*}
    be an exact sequence in $\Dc$ with $X$ and $Z$ in $\Cc$. By assumption, there exists a commuting diagram
    in $\Dc$
    \begin{equation*}
        \begin{xy}
            \square/^{ (}->`>`>`^{ (}->/[V`Y`X`F;```]
            \square(500,0)/->>`>`=`->>/[Y`Z`F`Z;```]
        \end{xy}
    \end{equation*}
    in which the top row is an exact sequence in $\Cc$ and the bottom row is an exact sequence in $\Dc$. Pushing out the admissible monic $V\into Y$ along the map $V\to X$, we obtain a second commuting diagram in $\Dc$
    \begin{equation*}
        \begin{xy}
            \square/^{ (}->`=`>`^{ (}->/[X`Y'`X`F;```]
            \square(500,0)/->>`>`=`->>/[Y'`Z`F`Z;```]
        \end{xy}
    \end{equation*}
    in which the top row is an exact sequence in $\Cc$ and the bottom row is an exact sequence in $\Dc$. The 5-Lemma \cite[Corollary 3.2]{Buh:10} implies that the map $Y'\to F$ is an isomorphism.
\end{proof}

\begin{definition}\label{def:lfilt}
    An exact, full sub-category $\Cc\subset\Dc$ is \emph{left filtering} if every morphism $X\to F$ in $\Dc$, with $X\in\Cc$, factors through an admissible monic $X'\into F$ with $X'\in\Cc$:
    \begin{equation*}
        \begin{xy}
            \qtriangle/>`>`<-_{) }/<400,400>[X`F`X';``]
        \end{xy}
    \end{equation*}
    We say $\Cc$ is \emph{right filtering} if $\Cc^{\op}$ is left filtering in $\Dc^{\op}$.
\end{definition}

\begin{definition}[Schlichting \cite{Sch:04}]\label{def:lsfilt}
    An exact, full sub-category $\Cc\subset\Dc$ is \emph{left special filtering}, or ``left s-filtering'' for short, if it is left special and left filtering. We say $\Cc$ is \emph{right s-filtering} if $\Cc^{\op}$ is left s-filtering in $\Dc^{\op}$.
\end{definition}

\begin{remark}
    We differ slightly from Schlichting in our presentation of left s-filtering. See Appendix \ref{sec:buhler} for a proof that the definitions agree.
\end{remark}

We record the following elementary observations.
\begin{lemma}\label{lemma:leftstransitive}
    Let $\Cc\subset\Dc\subset\Dc'$ be a chain of exact, fully faithful embeddings.
    \begin{enumerate}
        \item If $\Cc$ is closed under extensions in $\Dc'$, then $\Cc$ is closed under extensions in $\Dc$.
        \item If $\Cc$ is left special in $\Dc'$, then $\Cc$ is left special in $\Dc$.
        \item If $\Cc$ is left special in $\Dc'$ and left filtering in $\Dc$, then $\Cc$ is left s-filtering in $\Dc$.
    \end{enumerate}
\end{lemma}

The conditions of left filtering and left s-filtering play a role in forming quotient categories. We summarize the key facts, which we learned from Schlichting and B\"{u}hler, here.
\begin{proposition}
    Let $\Cc\subset\Dc$ be a full sub-category of an exact category. Denote by $\Sigma_e$ the collection of admissible epics in $\Dc$ with kernel in $\Cc$.
    \begin{enumerate}
        \item Then $\Sigma_e$ admits a calculus of left fractions in $\Dc$ if and only if $\Cc$ is left filtering and closed under extensions in $\Dc$.
        \item Denote by $\Dc[\Sigma_e^{-1}]$ the localization of $\Dc$ at $\Sigma_e$. If $\Cc$ is left s-filtering in $\Dc$, then every admissible monic in $\Dc$ with cokernel in $\Cc$ is invertible in $\Dc[\Sigma_e^{-1}]$. In this case, we alternately denote $\Dc[\Sigma_e^{-1}]$ by $\Dc/\Cc$.
        \item If $\Cc\subset\Dc$ is left s-filtering, then $\Dc/\Cc$ has a natural structure of an exact category in which a sequence is exact if and only if it is the image of an exact sequence under the map $\Dc\to\Dc/\Cc$.
    \end{enumerate}
\end{proposition}
The first two statements are due to B\"{u}hler. Once one knows the statements, the proofs are straightforward; we leave them to the interested reader. The third statement is due to Schlichting; we refer the reader to \cite[Proposition 1.16]{Sch:04} for the proof.

\subsection{Left Exact Presheaves}
\begin{definition}
    Let $\Cc$ be an exact category. Denote by $\lex(\Cc)$ the category of left-exact presheaves of abelian groups, i.e.\ functors $F\colon\Cc^{\op}\to\ab$ such that if
    \begin{equation*}
        X\into Y\onto Z
    \end{equation*}
    is a short exact sequence in $\Cc$, then the sequence of abelian groups
    \begin{equation*}
        0\to F(Z)\to F(Y)\to F(X)
    \end{equation*}
    is exact.
\end{definition}

The category $\lex(\Cc)$ is familiar in many contexts.
\begin{lemma}\label{lemma:lexR}
    Let $R$ be a ring. Denote by $\Mod(R)$ the category of (left) $R$-modules. The assignment
    \begin{equation*}
        M\mapsto\hom_R(-,M)
    \end{equation*}
    defines an equivalence of categories
    \begin{equation*}
        \begin{xy}
            \morphism<750,0>[\Mod(R)`\lex(P_f(R));\simeq]
        \end{xy}
    \end{equation*}
\end{lemma}
\begin{proof}
    The inverse equivalence can be described as follows. Let $F$ be a left-exact presheaf of abelian groups. Denote by $P_f(R)\downarrow F$ the category whose objects are morphisms of left exact presheaves $N\rightarrow F$ where $N$ is a finitely generated projective left $R$-module. Morphisms of $P_f(R)\downarrow F$ are commuting triangles over $F$. The inverse equivalence $\lex(P_f(R))\rightarrow \Mod(R)$ is given on objects by
    \begin{equation*}
        F\mapsto\colim_{P_f(R)\downarrow F} N
    \end{equation*}
    where the colimit is formed in the category of $R$-modules. Because finitely generated projective modules are in fact finitely presented, we can conclude that the above functor is a left and a right inverse.
\end{proof}

Similarly, we have the following.
\begin{lemma}\label{lemma:qcoh=lex(coh)}
    Let $X$ be a Noetherian scheme, and denote by $\Coh(X)$ and $\QCoh(X)$ the categories of coherent and quasi-coherent sheaves on $X$. Then
    \begin{equation*}
        \QCoh(X)\simeq \lex(\Coh(X)).
    \end{equation*}
\end{lemma}

These examples anticipate two other characterizations of $\lex(\Cc)$:
\begin{enumerate}
    \item Keller \cite[Appendix A]{Kel:90}, following Freyd and Quillen, exhibits $\lex(\Cc)$ as a localization, with respect to a Serre sub-category, of the category $\ab^{\Cc^{\op}}$ of presheaves of abelian groups on $\Cc$. Define a presheaf $F$ of abelian groups to be \emph{effaceable} if, for every $Y\in\Cc$ and every section
        \begin{equation*}
            \begin{xy}
                \morphism[Y`F;]
            \end{xy}
        \end{equation*}
        there exists an admissible epic in $\Cc$, as in the figure below, whose target is $Y$ and such that the restriction of the section along this epic is 0.
        \begin{equation*}
            \begin{xy}
                \btriangle/->>`>`>/<400,400>[X`Y`F;`0`]
            \end{xy}
        \end{equation*}
        Effaceable presheaves form a Serre sub-category of $\ab^{\Cc^{\op}}$, and $\lex(\Cc)$ is the associated Serre quotient. This exhibits $\lex(\Cc)$ as an abelian category.
    \item Thomason and Trobaugh \cite[Appendix A]{ThT:90}, following Laumon and Gabriel, observe that the admissible epics in $\Cc$ define a pre-topology.\footnote{See also \cite[Appendix A]{Buh:10}.}  They exhibit $\lex(\Cc)$ as the category of sheaves of abelian groups with respect to this topology. This shows that the inclusion of $\lex(\Cc)$ into $\ab^{\Cc^{\op}}$ preserves and reflects limits, and that the localization
        \begin{equation*}
            \begin{xy}
                \morphism[\ab^{\Cc^{\op}}`\lex(\Cc);]
            \end{xy}
        \end{equation*}
        preserves finite limits.
\end{enumerate}

\begin{proposition}\label{prop:lexleftspecial}
    Every exact category $\Cc$ is left special in $\lex(\Cc)$.
\end{proposition}
\begin{proof}
    Let $F\onto Z$ be an epic in $\lex(\Cc)$ with $Z\in\Cc$. The cokernel, in $\ab^{\Cc^{\op}}$, of this map is an effaceable presheaf,\footnote{If we wished to use Thomason's description of $\lex(\Cc)$, the existence of this triangle follows from the observation that the map $F\onto Z$ is an epimorphism of sheaves. Therefore there exists a cover $Z'\onto Z$, i.e. an admissible epic, which factors through the map from $F$.} so there exists an admissible epic $Y\onto Z$ in $\Cc$ fitting into a commuting triangle
    \begin{equation*}
        \begin{xy}
            \dtriangle/>`->>`->>/<400,400>[Y`F`Z;``]
        \end{xy}.
    \end{equation*}
    Taking the kernels of the maps to $Z$, we obtain the desired commuting diagram
    \begin{equation*}
        \begin{xy}
            \square/^{ (}->`>`>`^{ (}->/[X`Y`G`F;```]
            \square(500,0)/->>`>`=`->>/[Y`Z`F`Z;```]
        \end{xy}
    \end{equation*}
\end{proof}

\begin{remark}
    In general, $\Cc$ is very far from being left filtering in $\lex(\Cc)$. As a basic example, consider the category $F_f(\mathbb{Z})$ of finitely generated free abelian groups. The category $\lex(F_f(\mathbb{Z}))$ is equivalent to the category of abelian groups. The map $\mathbb{Z}\to\mathbb{Z}/2$ does not factor through a monic from a free abelian group.
\end{remark}

The proof of Lemma \ref{lemma:lexR} contains a basic construction we will use again.
\begin{definition}\label{def:catel}
    Let $\Cc$ be an exact category, and let $F$ be a left exact presheaf. Define the \emph{category of elements} $\Cc\downarrow F$ of $F$ as follows.
    \begin{enumerate}
        \item Objects are maps $X\to F$ in $\lex(\Cc)$ from an object $X\in\Cc$ to $F$.
        \item A morphism from $X\to F$ to $Y\to F$ is a commuting triangle
            \begin{equation*}
                \begin{xy}
                    \Vtriangle<350,400>[X`Y`F;``]
                \end{xy}
            \end{equation*}
    \end{enumerate}
\end{definition}

The following is a standard fact about categories of elements (e.g. see \cite[Theorem 2.15.6]{Bor:94-1}).
\begin{lemma}
    Let $F$ be a left exact presheaf on $\Cc$. The canonical map
    \begin{equation*}
        \colim_{\Cc\downarrow F} X\to F
    \end{equation*}
    is an isomorphism, where the colimit is taken in $\lex(\Cc)$.
\end{lemma}

\section{Admissible Ind-Objects}\label{sec:ind}
In this section, we develop the properties of admissible Ind-objects in an exact category $\Cc$. Admissible Ind-objects sit in relation to objects of $\Cc$ as the $R$-module $R[t]$ sits in relation to finitely generated free $R$-modules.\footnote{For definiteness, the $R$-module $R[t]$ is the direct sum $\bigoplus_{n\in\mathbb{N}} R$.} More generally, the results of this section should be viewed as an elaboration, in the setting of exact categories, of the dual to Artin--Mazur \cite[Appendix 2]{ArM:69}.

\subsection{The Category of Admissible Ind-Objects}
\begin{definition}
    Let $I$ and $J$ be directed posets. A functor $\varphi:I\rightarrow J$ is \emph{final} if for every $j\in J$, there exists $i\in I$ with $j\leq\varphi(i)$.
\end{definition}

\begin{definition}
    Let $\Cc$ be an exact category. Let $\kappa$ be a infinite cardinal. An \emph{admissible Ind-diagram of size at most $\kappa$} is a functor
    \begin{equation*}
        \begin{xy}
            \morphism[I`\Cc;X]
        \end{xy}
    \end{equation*}
    such that $I$ is a directed poset of cardinality at most $\kappa$, and such that $X$ takes any arrow in $I$ to an admissible monic in $\Cc$. Morphisms of admissible Ind-diagrams are 2-commuting triangles
    \begin{equation*}
        \begin{xy}
            \Vtriangle<350,400>[I`J`\Cc;\varphi`X`Y]
            \place(350,200)[\twoar(1,0)]
            \place(350,300)[_{\alpha}]
        \end{xy}.
    \end{equation*}
    Denote the category of admissible Ind-diagrams of cardinality at most $\kappa$ by $\Dirk(\Cc)$.
\end{definition}

An admissible Ind-diagram defines a left-exact presheaf by the assignment
\begin{equation*}
    \begin{xy}
        \morphism/|->/<750,0>[Y`\colim_I \hom_{\Cc}(Y,X_i);]
    \end{xy}.
\end{equation*}
This extends to a functor
\begin{equation}\label{eq:dirfun}
    \begin{xy}
        \morphism[\Dirk(\Cc)`\lex(\Cc);\widehat{(-)}]
    \end{xy}.
\end{equation}

\begin{definition}\label{defi:ind}
    Define the category $\Indk(\Cc)$ of \emph{admissible Ind-objects in $\Cc$ of size at most $\kappa$} to be the full sub-category of $\lex(\Cc)$ consisting of objects in the essential image of \ref{eq:dirfun}.
\end{definition}

We can also omit the cardinality bound $\kappa$.
\begin{definition}
    Denote by $\Dir(\Cc)$ the (large) category of admissible Ind-diagrams of arbitrary cardinality. Denote by $\Ind(\Cc)$ the full sub-category of $\lex(\Cc)$ consisting of objects in the essential image of $\Dir(\Cc)$.
\end{definition}

\begin{remark}
    Countable admissible Ind-objects have been studied for some time, for instance in \cite[Appendix A]{Kel:90}, or \cite{Sch:04}. In Section \ref{sec:countableind}, we show that for $\kappa=\aleph_0$, the category $\Ind_{\aleph_0}(\Cc)$ recovers Keller's treatment (Proposition \ref{prop:countableind}). By allowing for uncountable admissible Ind-objects, we can treat examples such as the categories of projective or flat Mittag-Leffler modules over a ring in terms of admissible Ind-objects (see Section \ref{sec:indRmod}).
\end{remark}

If $X\colon I\to\Cc$ and $Y\colon J\to\Cc$ are admissible Ind-diagrams in $\Cc$, then the definition ensures that
\begin{equation*}
    \hom_{\Indk(\Cc)}(\widehat{X},\widehat{Y})\cong\lim_I\colim_J\hom_{\Cc}(X_i,Y_j)
\end{equation*}
In particular, we see that, for any map $f\colon X\to \widehat{Y}$ in $\Indk(\Cc)$ with $X\in\Cc$, and for any admissible diagram $Y\colon J\to\Cc$ representing an $\widehat{Y}$, there exists $j\in J$ such that $f$ factors through some $Y_j\into\widehat{Y}$. This is a key property in what follows, and part of the general phenomenon that objects of $\Cc$ are \emph{finitely presentable} in $\lex(\Cc)$.

\begin{remark}
    We could define an admissible Ind-diagram to be a map $X\colon I\rightarrow\Cc$ such that $I$ is a filtering category, rather than a directed poset. However, this would give an equivalent notion of admissible Ind-object. Indeed, the image of such a diagram $X$ is a directed poset because the maps $X_i\rightarrow X_j$ are all monic.
\end{remark}

\begin{theorem}\label{thm:indleftspecial}
    For any infinite cardinal $\kappa$, the category $\Indk(\Cc)$ is closed under extensions in $\lex(\Cc)$. In particular, $\Indk(\Cc)$ admits a canonical structure of an exact category. Further, if $\Cc$ is weakly idempotent complete, then $\Indk(\Cc)$ is left special in $\lex(\Cc)$. The same is true for $\Ind(\Cc)$.
\end{theorem}

The following lemma contains the core of the proof of the theorem.
\begin{lemma}\label{lemma:indleftspecial}
    Let $F\onto\widehat{Z}$ be an epic in $\lex(\Cc)$ with $\widehat{Z}\in\Indk(\Cc)$. For any $Z\colon I\to\Cc$ in $\Dirk(\Cc)$ representing $\widehat{Z}$, there exists a morphism in $\Dirk(\Cc)$
    \begin{equation*}
        \begin{xy}
            \Vtriangle<350,400>[K`J`\Cc;\varphi`Y`Z]
            \place(350,200)[\twoar(1,0)]
            \place(350,300)[_{\alpha}]
        \end{xy}
    \end{equation*}
    such that:
    \begin{enumerate}
        \item the map $\varphi$ is final,
        \item for all $k\in K$, the map $\alpha_k\colon Y_k\to Z_{\varphi(k)}$ is an admissible epic, and
        \item the induced map $\widehat{Y}\onto\widehat{Z}$ factors through the map $F\onto\widehat{Z}$ as in
        \begin{equation*}
            \begin{xy}
                \dtriangle/>`->>`->>/<400,400>[\widehat{Y}`F`\widehat{Z};``]
            \end{xy}
        \end{equation*}
    \end{enumerate}
\end{lemma}

We also need the following minor restatement of the duals of \cite[Proposition A.3.1]{ArM:69} and \cite[Corollary A.3.2]{ArM:69}.\footnote{The only difference, besides notation, is that our assumption that all maps in the diagrams are monic allows us to replace ``$U$-small'' by a definite cardinality bound on the size of $I\downarrow_{\Cc}J$.}
\begin{lemma}[Straightening Morphisms]\label{lemma:straight}
    Let $X\colon I\to \Cc$ and $Y\colon J\to\Cc$ be diagrams in $\Cc$ where $I$ and $J$ are directed posets of cardinality at most $\kappa$, and where, for all $i\le i'$ in $I$ and $j\le j'$ in $J$, the maps $X_i\to X_{i'}$ and $Y_j\to Y_{j'}$ are (not necessarily admissible) monics. Denote by $\widehat{X}$ and $\widehat{Y}$ the colimits of these diagrams in $\lex(\Cc)$. Given a map $f\colon\widehat{X}\rightarrow\widehat{Y}$, define $I\downarrow_{\Cc} J$ to be the category whose objects are triples $(i,j,\alpha_{ij})$ where $i\in I$, $j\in J$ and $\alpha_{ij}$ is a morphism which fits into a commuting square (in $\lex(\Cc)$)
    \begin{equation}\label{eq:strtsq}
        \begin{xy}
            \square[X_i`Y_j`\widehat{X}`\widehat{Y};\alpha_{ij}```f]
        \end{xy}
    \end{equation}
    Then $I\downarrow_{\Cc} J$ is a directed poset of cardinality at most $\kappa$.

    Further, define $\varphi\colon I\downarrow_{\Cc} J\to I$ and $\psi\colon I\downarrow_{\Cc}J\to J$ to be the projections in the $I$ and $J$ factors. Both $\varphi$ and $\psi$ are final, and they give rise to a span in $\Dirk(\Cc)$
    \begin{equation*}
        \begin{xy}
            \Vtrianglepair/<-`>`>`>`>/[I`I\downarrow_{\Cc} J`J`\Cc;\varphi`\psi`X``Y]
            \place(700,350)[\twoar(1,0)^{\alpha}]
        \end{xy}
    \end{equation*}
    in which the triangle on the left strictly commutes, and the component of $\alpha$ at $(i,j,\alpha_{ij})$ is given by $\alpha_{ij}$.

    In particular, any map of admissible Ind-objects can be straightened to a map of admissible Ind-diagrams.
\end{lemma}

\begin{proof}[Proof of Theorem \ref{thm:indleftspecial}]
    We use the lemmas to show that $\Indk(\Cc)\subset\lex(\Cc)$ is closed under extensions. If $\Cc$ is weakly idempotent complete, our argument will also show that $\Indk(\Cc)$ is left special in $\lex(\Cc)$.

    Let
    \begin{equation*}
        \widehat{X}\into F\onto \widehat{Z}
    \end{equation*}
    be a short exact sequence in $\lex(\Cc)$ where $\widehat{X}$ and $\widehat{Z}$ are in $\Indk(\Cc)$. Let $X\colon I\to\Cc$ and $Z\colon J\to\Cc$ be admissible Ind-diagrams of cardinality at most $\kappa$ representing $\widehat{X}$ and $\widehat{Z}$. Lemma \ref{lemma:indleftspecial} guarantees the existence of a map of admissible Ind-diagrams of cardinality at most $\kappa$
    \begin{equation*}
        \begin{xy}
            \Vtriangle<350,400>[K`J`\Cc;\varphi`Y'`Z]
            \place(350,200)[\twoar(1,0)]
            \place(350,300)[_{\alpha}]
        \end{xy}
    \end{equation*}
    such that all of the components of $\alpha$ are admissible epics and such that the induced map $\widehat{Y'}\onto\widehat{Z}$ factors through the map $F\onto\widehat{Z}$. Define a directed diagram
    \begin{equation*}
        \begin{xy}
            \morphism[K`\Cc;X']
            \morphism(0,-200)/|->/[k`\ker(\alpha_{k});]
        \end{xy}
    \end{equation*}
    The maps in this diagram are the monics induced by the universal property of kernels.\footnote{Note that, while the maps in this diagram are monics, we do not claim that they are admissible monics; in general, this is only the case under the additional assumption that $\Cc$ is weakly idempotent complete.} This diagram fits into an exact sequence with the admissible Ind-diagrams $Y'$ and $Z$
    \begin{equation*}
        \begin{xy}
            \Vtrianglepair/=`>`>`>`>/[K`K`J`\Cc;``X'`Y'`Z]
            \place(350,250)[\twoar(1,0)]
            \place(350,350)[_{\beta}]
            \place(650,250)[\twoar(1,0)]
            \place(650,350)[_{\alpha}]
        \end{xy}
    \end{equation*}
    where the map $K\rightarrow J$ is final, the components of $\alpha$ are admissible epics, and the components of $\beta$ are admissible monics. Taking the colimit of these diagrams in $\lex(\Cc)$, we obtain a commuting diagram with exact rows
    \begin{equation*}
        \begin{xy}
            \square/^{ (}->`>`>`^{ (}->/[\widehat{X'}`\widehat{Y'}`\widehat{X}`F;```]
            \square(500,0)/->>`>`=`->>/[\widehat{Y'}`\widehat{Z}`F`\widehat{Z};```]
        \end{xy}
    \end{equation*}
    Note that the map from $\widehat{X'}$ to $\widehat{X}$ is induced by the universal property of kernels. Using Lemma \ref{lemma:straight}, we straighten this map to a span
    \begin{equation*}
        \begin{xy}
            \Vtrianglepair/<-`>`>`>`>/[K`K\downarrow_{\Cc}I`I`\Cc;``X'``X]
        	\place(600,250)[\twoar(1,0)]
            \place(600,350)[_{\gamma}]
        \end{xy}
    \end{equation*}
    and define
    \begin{equation*}
        \begin{xy}
            \morphism<750,0>[K\downarrow_{\Cc}I`\Cc;Y]
            \morphism(0,-200)/|->/<750,0>[(k,i,\gamma_{ki})`X_i\cup_{X'_{k}}Y'_{k};]
        \end{xy}
    \end{equation*}
    A map $(k,i,\gamma_{ki})\to(\ell,j,\eta_{\ell j})$ induces a map $Y_{(k,i,\gamma_{ki})}\to Y_{(\ell,j,\eta_{\ell j})}$ by the universal property of pushouts.

    The diagram $Y$ is admissible. Indeed, for each arrow $(k,i,\gamma_{ki})\rightarrow(\ell,j,\eta_{\ell j})$ in $K\downarrow_{\Cc} I$, we have a diagram with exact rows
    \begin{equation*}
        \begin{xy}
            \square/^{ (}->`^{ (}->`^{ (}->`^{ (}->/<750,500>[X_i`Y_{(k,i,\gamma_{ki})}`X_j`Y_{(\ell,j,\eta_{\ell j})};```]
            \square(750,0)/->>`>`^{ (}->`->>/<750,500>[Y_{(k,i,\gamma_{ki})}`Z_{\varphi(k)}`Y_{(\ell,j,\eta_{\ell j})}`Z_{\varphi(\ell)};```]
        \end{xy}
    \end{equation*}
    The left and right vertical arrows are admissible monics. Therefore the middle vertical arrow is an admissible monic as well \cite[Corollary 3.2]{Buh:10}.

    Passing back to the associated admissible Ind-objects, we obtain a commuting diagram in $\lex(\Cc)$ with exact rows
    \begin{equation*}
        \begin{xy}
            \square/^{ (}->`=`>`^{ (}->/[\widehat{X}`\widehat{Y}`\widehat{X}`F;```]
            \square(500,0)/->>`>`=`->>/[\widehat{Y}`\widehat{Z}`F`\widehat{Z};```]
        \end{xy}
    \end{equation*}
    The 5-Lemma shows that the map from $\widehat{Y}$ to $F$ is an isomorphism. We conclude that $\Indk(\Cc)$ is closed under extensions in $\lex(\Cc)$, and is therefore a fully exact sub-category.

    If $\Cc$ is weakly idempotent complete, let
    \begin{equation*}
        G\into F\onto\widehat{Z}
    \end{equation*}
    be a short exact sequence in $\lex(\Cc)$ with $\widehat{Z}\in\Indk(\Cc)$. The construction above above yields a commuting diagram with exact rows
    \begin{equation*}
        \begin{xy}
            \square/^{ (}->`>`>`^{ (}->/[\widehat{X'}`\widehat{Y'}`G`F;```]
            \square(500,0)/->>`>`=`->>/[\widehat{Y'}`\widehat{Z}`F`\widehat{Z};```]
        \end{xy}
    \end{equation*}
    Because $\Cc$ is weakly idempotent complete, the maps in the diagram $X'\colon K\to\Cc$ are all admissible monics. We conclude that $\widehat{X'}\in\Indk(\Cc)$, and that $\Indk(\Cc)$ is left special in $\lex(\Cc)$.
\end{proof}

\begin{proof}[Proof of Lemma \ref{lemma:indleftspecial}]
    Our argument follows Keller \cite[Appendix B]{Kel:90}. The only change is that, in considering $\kappa>\aleph_0$, we need to work with directed posets rather than only linear orders. Let
    \begin{equation*}
        F\onto\widehat{Z}
    \end{equation*}
    be an epic in $\lex(\Cc)$ with $\widehat{Z}\in\Indk(\Cc)$. Let $Z\colon J\to\Cc$ be an admissible Ind-diagram representing $\widehat{Z}$. For all $j\in J$, we can pull back $F$ along the map $Z_j\into\widehat{Z}$. Denote the resulting epic by $F_j\onto Z_j$. The category $\Cc$ is left special in $\lex(\Cc)$, so there exists an admissible epic $Y_j\onto Z_j$ in $\Cc$ such that we have a triangle
    \begin{equation*}
        \begin{xy}
            \dtriangle/>`->>`->>/<400,400>[Y_j`F_j`Z_j;``]
        \end{xy}.
    \end{equation*}
    We use these triangles to construct an admissible Ind-diagram
    \begin{equation*}
        \begin{xy}
            \morphism[K`\Cc;Y]
        \end{xy}
    \end{equation*}
    with the desired properties.

    Let $K$ be the category whose objects are finite directed sub-posets $J_a\subset J$. Morphisms in $K$ are inclusions of sub-diagrams. We first observe that $|K|\leq\coprod_{n\in\mathbb{N}}|J|^n\leq\kappa$.

    We now show that $K$ is directed. Because $J$ is directed, for any finite collection $\{j_i\}_{i=0}^n$ of objects of $J$, there exists $j\in J$ such that $j_i\leq j$ for $i=0,\ldots,n$. Given any two objects $J_{a_0},J_{a_1}\in K$, their union in $J$ consists of a finite collection of objects of $J$. Taking an upper bound on this diagram, we obtain a finite directed sub-poset of $J$ which contains both $J_{a_0}$ and $J_{a_1}$.

    The assignment
    \begin{equation*}
        \begin{xy}
            \morphism/|->/[J_a`\max J_a;]
        \end{xy}
    \end{equation*}
    defines a final map from $K$ to $J$.

    Define an admissible Ind-diagram by
    \begin{equation*}
        \begin{xy}
            \morphism[K`\Cc;Y]
            \morphism(0,-200)/|->/[J_a`\bigoplus_{j\in J_a} Y_j;]
        \end{xy}
    \end{equation*}
    Morphisms are given by inclusions of summands.

    Define a natural transformation
    \begin{equation*}
        \begin{xy}
            \Vtriangle<350,400>[K`J`\Cc;`Y`Z]
        	\place(350,200)[\twoar(1,0)]
        	\place(350,300)[_{\alpha}]
            \end{xy}
    \end{equation*}
    by defining the restriction of the component $\alpha_{J_a}$ to the $Y_j$-summand of $Y_{J_a}$ to be the composite
    \begin{equation*}
        Y_j\to Z_j\to\cdots\to Z_{\max J_a}.
    \end{equation*}
    The $\alpha_{J_a}$ are admissible epics for all $J_a\in K$. Indeed, consider the shear map
    \begin{equation*}
        \begin{xy}
            \morphism<750,0>[\oplus_{j\in J_a} Y_j`\oplus_{j\in J_a} Y_j;\sigma]
        \end{xy}.
    \end{equation*}
    i.e. the map whose restriction to the $Y_j$ summand is given by
    \begin{equation*}
        \begin{xy}
            \morphism<650,0>[\sum_{j'\ge j}(Y_j`Y_{j'});]
        \end{xy}
    \end{equation*}
    The shear map is an isomorphism, because $|J_a|$ is finite. It factors $\alpha_{J_a}$ as
    \begin{equation*}
        \begin{xy}
            \morphism/->>/<750,0>[\oplus_{j\in J_a} Y_j`\oplus_{j\in J_a} Z_j;]
            \morphism(750,0)<750,0>[\oplus_{j\in J_a} Z_j`\oplus_{j\in J_a} Z_j;\sigma]
            \morphism(1500,0)<750,0>[\oplus_{j\in J_a} Z_j`Z_{\max J_a};]
        \end{xy}
    \end{equation*}
    where the last map is the projection onto the factor $Z_{\max J_a}$. Each map in this factorization is an admissible epic in $\Cc$, so $\alpha_{J_a}$ is an admissible epic as well.

    We now define a co-cone on the diagram $Y$ with co-cone point $F$. For $J_a\in K$, the restriction of the map $Y_{J_a}\to F$ to the $Y_j$ summand is given by the composite
    \begin{equation*}
        Y_j\to F_j\to\cdots\to F_{\max J_a}\to F
    \end{equation*}
    By the universal property of colimits, this determines a unique map $\widehat{Y}\to F$ fitting into the desired commuting triangle.
\end{proof}

\subsection{Properties of Admissible Ind-Objects}\label{sec:indprop}
\subsubsection{\texorpdfstring{$\Cc$}{C} as an Exact, Full Sub-Category of \texorpdfstring{$\Indk(\Cc)$}{Ind(C)}}
\begin{proposition}\label{prop:cleftsfiltinind}
    The category $\Cc$ is left s-filtering in $\Indk(\Cc)$.
\end{proposition}
In Proposition \ref{prop:lexleftspecial} we showed that $\Cc$ is left special in $\lex(\Cc)$. To prove the proposition, it suffices to show that $\Cc$ is left filtering in $\Indk(\Cc)$. This is a consequence of the following.
\begin{lemma}\label{lemma:indsubob}
    Let $X\colon I\to\Cc$ be an admissible Ind-diagram. Then for any $i\in I$, the map $X_i\to\widehat{X}$ is an admissible monic in $\Indk(\Cc)$.
\end{lemma}
\begin{proof}
    Denote by $I_i\subset I$ the sub-poset consisting of all $j\in I$ such that $i\leq j$. The inclusion $I_i\into I$ is final. The diagram
    \begin{equation*}
        \begin{xy}
            \morphism[I_i`\Cc;X/X_i]
            \morphism(0,-200)/|->/[j`X_j/X_i;]
        \end{xy}
    \end{equation*}
    is admissible, because admissible monics push out. It fits into a sequence of admissible Ind-diagrams
    \begin{equation*}
        \begin{xy}
            \Vtrianglepair/=`=`>`>`>/[I_i`I_i`I_i`\Cc;``X_i`X`X/X_i]
            \place(375,250)[\twoar(1,0)]
            \place(375,350)[_{\alpha}]
            \place(650,250)[\twoar(1,0)]
            \place(650,350)[_{\beta}]
        \end{xy}
    \end{equation*}
    This sequence defines an exact sequence of admissible Ind-objects, whose first map is the canonical map $X_i\to\widehat{Y}$.
\end{proof}

\begin{proof}[Proof of Proposition \ref{prop:cleftsfiltinind}]
    Let $f\colon X\to\widehat{Y}$ be a morphism in $\Indk(\Cc)$ with $X\in\Cc$. For any diagram $Y\colon J\to\Cc$ presenting $\widehat{Y}$, there exists $j\in J$ such that $f$ factors through the map $Y_j\into\widehat{Y}$. This map is an admissible monic by the previous lemma. We conclude that $\Cc$ is left filtering, and thus left s-filtering, in $\Indk(\Cc)$.
\end{proof}

\subsubsection{Exact Sequences of Admissible Ind-Objects}
\begin{proposition}[Straightening Exact Sequences]\label{prop:inde=eind}
    The exact categories $\Indk(\Ec\Cc)$ and $\Ec\Indk(\Cc)$ are canonically equivalent.
\end{proposition}
\begin{proof}
    The category $\Ec\Cc$ has a canonical exact structure \cite{Hel:58} in which admissible monics (epics) are maps of exact sequences which are admissible monics (epics) at each term in the sequence.

    An admissible Ind-diagram in $\Ec\Cc$ consists of a pair of maps of admissible Ind-diagrams in $\Cc$
    \begin{equation*}
        \begin{xy}
            \Vtrianglepair/=`=`>`>`>/[I`I`I`\Cc;``X`Y`Z]
            \place(375,250)[\twoar(1,0)]
            \place(375,350)[_{\alpha}]
            \place(650,250)[\twoar(1,0)]
            \place(650,350)[_{\beta}]
        \end{xy}
    \end{equation*}
    where the components of $\alpha$ are admissible monics, and those of $\beta$ are admissible epics.

    Because directed colimits in $\lex(\Cc)$ preserve kernels and cokernels, an admissible Ind-diagram of exact sequences in $\Cc$ canonically defines an exact sequence of admissible Ind-objects. This extends to a faithful, exact functor
    \begin{equation}\label{fun:indetoeind}
        \begin{xy}
            \morphism<750,0>[\Indk(\Ec\Cc)`\Ec\Indk(\Cc);]
        \end{xy}
    \end{equation}
    The construction in the proof of Theorem \ref{thm:indleftspecial} implies that this functor is essentially surjective: every short exact sequence of admissible Ind-objects arises as the colimit of a directed diagram of short exact sequences in $\Cc$.

    We now show that the functor \ref{fun:indetoeind} is full. Let
    \begin{equation}\label{eq:indsqfull}
        \begin{xy}
            \square|almb|/^{ (}->`>`>`^{ (}->/[\widehat{X_0}`\widehat{Y_0}`\widehat{X_1}`\widehat{Y_1};`g_X`g_Y`]
            \square(500,0)|amrb|/->>`>`>`->>/[\widehat{Y_0}`\widehat{Z_0}`\widehat{Y_1}`\widehat{Z_1};`g_Y`g_Z`]
        \end{xy}
    \end{equation}
    be a morphism in $\Ec\Indk(\Cc)$. Let
    \begin{equation*}
        \begin{xy}
            \Vtrianglepair/=`=`>`>`>/[I`I`I`\Cc;``X_1`Y_1`Z_1]
            \place(375,250)[\twoar(1,0)]
            \place(375,350)[_{\alpha_1}]
            \place(650,250)[\twoar(1,0)]
            \place(650,350)[_{\beta_1}]
        \end{xy}
    \end{equation*}
    be an admissible Ind-diagram of exact sequences representing the bottom row of \ref{eq:indsqfull}. Straightening (Lemma \ref{lemma:straight}) allows us to represent the map $g_Z$ as a 2-commuting triangle
    \begin{equation*}
        \begin{xy}
            \Vtriangle<350,400>[J`I`\Cc;`Z_0`Z_1]
            \place(350,200)[\twoar(1,0)]
            \place(350,300)[_{\gamma_Z}]
        \end{xy}
    \end{equation*}
    We can represent the admissible epic $\widehat{Y_0}\onto\widehat{Z_0}$ as a 2-commuting triangle
    \begin{equation*}
        \begin{xy}
            \Vtriangle<350,400>[K`J`\Cc;`Y_0'`Z_0]
            \place(350,200)[\twoar(1,0)]
            \place(350,300)[_{\beta_0'}]
        \end{xy}
    \end{equation*}
    where the map $K\to J$ is final.

    We straighten $g_Y$ to a 2-commuting triangle
    \begin{equation*}
        \begin{xy}
            \Vtriangle<350,400>[L`I`\Cc;\varphi`Y_0'`Y_1]
            \place(350,200)[\twoar(1,0)]
            \place(350,300)[_{\gamma_{Y'}}]
         \end{xy}
    \end{equation*}
    The straightening construction also produces a final map $\psi'\colon L\rightarrow K$. Denote by $\psi$ the composite $L\to^{\psi'} K\to J$. The 2-commuting triangle above embeds in a 2-commuting pyramid
    \begin{equation*}
        \begin{xy}
            \square/=`>`>`=/<1000,1000>[L`L`I`I;`\varphi`\varphi`]
            \morphism(0,1000)|m|<500,-500>[L`\Cc;Y_0']
            \morphism(1000,1000)|m|<-500,-500>[L`\Cc;Z_0\psi]
            \morphism(0,0)|m|<500,500>[I`\Cc;Y_1]
            \morphism(1000,0)|m|<-500,500>[I`\Cc;Z_1]
            \place(500,250)[\twoar(1,0)]
            \place(500,150)[^{\beta_1}]
            \place(500,750)[\twoar(1,0)]
            \place(500,850)[_{\beta_0'\psi'}]
            \place(150,500)[\twoar(0,-1)]
            \place(250,500)[_{\gamma_{Y'}}]
            \place(850,500)[\twoar(0,-1)]
            \place(750,500)[_{\gamma_Z\psi}]
        \end{xy}.
    \end{equation*}
    Indeed, for every $l\in L$, we have a pair of commuting squares of admissible Ind-objects
    \begin{equation*}
        \begin{xy}
            \square|blrb|/>`^{ (}->`^{ (}->`>/<1000,500>[Y_{0,l}'`Z_{1,\varphi(l)}`\widehat{Y}_0`\widehat{Z}_1;\beta_{1,\varphi(l)}\circ \gamma_{Y',l}```]
            \morphism(0,500)|a|/{@<4pt>}/<1000,0>[Y_{0,l}'`Z_{1,\varphi(l)};\gamma_{Z,\psi(l)}\circ\beta_{0,\psi'(l)}']
        \end{xy}
    \end{equation*}
    Because the map $Z_{1,\varphi(l)}\to\widehat{Z}_1$ is monic, we see that $\gamma_{Z,\psi(l)}\circ\beta_{0,\psi'(l)}'=\beta_{1,\varphi(l)}\circ\gamma_{Y',l}$,

    The components of the natural transformation $\beta_0\psi'$ are admissible epics. Define a directed diagram
    \begin{equation*}
        \begin{xy}
            \morphism[L`\Cc;X_0']
            \morphism(0,-200)/|->/[l`\ker(\beta_{0,\psi'(l)});]
        \end{xy}
    \end{equation*}
    Denote by $\widehat{X'}_0$ the colimit of this diagram in $\lex(\Cc)$. The universal property of kernels induces a canonical isomorphism
    \begin{equation}\label{nuiso}
        \widehat{X'}_0\to^\cong \widehat{X}_0
    \end{equation}
    As in the proof of Theorem \ref{thm:indleftspecial}, $X'$ may not be an admissible Ind-diagram. Nevertheless, let
    \begin{equation*}
        \begin{xy}
            \Vtriangle<350,400>[M`I`\Cc;\xi`X_0`X_1]
            \place(350,200)[\twoar(1,0)]
            \place(350,300)[_{\gamma_X}]
        \end{xy}
    \end{equation*}
    be a 2-commuting triangle representing the map $g_X$. The isomorphism \eqref{nuiso} lifts to a span of directed diagrams
    \begin{equation*}
        \begin{xy}
            \Vtrianglepair/<-`>`>`>`>/[L`N`M`\Cc;\lambda`\mu`X'_0``X_0]
            \place(650,250)[\twoar(1,0)]
        \end{xy}
    \end{equation*}
    where the maps $\lambda$ and $\mu$ are final. We define a directed diagram
    \begin{equation*}
        \begin{xy}
            \morphism<750,0>[N`\Cc;Y_0]
            \morphism(0,-200)/|->/<750,0>[n`X_{0,\mu(n)}\cup_{X'_{0,\lambda(n)}}Y_{0,\lambda(n)}';]
        \end{xy}
    \end{equation*}
    Denote by $\beta_0$ the natural transformation $\beta'_0\psi\lambda$. The diagram $Y_0$ fits into a 2-commuting triangle
    \begin{equation}\label{eq:exindtopsq}
        \begin{xy}
            \Vtrianglepair/=`=`>`>`>/[N`N`N`\Cc;``X_0\mu`Y_0`Z_0\psi\lambda]
            \place(375,250)[\twoar(1,0)]
            \place(375,350)[_{\alpha_0}]
            \place(650,250)[\twoar(1,0)]
            \place(650,350)[_{\beta_0}]
        \end{xy}
    \end{equation}
    in which the components of $\alpha_0$ are admissible monics and the components of $\beta_0$ are admissible epics. The diagram $Y_0$ is an admissible Ind-diagram because $X_0$ and $Z_0$ are, just as at the end of the proof of Theorem \ref{thm:indleftspecial}. It is of cardinality at most $\kappa$ by construction.

    The 2-commuting triangle \ref{eq:exindtopsq} fits into 2-commuting diagram
    \begin{equation}\label{indemap}
        \begin{xy}
            \square|almb|/=`>`>`=/<1000,1000>[N`N`I`I;`\varphi\lambda`\varphi\lambda`]
            \morphism(0,1000)|m|<500,-500>[N`\Cc;X_0\mu]
            \morphism(1000,1000)|m|<-500,-500>[N`\Cc;Y_0]
            \morphism(0,0)|m|<500,500>[I`\Cc;X_1]
            \morphism(1000,0)|m|<-500,500>[I`\Cc;Y_1]
            \place(500,250)[\twoar(1,0)]
            \place(500,150)[^{\alpha_1}]
            \place(500,750)[\twoar(1,0)]
            \place(500,850)[_{\alpha_0}]
            \place(150,500)[\twoar(0,-1)]
            \place(250,500)[_{\gamma_X\mu}]
            \place(850,500)[\twoar(0,-1)]
            \place(750,500)[_{\gamma_Y}]
            \square(1000,0)|amrb|/=`>`>`=/<1000,1000>[N`N`I`I;`\varphi\lambda`\varphi\lambda`]
            \morphism(1000,1000)|m|<500,-500>[N`\Cc;Y_0]
            \morphism(2000,1000)|m|<-500,-500>[N`\Cc;Z_0\psi\lambda]
            \morphism(1000,0)|m|<500,500>[I`\Cc;Y_1]
            \morphism(2000,0)|m|<-500,500>[I`\Cc;Z_1]
            \place(1500,250)[\twoar(1,0)]
            \place(1500,150)[^{\beta_1}]
            \place(1500,750)[\twoar(1,0)]
            \place(1500,850)[_{\beta_0}]
            \place(1150,500)[\twoar(0,-1)]
            \place(1250,500)[_{\gamma_Y}]
            \place(1850,500)[\twoar(0,-1)]
            \place(1750,500)[_{\gamma_Z\lambda}]
        \end{xy}.
    \end{equation}
    This 2-commuting diagram represents the map of exact sequences of admissible Ind-objects \ref{eq:indsqfull}.

    Now suppose that the map \eqref{eq:indsqfull} is an admissible monic (epic). In this case, the straightening construction for exact sequence shows that we can assume, without loss of generality, that each component of $\gamma_X$ or $\gamma_Z$ is an admissible monic (epic). The same is therefore true for the maps $\gamma_X\mu$ and $\gamma_Z\lambda$ in \eqref{indemap}. By the 5-Lemma (\cite[Corollary 3.2]{Buh:10}), this implies that each component of $\gamma_Y$ is an admissible monic (epic). We conclude that every exact sequence in $\Ec\Indk(\Cc)$ is the image, under \eqref{fun:indetoeind}, of an exact sequence in $\Indk(\Ec\Cc)$.
\end{proof}

\subsubsection{Admissible Ind-Objects and the S-Construction}
\begin{definition}[Waldhausen]
    Let $\Cc$ be an exact category. For $k\ge 0$, define $S_k\Cc$ to be the category whose objects are chains of length $k$ of admissible monics in $\Cc$
    \begin{equation*}
        0\into X_1\into\cdots\into X_k.
    \end{equation*}
    Morphisms are commuting diagrams in $\Cc$ of the obvious form.
\end{definition}

The 5-Lemma \cite[Corollary 3.2]{Buh:10} shows that $S_k\Cc$ is closed under extensions in the exact category $\Fun([n],\Cc)$, so $S_k\Cc$ inherits a canonical exact structure in which admissible monics (epics) are maps of sequences which are admissible monics (epics) at each term in the sequence.

\begin{proposition}\label{prop:indsk=skind}
    Let $\Cc$ be an exact category. For each $k\ge 0$, the canonical map $\Indk(S_k\Cc)\to S_k\Indk(\Cc)$ is an exact equivalence.
\end{proposition}
\begin{proof}
    The proof is an elaboration of the proof of Proposition \ref{prop:inde=eind}. We prove the proposition by induction on $k$. For $k\le 1$, the statement is trivial.  For $k=2$, this is Proposition \ref{prop:inde=eind}. Now assume that we have shown the result for $k-1$.

    We begin by showing that the functor is an equivalence of categories. By inspection, it is faithful. That it is full follows by induction and the proof of Proposition Proposition \ref{prop:inde=eind}. We now show that the functor is essentially surjective. Let
    \begin{equation}\label{indsurj}
        \widehat{X}_1\into\cdots\into\widehat{X}_k
    \end{equation}
    be an object of $S_k\Indk(\Cc)$. By the inductive hypothesis, there exists a sequence of admissible monics
    \begin{equation}\label{indinduct}
        \begin{xy}
            \Vtrianglepair/=`=`>``>/[I`\cdots`I`\Cc;``X_1`\cdots`X_{k-1}]
            \place(350,350)[\twoar(1,0)]
            \place(650,350)[\twoar(1,0)]
        \end{xy}
    \end{equation}
    of admissible Ind-diagrams whose colimit is isomorphic to
    \begin{equation*}
        \widehat{X}_1\into\cdots\into\widehat{X}_{k-1}.
    \end{equation*}
    We next straighten the exact sequence
    \begin{equation*}
        \widehat{X}_{k-1}\into\widehat{X}_k\onto\widehat{X}_k/\widehat{X}_{k-1}
    \end{equation*}
    as in the proof of Proposition \ref{prop:inde=eind}, and obtain a directed poset $K$, a final map $K\to I$ and a 2-commuting diagram
    \begin{equation*}
        \begin{xy}
            \Vtrianglepair/<-`=`>`>`>/[I`K`K`\Cc;``X_{k-1}``X_k]
            \place(650,350)[\twoar(1,0)]
        \end{xy}
    \end{equation*}
    in which the left triangle strictly commutes, and the right triangle is an Ind-diagram of admissible monics. By restricting \eqref{indinduct} along the final map $K\to I$, we obtain an admissible Ind-diagram in $S_k\Cc$ whose colimit is isomorphic to \eqref{indsurj}.

    It remains to show that the functor is fully exact. Exactness is clear. That it is fully exact follows from the inductive hypothesis and the proof of Proposition \ref{prop:inde=eind}.
\end{proof}

\subsubsection{Admissible Ind-Objects as a Localization}
\begin{proposition}[See also \cite{Pre:11}]\label{prop:indloc}
    Denote by $W\subset\Dirk(\Cc)$ the sub-category consisting of all morphisms of admissible Ind-diagrams given by strictly commuting triangles
    \begin{equation*}
        \begin{xy}
            \Vtriangle<350,400>[I`J`\Cc;\varphi`X`Y]
        \end{xy}
    \end{equation*}
    in which the map $\varphi$ is final. The functor $\widehat{(-)}\colon\Dirk(\Cc)\to\Indk(\Cc)$ takes morphisms in $W$ to isomorphisms of admissible Ind-objects. The induced functor
    \begin{equation*}
        \Dirk(\Cc)[W^{-1}]\to\Indk(\Cc)
    \end{equation*}
    is an equivalence of categories.
\end{proposition}
\begin{proof}
    The functor is essentially surjective by definition. Straightening (Lemma \ref{lemma:straight}) shows that it is full. It remains to show that it is faithful.

    Suppose there exist morphisms
    \begin{equation*}
        \begin{xy}
            \morphism|a|/{@<3pt>}/[X`Y;(\varphi_0,\alpha_0)]
            \morphism|b|/{@<-3pt>}/[X`Y;(\varphi_1,\alpha_1)]
        \end{xy}
    \end{equation*}
    in $\Dirk(\Cc)$ which induce equal maps of admissible Ind-objects. For $a=0,1$, the pair $(\varphi_a,\alpha_a)$ induces a section of the map $I\downarrow_{\Cc}J\rightarrow I$. These sections fit into a commuting triangle
    \begin{equation*}
        \begin{xy}
            \Vtriangle/<-`>`>/<350,400>[I`I\downarrow_{\Cc}J`\Cc;`X`]
            \morphism(0,400)|a|/{@<5pt>}/<700,0>[I`I\downarrow_{\Cc}J;(\varphi_a,\alpha_a)]
        \end{xy}
    \end{equation*}
    The existence of this commuting triangle implies that, for $a=0,1$, the image of the map $X\to^{(\varphi_a,\alpha_a)} Y$ in the localization $\Dirk(\Cc)[W^{-1}]$ is equal to the map represented by the zig-zag
    \begin{equation*}
        \begin{xy}
            \Vtrianglepair/<-`>`>`>`>/[I`I\downarrow_{\Cc}J`J`\Cc;``X``Y]
            \place(650,250)[\twoar(1,0)]
        \end{xy}
    \end{equation*}
    We conclude that the functor is faithful.
\end{proof}

\subsubsection{Functoriality of the Construction}
\begin{proposition}\label{prop:FtoIndF}
    An exact functor $F\colon\Cc\to\Dc$ extends canonically to an exact functor $\widetilde{F}\colon\Indk(\Cc)\to\Indk(\Dc)$ which fits into a 2-commuting diagram
    \begin{equation*}
        \begin{xy}
            \square/>`>`>`-->/<750,500>[\Cc`\Dc`\Indk(\Cc)`\Indk(\Dc);F```\exists\widetilde{F}]
            \place(375,250)[\cong]
        \end{xy}.
    \end{equation*}
    If $F$ is faithful, fully faithful, or an equivalence, then so is $\widetilde{F}$.
\end{proposition}
\begin{proof}
    Let $P$ be a left exact presheaf on $\Cc$. Recall that $\Cc\downarrow P$ denotes the category of elements of $P$ (Definition \ref{def:catel}), and that $\colim_{\Cc\downarrow P}X\cong P$. The functor $F$ induces a colimit-preserving functor
    \begin{equation*}
        \begin{xy}
            \morphism<750,0>[\lex(\Cc)`\lex(\Dc);\widetilde{F}]
            \morphism(0,-250)/|->/<750,0>[P`\colim_{\Cc\downarrow P}FX;]
        \end{xy}
    \end{equation*}
    We show that $\widetilde{F}$ preserves admissible Ind-objects. Let $\widehat{X}\in\Indk(\Cc)$ be represented by an admissible Ind-diagram
    \begin{equation*}
        \begin{xy}
            \morphism[I`\Cc;X]
        \end{xy}.
    \end{equation*}
    The functor $F$ is exact, so
    \begin{equation*}
        \begin{xy}
            \morphism[I`\Dc;FX]
            \morphism(0,-200)/|->/[i`FX_i;]
        \end{xy}
    \end{equation*}
    is an admissible Ind-diagram in $\Dc$. The canonical map
    \begin{equation*}
        \begin{xy}
            \morphism<750,0>[I`\Cc\downarrow\widehat{X};]
            \morphism(0,-200)/|->/<750,0>[i`(X_i\rightarrow\widehat{X});]
        \end{xy}
    \end{equation*}
    induces an isomorphism
    \begin{equation*}
        \begin{xy}
            \morphism<1000,0>[\widehat{X}\cong\colim_I X_i`\colim_{\Cc\downarrow\widehat{X}}X;\cong]
        \end{xy}.
    \end{equation*}
    Similarly, we obtain an isomorphism
    \begin{equation*}
        \begin{xy}
            \morphism<1250,0>[\colim_I FX_i`\colim_{\Cc\downarrow\widehat{X}}FX=:\widetilde{F}\widehat{X};\cong]
        \end{xy}.
    \end{equation*}
    This shows that $\widetilde{F}\widehat{X}\in\Indk(\Dc)$.

    Now let
    \begin{equation*}
        \begin{xy}
	       \morphism/^{ (}->/[\widehat{X}`\widehat{Y};]
	       \morphism(500,0)/->>/[\widehat{Y}`\widehat{Z};]
    	\end{xy}
    \end{equation*}
    be an exact sequence in $\Indk(\Cc)$. Represent this as a sequence of admissible Ind-diagrams
    \begin{equation*}
        \begin{xy}
            \Vtrianglepair[I`J`K`\Cc;``X`Y`Z]
            \place(375,250)[\twoar(1,0)]
            \place(375,350)[_{\alpha}]
            \place(650,250)[\twoar(1,0)]
            \place(650,350)[_{\beta}]
        \end{xy}
    \end{equation*}
    where the components of $\alpha$ are admissible monics, and the components of $\beta$ are admissible epics. Apply the exact functor $F$ to obtain a sequence of admissible Ind-diagrams in $\Dc$
    \begin{equation*}
        \begin{xy}
            \Vtrianglepair[I`J`K`\Dc;``FX`FY`FZ]
            \place(375,250)[\twoar(1,0)]
            \place(375,350)[_{F\alpha}]
            \place(650,250)[\twoar(1,0)]
            \place(650,350)[_{F\beta}]
        \end{xy}.
    \end{equation*}
    Taking the colimits, we obtain an exact sequence in $\Indk(\Dc)$ isomorphic to the sequence
    \begin{equation*}
        \begin{xy}
	       \morphism[\widetilde{F}\widehat{X}`\widetilde{F}\widehat{Y};]
	       \morphism(500,0)[\widetilde{F}\widehat{Y}`\widetilde{F}\widehat{Z};]
    	\end{xy}.
    \end{equation*}
    We conclude that $\widetilde{F}$ is exact.

    Now suppose $F$ is faithful. For any $X\colon I\to\Cc$ and $Y\colon J\to\Cc$ in $\Dirk(\Cc)$, we have
    \begin{equation*}
        \hom_{\Indk(\Cc)}(\widehat{X},\widehat{Y})\cong\lim_I\colim_J\hom_{\Cc}(X_i,Y_j).
    \end{equation*}
    The construction of directed colimits and inductive limits in the category of sets shows that a map of directed or inductive diagrams which is injective at each object in the diagram induces an injection in the colimit or limit. If $F$ is faithful, then
    \begin{align*}
        \hom_{\Indk(\Cc)}(\widehat{X},\widehat{Y})&\cong\lim_I\colim_J\hom_{\Cc}(X_i,Y_j)\\
        &\subseteq \lim_I\colim_J\hom_{\Dc}(FX_i,FY_j)\\
        &\cong \hom_{\Indk(\Dc)}(\widetilde{F}\widehat{X},\widetilde{F}\widehat{Y})
    \end{align*}
    and we conclude that $\widetilde{F}$ is faithful as well. If $F$ is fully faithful, the previous argument shows $\widetilde{F}$ is as well. If $F$ is fully faithful and essentially surjective, then any diagram in $\Dirk(\Dc)$ is equivalent to the image of a diagram in $\Dirk(\Cc)$; therefore $\widetilde{F}$ is essentially surjective as well.
\end{proof}

\subsubsection{Countable Admissible Ind-Objects}\label{sec:countableind}
Countable admissible Ind-objects have been in the literature for some time. We take Keller \cite{Kel:90} as a basic reference.

\begin{definition}\cite[Appendix B]{Kel:90}
    The \emph{countable envelope} $\Cc^{\thicksim}$ of an exact category $\Cc$ is the full sub-category of $\lex(\Cc)$ consisting of all left-exact presheaves $\widehat{X}$ which are representable by an admissible Ind-diagram $X\colon\mathbb{N}\to\Cc$.
\end{definition}

\begin{proposition}\label{prop:countableind}
    The embedding $\Cc^{\thicksim}\into\Indc(\Cc)$ is an equivalence of exact categories.
\end{proposition}
\begin{proof}
    The embedding is fully faithful by definition. We show it is essentially surjective. Let $\widehat{X}\in\Indc(\Cc)$ be represented by a countable admissible Ind-diagram $X\colon I\to\Cc$. Every countable directed poset $I$ admits a final map $f:\mathbb{N}\to I$.\footnote{Pick a bijection $\mathbb{N}\rightarrow I$. We construct $f$ by induction. Let $f(0):=i_0$. Suppose we have defined $f(l)$ for $0\le l\le n$. Pick $f(n+1)\in I$ such that $f(n+1)\geq i_l$ for $0\leq l\leq n$ and such that $f(n+1)\geq f(n)$. This completes the induction step.}  The isomorphism $\colim_{n\in\mathbb{N}} X_{f(n)}\cong\widehat{X}$ shows that $\widehat{X}\in\Cc^{\thicksim}$.
\end{proof}

Recall that an exact category $\Cc$ is \emph{split exact} if every exact sequence in $\Cc$ splits.
\begin{corollary}\label{cor:indcsum}
    If $\Cc$ is split exact, then $\Indc(\Cc)\subset\lex(\Cc)$ is the full sub-category consisting of countable direct sums of objects in $\Cc$.
\end{corollary}
\begin{proof}
    Let $\widehat{X}\in\Indc(\Cc)$. Proposition \ref{prop:countableind} shows that $\widehat{X}$ is the colimit of an admissible diagram $X\colon\Nb\to\Cc$. If $\Cc$ is split exact, then for each $i$, there exists a splitting of the admissible monic $X_i\into X_{i+1}$ into the inclusion of a summand $X_i\into X_i\oplus X_{i,i+1}\cong X_{i+1}$. By induction on $i$, we conclude that $\widehat{X}\cong X_0 \oplus \bigoplus_{i\in\Nb} X_{i,i+1}$.
\end{proof}

\begin{example}
    Let $R$ be a commutative Noetherian ring with $\Spec(R)$ connected. Denote by $P_{\aleph_0}(R)$ the category of countably generated projective $R$-modules. Bass's Theorem \cite[Corollary 4.5]{Bas:63} shows that every infinitely generated module $M\in P_{\aleph_0}(R)$ is a free. We conclude that $P_{\aleph_0}(R)\simeq\Indc(P_f(R))$.
\end{example}

Note that in the example above, every object in $\Indc(P_f(R))$ is a direct summand of the free module $R[t]$. This is part of a general phenomenon.
\begin{proposition}\label{prop:bigindobject}
    Let $\Cc$ be a split exact category for which there exists a collection of objects $\{S_i\}_{i\in\Nb}\subset\Cc$ such that every object $Y\in\Cc$ is a direct summand of $\bigoplus_{i=0}^n S_i$ for some $n$. Then every countable Ind-object in $\Cc$ is a direct summand of
    \begin{equation*}
        \widehat{\bigoplus_{\Nb} S}:=\bigoplus_{\Nb}(\bigoplus_{i\in\Nb} S_i).
    \end{equation*}
\end{proposition}
\begin{proof}
    Let $\widehat{X}\in\Indc(\Cc)$. By Corollary \ref{cor:indcsum}, there exist $X_i\in\Cc$ for $i\in\Nb$, such that $\widehat{X}\cong\bigoplus_{i\in\Nb} X_i$. By assumption, for each $i\in\Nb$, there exists $Y_i\in\Cc$ and $n_i\in\Nb$ such that \begin{equation*}
        X_i\oplus Y_i\cong \bigoplus_{i=0}^{n_i} S_i.
    \end{equation*}
    As a result, we have
    \begin{align*}
        \bigoplus_{\Nb}(\bigoplus_{i\in\Nb} S_i)&\cong \bigoplus_{i\in\Nb}(X_i\oplus Y_i\oplus\bigoplus_{n>n_i} S_n)\\
        &\cong \widehat{X}\oplus(\bigoplus_{i\in\Nb} Y_i\oplus\bigoplus_{n>n_i} S_n).
    \end{align*}
\end{proof}

\begin{proposition}\label{prop:indcsplit}
    Let $\Cc$ be a split exact category. Then $\Indc(\Cc)$ is split exact.
\end{proposition}
\begin{proof}
    By Propositions \ref{prop:inde=eind} and \ref{prop:countableind}, it suffices to show that every exact sequence of countable admissible Ind-diagrams
    \begin{equation*}
        \begin{xy}
            \Vtrianglepair/=`=`>`>`>/[\mathbb{N}`\mathbb{N}`\mathbb{N}`\Cc;``X`Y`Z]
            \place(375,250)[\twoar(1,0)]
            \place(375,350)[_{\alpha}]
            \place(650,250)[\twoar(1,0)]
            \place(650,350)[_{\beta}]
        \end{xy}
    \end{equation*}
    splits. By the usual argument, it suffices to construct a splitting of $\alpha$. Denote by
    \begin{equation*}
        \alpha_{\le n}\colon X_{\le n}\to Y_{\le n}
    \end{equation*}
    the restriction of $\alpha$ to $\{0<\ldots <n\}\subset\mathbb{N}$. We induct on $n$ to show that a retract of $\alpha_{\le n}$ exists for all $n$.

    Because $\Cc$ is split exact, a retract of $\alpha_0\colon X_0\onto Y_0$ exists. Now suppose that a retract $\rho_{\le n}$ of $\alpha_{\le n}$ exists. It suffices to construct a retract $\rho_{n+1}$ of $\alpha_{n+1}$ which fits into a commuting square
    \begin{equation*}
        \begin{xy}
            \square/>`<<-`<<-`>/[X_n`X_{n+1}`Y_n`Y_{n+1};`\rho_n`\rho_{n+1}`]
        \end{xy}
    \end{equation*}
    The maps $\rho_n$ and $\alpha_n$ induce an isomorphism $Y_n\cong X_n\oplus B_0$. We can also choose splittings  $X_{n+1}\cong X_n\oplus A$ and $Y_{n+1}\cong Y_n\oplus B_1$. With respect to these splittings, we can write $\alpha_{n+1}$ as the map $(1_{X_n}+\chi,\sigma_0,\sigma_1)$ where $\chi\colon A\to X_n$, and $\sigma_i\colon A\to B_i$ for $i=0,1$.

    The map $(\sigma_0,\sigma_1)\colon A\to B_0\oplus B_1$ is an admissible monic, as it is isomorphic to the pushout of $\alpha_{n+1}$ along the projection $\pi_A\colon X_{n+1}\onto A$. Choose a retraction $\rho'$ of $(\sigma_0,\sigma_1)$, and define $\rho_{n+1}:=\pi_{X_n}\oplus\rho'-\chi\rho'\pi_{B_0\oplus B_n}$, where $\pi_{(-)}$ denotes the projection onto $(-)$. Then $\rho_{n+1}$ is a retraction of $\alpha_{n+1}$ and fits into the commuting square above. This completes the induction.
\end{proof}

\begin{remark}
    As explained to us by J. \v S\v tov\'\i\v cek and J. Trlifaj, this proposition fails badly for uncountable $\kappa$; see Proposition \ref{r:split} for a family of counter-examples.
\end{remark}

\subsubsection{\texorpdfstring{$\Indk(\Cc)$}{Ind(C)} is Not Generally Idempotent Complete}\label{sec:notic}
In this section, we present a simple example of an idempotent complete category $\Cc$ for which $\Indk(\Cc)$ is not idempotent complete.\footnote{We are grateful to J. \v S\v tov\'\i\v cek and J. Trlifaj for pointing us to this example, and thus providing a much simpler alternative to our original discussion.} Note that Freyd \cite{Fre:66} has shown that an additive category which admits infinite direct sums, such as $\Indk(\Cc)$, is weakly idempotent complete if and only if it is idempotent complete. This example shows that, in general, neither condition on $\Cc$ is inherited by $\Indk(\Cc)$.

\begin{proposition}
    Let $R$ be a ring for which there exists a countably generated projective $R$-module $M$ which is not a direct sum of finitely generated projective modules.\footnote{An example of such a ring dates to Kaplansky; see e.g.\ \cite[(2.12D)]{Lam:99}.} Then $\Indc(P_f(R))$ is not idempotent complete.
\end{proposition}
\begin{proof}
    Because $M$ is projective and countably generated, it is a direct summand of the free $R$-module $R[t]$. Indeed, a generating set $\{x_i\}_{i\in\Nb}\subset M$ determines a surjection $\pi\colon R[t]\to M$ by the assignment $t^i\mapsto x_i$; the projectivity of $M$ ensures that a splitting $\alpha$ of this surjection exists. Recall that $\lex(P_f(R))\simeq\Mod(R)$ (Lemma \ref{lemma:lexR}). Corollary \ref{cor:indcsum} shows that $\Indc(P_f(R))\subset\Mod(R)$ is the full sub-category consisting of countable direct sums of finitely generated projective modules.  In particular, we have $R[t]\in\Indc(P_f(R))$, but $M$ is not an object in $\Indc(P_f(R))$. We conclude that the idempotent of $R[t]$ which corresponds to $M$ does not split in $\Indc(P_f(R))$.
\end{proof}

\begin{proposition}\label{prop:pc}
    Let $R$ be a ring. The category $P_{\aleph_0}(R)$ of countably generated projective $R$-modules is equivalent to the idempotent completion of $\Indc(P_f(R))$.
\end{proposition}
\begin{proof}
    Lemma \ref{lemma:lexR} shows that $\Indc(P_f(R))$ is a full sub-category of $\Mod(R)$. Corollary \ref{cor:indcsum} shows that every module in $\Indc(P_f(R))$ is projective. Therefore, it suffices to observe that every countably generated projective module is a direct summand of a (free) module in $\Indc(P_f(R))$, as we did above.
\end{proof}

\subsection{The Categories \texorpdfstring{$\FM(\Cc)$}{FM(C)} and \texorpdfstring{$P(\Cc)$}{P(C)}}
We now introduce two categories suggested by the study of $R$-modules.
\begin{definition}\label{def:fmk}
    Let $\Cc$ be an exact category. The category $\FMk(\Cc)$ of \emph{$\kappa$-generated flat Mittag-Leffler objects} in $\Cc$ is the idempotent completion of $\Indk(\Cc)$, i.e.
    \begin{equation*}
        \FMk(\Cc):=\Indk(\Cc)^{\ic}.
    \end{equation*}
\end{definition}

\begin{remark}
    A bi-exact tensor product $\otimes$ on $\Cc$ extends to a tensor product $\otimes$ on $\lex(\Cc)$ (by Day convolution). One can show that every object in $\FM(\Cc)$ is flat with respect to this tensor product.
\end{remark}

Recall that every set is the colimit of its directed poset of finite subsets. As a result, the category $\Ind(\Cc)$ is closed under arbitrary direct sums. This justifies the following.
\begin{definition}
    Let $\Cc$ be a split exact category. The category $P_\kappa(\Cc)$ of \emph{$\kappa$-generated projective objects} in $\Cc$ is the full sub-category of $\Indk(\Cc)^{\ic}$ consisting of direct summands of arbitrary direct sums of objects in $\Cc$.
\end{definition}

\begin{remark}
    If $\Cc$ is split exact, then the definition of $\lex(\Cc)$ ensures that $\Cc$ includes as a sub-category of projectives. One can similarly show that $P(\Cc)$ is a sub-category of projectives in $\lex(\Cc)$.
\end{remark}

\subsection{Admissible Ind R-Modules}\label{sec:indRmod}
Let $R$ be a ring. In this section, we describe the categories of admissible Ind-objects in various categories of finitely generated $R$-modules.

We have already shown that countably generated projective $R$-modules are direct summands of objects in $\Indc(P_f(R))$ (Proposition \ref{prop:pc}). Other simple examples include the following.
\begin{example}\mbox{}
    \begin{enumerate}
        \item Let $R$ be a Noetherian ring. Denote by $\Mod_f(R)$ the abelian category of finitely presented $R$-modules. Every $R$-module is the directed colimit of its finitely presented sub-modules. As a consequence, the category $\Indk(\Mod_f(R))$ is equivalent to the abelian category of $R$-modules having at most $\kappa$ generators. If we omit the cardinality bound, we have $\Ind(\Mod_f(R))\simeq\Mod(R)$.
        \item Similarly, if $X$ is a Noetherian scheme, $\Ind(\Coh(X))\simeq\QCoh(X)$.
    \end{enumerate}
\end{example}

We now turn to two particular examples of interest: projective modules and flat Mittag-Leffler modules. Recall the following theorem of Kaplansky \cite{Kap:58}.
\begin{theorem}[Kaplansky]
    Let $R$ be a ring. Every projective $R$-module is a direct sum of countably generated projective modules.
\end{theorem}

\begin{corollary}\label{cor:projRgeneral}
    The category $P(R)$ of projective $R$-modules is equivalent to the category $P(P_f(R))$.
\end{corollary}
\begin{proof}
    Lemma \ref{lemma:lexR} shows that $P(P_f(R))$ is a full sub-category of $\Mod(R)$. Further, the definition of $P(P_f(R))$ ensures that every object in $P(P_f(R))$ is projective. It suffices to show that all projective modules are objects in $P(P_f(R))$.

    Let $M$ be a projective module. Kaplansky's theorem shows that there exists a set $I$ and countable generated projective modules $\{M_i\}_{i\in I}$ such that $M\cong\bigoplus_I M_i$. By Proposition \ref{prop:pc}, $M_i\in P(P_f(R))$ for all $i$, and the definition ensures that the sub-category $P(P_f(R))\subset\Mod(R)$ is closed under arbitrary direct sums. We conclude that $M\in P(P_f(R))$.
\end{proof}

Given a finitely generated projective (left) module $M\in P_f(R)$, denote by $M^\vee$ the finitely generated projective (right) module $\hom_R(M,R)$.
\begin{definition}\label{def:fml}
    A \emph{flat Mittag-Leffler module} over $R$ is a (left) module $M$ which is isomorphic to the colimit of a directed diagram $M\colon I\rightarrow P_f(R)$ such that for every $i\in I$ there exists $j\geq i$, with
    \begin{equation*}
        Im(M_k^{\vee}\rightarrow M_i^{\vee})=Im(M_j^\vee\rightarrow M_i^\vee)
    \end{equation*}
    for all $k\geq j$.
\end{definition}

\begin{remark}
    Mittag-Leffler modules were introduced by Raynaud and Gruson \cite{RG:71}. As observed in \cite{Dri:06}, when a Mittag-Leffler module is also flat, the Govorov--Lazard characterization of flat modules (\cite{Gov:65}, \cite[Th\'{e}or\`{e}me 1.2(iii)]{Laz:69}) allows one to restate Raynaud and Gruson's definition in the above form.
\end{remark}

\begin{proposition}\label{prop:indfml}
    The category $\Ind(P_f(R))$ embeds as a full sub-category of the category $\FM(R)$ of flat Mittag-Leffler $R$-modules.
\end{proposition}
\begin{proof}
    By definition, $\Ind(P_f(R))$ is a full sub-category of $\lex(P_f(R))\simeq\Mod(R)$. We show that every module in $\Ind(P_f(R))$ is flat and Mittag-Leffler. Indeed, an admissible monic $N_i\into N_j$ in $P_f(R)$ is an inclusion of a direct summand. In particular, $N_i\into N_j$ induces a surjective map $N_j^\vee\onto N_i^\vee$. We conclude that the colimit in $\Mod(R)$ of any admissible Ind-diagram in $P_f(R)$ is a flat Mittag-Leffler module.
\end{proof}

The following is proven in Appendix \ref{sec:stovicektrlifaj} (see Lemma \ref{l:rel} and Proposition \ref{p:more}).
\begin{proposition}[J. \v S\v tov\'\i\v cek, J. Trlifaj]\label{prop:fmR}
    The categories $\FM(P_f(R))$ and $\FM(R)$ are equivalent. Similarly, the category $\FMk(P_f(R))$ is equivalent to the category $\FMk(R)$ of $\kappa$-generated flat Mittag-Leffler modules.
\end{proposition}

\begin{remark}\mbox{}
    \begin{enumerate}
        \item In combination with \cite{EGT:11}, \v S\v tov\'\i\v cek and Trlifaj's proposition shows that the categories $\FMk(P_f(R))$ satisfy Zariski descent for all $\kappa$. At present, we do not know if this is a general phenomenon for other sheaves of exact categories.
        \item More generally, Mittag-Leffler conditions on modules have been the subject of considerable investigation, including \cite{RG:71}, \cite{AHH:08}, \cite{HT:12}, \cite{EGPT:12}, \cite{EGT:11} and \cite{BaS:12}.
    \end{enumerate}
\end{remark}

As an immediate corollary, we recover a classical result of Raynaud and Gruson \cite{RG:71}.
\begin{corollary}
    A countably generated $R$-module $M$ is projective if and only if it is flat and Mittag-Leffler.
\end{corollary}
\begin{proof}
    By definition $\FM_{\aleph_0}(P_f(R))=P_{\aleph_0}(P_f(R))=\Indc(P_f(R))^{\ic}$. Corollary \ref{cor:projRgeneral} shows that$P_{\aleph_0}(R)\simeq P_{\aleph_0}(P_f(R))$. Proposition \ref{prop:fmR} shows that $\FM_{\aleph_0}(R)\simeq\FM_{\aleph_0}(P_f(R))$.
\end{proof}

This combines with Kaplansky's theorem to give the following.
\begin{theorem}[Kaplansky--Raynaud--Gruson]
    An $R$-module $M$ is projective if and only if $M$ is a direct sum of countably generated $R$-modules and $M$ is a flat Mittag-Leffler module.
\end{theorem}

\subsection{The Calkin Category}\label{sec:cal}
Let $R$ be a ring. The Calkin algebra $\Calk(R)$ is the quotient of the algebra $\End_R(R[t])$ by the ideal of finite rank endomorphisms. By analogy, we define the following.

\begin{definition}
    Let $\Cc$ be an exact category. Let $\kappa$ be an infinite cardinal. Define the \emph{$\kappa$-Calkin category} by
    \begin{equation*}
        \Calkk(\Cc):=(\Indk(\Cc)/\Cc)^{\ic}.
    \end{equation*}
\end{definition}

\begin{example}\mbox{}
    \begin{enumerate}
        \item For any ring $R$, we have
            \begin{equation*}
                \End_{\Calkc(P_f(R))}(R[t])\cong\Calk(R).
            \end{equation*}
        \item If $R$ is local, then $\Indc(P_f(R))/P_f(R)$ has two objects up to isomorphism, the zero object and $R[t]$.\footnote{Because $R$ is local, all projective modules are free, so Corollary \ref{cor:indcsum} shows that $R[t]$ is the only object in $\Indc(P_f(R))\setminus P_f(R)$.} By comparison, the category $\Calkc(P_f(R))$ has additional objects corresponding to $p\in\End_R(R[t])$ such that $p^2-p$ has finite rank.
    \end{enumerate}
\end{example}

\begin{remark}\label{rmk:Kcalk}
    One feature of the Calkin algebra $\Calk(R)$ is the $K$-theory isomorphism
    \begin{equation*}
        K_i(\Calk(R))\cong K_{i-1}(R)
    \end{equation*}
    Schlichting \cite{Sch:04} has also shown that
    \begin{equation*}
        K_i(\Calkk(\Cc))\cong K_{i-1}(\Cc)
    \end{equation*}
    when $\Cc$ is idempotent complete.
\end{remark}

\section{Admissible Pro-Objects}\label{sec:pro}
Admissible Pro-objects are dual to admissible Ind-objects. They sit in relation to objects of $\Cc$ as the topological $R$-module $R[[t]]$, with the $t$-adic topology, sits in relation to finitely generated free $R$-modules. More generally, the results of this section form an elaboration, in the setting of exact categories, of Artin--Mazur \cite[Appendix 2]{ArM:69}.

\subsection{The Category of Admissible Pro-Objects and its Properties}
\begin{definition}\label{defi:pro}
    Let $\Cc$ be an exact category. Let $\kappa$ be an infinite cardinal. The category $\Prok(\Cc)$ of \emph{admissible Pro-objects in $\Cc$ of size at most $\kappa$} is the opposite of the category $\Indk(\Cc^{\op})$, i.e.
    \begin{equation*}
        \Prok(\Cc):=\Indk(\Cc^{\op})^{\op}.
    \end{equation*}
    We can also omit the cardinality bound and define
    \begin{equation*}
        \Pro(\Cc):=\Ind(\Cc^{\op})^{\op}.
    \end{equation*}
\end{definition}

If $X\colon I\to\Cc$ and $Y\colon J\to\Cc$ are admissible Pro-diagrams in $\Cc$, then the definition ensures that
\begin{align*}
    \hom_{\Prok(\Cc)}(\widehat{X},\widehat{Y})=\lim_J\colim_I\hom_{\Cc}(X_i,Y_j)
\end{align*}
In particular, we see that, for any map $f\colon \widehat{X}\to Y$ in $\Prok(\Cc)$ with $Y\in\Cc$, and for any admissible diagram $X\colon I\to\Cc$ representing an $\widehat{X}$, there exists $i\in I$ such that $f$ factors through some $\widehat{X}\onto X_i$.

The proofs in the previous section dualize to give the following.
\begin{theorem}\label{thm:prox}\mbox{}
    \begin{enumerate}
        \item The category $\Prok(\Cc)$ is closed under extensions in $\lex(\Cc^{\op})^{\op}$. If $\Cc$ is weakly idempotent complete, then $\Prok(\Cc)$ is right special in $\lex(\Cc^{\op})^{\op}$.
        \item An exact category $\Cc$ embeds as a right s-filtering sub-category of $\Prok(\Cc)$.
        \item The exact category $\Prok(\Ec\Cc)$ is canonically equivalent to $\Ec\Prok(\Cc)$.
        \item Define the category $\Invk(\Cc)$ of admissible Pro-diagrams by
            \begin{equation*}
                \Invk(\Cc):=\Dirk(\Cc^{\op})^{\op}.
            \end{equation*}
            The category $\Prok(\Cc)$ is the localization of $\Invk(\Cc)$ at the sub-category of co-final morphisms.
        \item An exact functor $F:\Cc\to\Dc$ extends canonically to an exact functor
            \begin{equation*}
                \begin{xy}
                    \morphism<750,0>[\Prok(\Cc)`\Prok(\Dc);\widetilde{F}]
                \end{xy}
            \end{equation*}
            If $F$ is faithful, fully faithful, or an equivalence, then the same is true of $\widetilde{F}$.
        \item If $\Cc$ is split exact, then:
            \begin{enumerate}
                \item the category $\Proc(\Cc)$ is the full sub-category of $\Pro(\Cc)$ consisting of countable products of objects in $\Cc$,
                \item the category $\Proc(\Cc)$ is split exact, and
                \item if there exists a collection of objects $\{S_i\}_{i\in\Nb}\subset\Cc$ such that every object $Y\in\Cc$ is a direct summand of $\bigoplus_{i=0}^n S_i$ for some $n$. Then every countable Pro-object in $\Cc$ is a direct summand of
                    \begin{equation*}
                        \widehat{\prod_{\Nb} S}:=\prod_{\Nb}(\prod_{i\in\Nb} S_i).
                    \end{equation*}
            \end{enumerate}
        \item The category $\Prok(\Cc)$ is not in general idempotent complete.
    \end{enumerate}
\end{theorem}

With minimal changes, the proof of Proposition \ref{prop:indsk=skind} can be modified to show the following.
\begin{proposition}
    For $k\ge 0$, the exact category $\Prok(S_k\Cc)$ is canonically equivalent to $S_k\Prok(\Cc)$.
\end{proposition}

\subsection{The Categories \texorpdfstring{$\FM^\vee(\Cc)$}{FMv(C)} and \texorpdfstring{$P^\vee(\Cc)$}{Pv(C)}}
\begin{definition}
    Let $\Cc$ be an exact category. Define the category $\FMk^\vee(\Cc)$ of \emph{formal duals to $\kappa$-generated flat Mittag-Leffler objects} in $\Cc$ to be the idempotent completion of $\Prok(\Cc)$, i.e.
    \begin{equation*}
        \FMk^\vee(\Cc):=\Prok(\Cc)^{\ic}.
    \end{equation*}
\end{definition}

\begin{proposition}\label{prop:inddualpro}
    Let $\Cc$ and $\Dc$ be exact categories, and let $F\colon\Cc^{\op}\to\Dc$ be an exact functor. Let $\kappa$ be an infinite cardinal. Then $F$ extends canonically to an exact functor
    \begin{equation*}
        \begin{xy}
            \morphism<750,0>[\FMk(\Cc)^{\op}`\FMk^\vee(\Dc);\widetilde{F}]
        \end{xy}.
    \end{equation*}
    If $F$ is faithful, fully faithful, or an equivalence, then the same is true of $\widetilde{F}$.
\end{proposition}
\begin{proof}
    By definition, $\Prok(\Cc^{\op})=\Indk(\Cc)^{\op}$. Theorem \ref{thm:prox} guarantees that $F$ extends to an exact functor
    \begin{equation*}
        \begin{xy}
            \morphism/=/<500,0>[\Indk(\Cc)^{\op}`\Prok(\Cc^{\op});]
            \morphism(500,0)<750,0>[\Prok(\Cc^{\op})`\Prok(\Dc);\widetilde{F}]
        \end{xy}
    \end{equation*}
    which is faithful, fully faithful, or an equivalence if $F$ is. By the universal property of idempotent completion, $F$ extends to a functor $\FMk(\Cc)^{\op}\to\FMk^\vee(\Dc)$ with the same properties.
\end{proof}

\begin{example}\label{ex:Rindprodual}
    Let $R$ be a ring. Denote by $P_f(R^\circ)$ the category of finitely generated projective (right) $R$-modules. By Propositions \ref{prop:fmR} and \ref{prop:inddualpro}, the duality equivalences
    \begin{equation*}
        \begin{xy}
            \morphism|a|/@{->}@<5pt>/<1500,0>[\hom_{R^\circ}(-,R)\colon P_f(R^\circ)^{\op}`P_f(R)\colon\hom_R(-,R);\simeq]
            \morphism|b|/@{<-}@<-5pt>/<1500,0>[\hom_{R^\circ}(-,R)\colon P_f(R^\circ)^{\op}`P_f(R)\colon\hom_R(-,R);\simeq]
        \end{xy}
    \end{equation*}
    extend to the exact equivalences
    \begin{equation*}
        \begin{xy}
            \morphism|a|/@{->}@<5pt>/<2000,0>[\hom_{R^\circ}(-,R)\colon\FMk(R^\circ)^{\op}`\FMk^\vee(P_f(R))\colon\hom_{\FMk^\vee(P_f(R))}(-,R);\simeq]
            \morphism|b|/@{<-}@<-5pt>/<2000,0>[\hom_{R^\circ}(-,R)\colon\FMk(R^\circ)^{\op}`\FMk^\vee(P_f(R))\colon\hom_{\FMk^\vee(P_f(R))}(-,R);\simeq]
        \end{xy}.
    \end{equation*}
\end{example}

\begin{definition}
    Let $\Cc$ be a split exact category. Define the category $P_\kappa^\vee(\Cc)$ of \emph{formal duals of $\kappa$-generated projective objects} in $\Cc$ to be the full sub-category of $\Prok(\Cc)^{\ic}$ consisting of direct summands of arbitrary direct products of objects in $\Cc$.
\end{definition}

\subsection{Admissible Pro R-Modules}
Let $R$ be a ring and denote its opposite by $R^\circ$. We now relate admissible Pro-objects in $P_f(R)$ to categories of topological $R$-modules.

\begin{remark}
    By way of background, the category of admissible Pro objects in finite dimensional vector spaces over a discrete field dates back at least to Lefschetz \cite[Chapter II.25]{Lef:42}. Lefschetz described this as a category of ``linearly compact'' topological vector spaces. One might ask, ``What is a family of linearly compact vector spaces?''

    A first approach, over $\Spec(R)$, might be to consider linearly compact topological $R$-modules; these have been intensively studied by Zelinsky \cite{Zel:53}, Kaplansky \cite{Kap:53} and many others. However, a satisfactory answer must include finitely generated projective $R$-modules with the discrete topology (i.e. families of discrete linearly compact vector spaces); for many rings of interest (e.g. $\mathbb{Z}$ or $k[x]$), such modules fail to be linearly compact.\footnote{Finitely generated projective modules are linearly compact in the discrete topology if $R$ is Artinian, or, more generally, if $R$ is an inverse limit of Artinian rings.}${}^,~$\footnote{We understand this failure for general rings as follows: a topological module encodes the topology on the total space of a family, and this is necessarily a mixture of the topology on the fiber and the topology on the base.}

    In the spirit of \cite{Dri:06}, we instead propose two answers for families over $\Spec(R)$: $P^\vee(P_f(R))$ and $\FM^\vee(P_f(R))$. As we will show, the former admits a natural description in terms of topological $R$-modules, and corresponds to taking (big) projective modules as one's families of discrete vector spaces. The latter embeds faithfully, but not fully, in the category of topological $R$-modules, and corresponds to taking flat Mittag-Leffler modules in lieu of projectives. For more general $\Cc$, $P^\vee(\Cc)$ and $\FM^\vee(\Cc)$ might be thought of as non-commutative families of linearly compact vector spaces.
\end{remark}

Denote by $\Mod(R)_{top}$ the category of topological (left) $R$-modules and continuous homomorphisms. View $P_f(R)$ as a full sub-category of $\Mod(R)_{top}$ consisting of discrete modules. The assignment
\begin{equation*}
    \widehat{M}\mapsto\lim_{(\widehat{M}\downarrow (P_f(R))^{\op})^{\op}}P
\end{equation*}
extends to a faithful embedding
\begin{equation}\label{prointop}
    \begin{xy}
        \morphism<1000,0>[\FM^\vee(P_f(R))`\Mod(R)_{top};\tau]
    \end{xy}.
\end{equation}

\begin{proposition}\label{prop:tprointop}
    The sub-category $P^\vee(P_f(R))$ embeds fully in $\Mod(R)_{top}$ under \eqref{prointop}.
\end{proposition}
\begin{proof}
    Let $\widehat{X}$ and $\widehat{Y}$ be in $P^\vee(P_f(R))$. By definition, there exist isomorphisms $\widehat{X}\cong\prod_I\widehat{X}_i$ and $\widehat{Y}\cong\prod_J\widehat{Y}_j$ for some sets $I$ and $J$, and $\widehat{X}_i,~\widehat{Y}_j\in\Proc(P_f(R))$.  Because $P_f(R)$ is split exact, for each $i\in I$ and $j\in J$, there exist isomorphisms $\widehat{X}_i\cong\prod_{\Nb} X_{i,n}$ and $\widehat{Y}_j\cong\prod_{\Nb} Y_{j,n}$ with $X_{i,n},~Y_{j,n}\in P_f(R)$ for all $n$ (Theorem \ref{thm:prox}). The universal property of limits now ensures that
    \begin{equation*}
        \hom_{cts}(\tau(\widehat{X}),\tau(\widehat{Y}))\cong\prod_J\prod_{\Nb}\hom_{cts}(\tau(\widehat{X}),Y_{j,n})
    \end{equation*}
    It suffices to show that for any $Z\in P_f(R)$, the canonical injection
    \begin{equation}\label{procts}
        \bigoplus_I\bigoplus_{\Nb}\hom_R(X_{i,n},Z)\into\hom_{cts}(\tau(\widehat{X}),Z)
    \end{equation}
    is bijective. Let $f\colon\tau(\widehat{X})\to Z$ be a continuous map. Because $Z$ is discrete, $\ker(f)=f^{-1}(0)$ is an open neighborhood of $0\in\tau(\widehat{X})$. The definition of the product topology ensures that every open neighborhood of $0$ contains the kernel of some projection $\pi_{\overline{\imath},n}\colon\widehat{X}\to\prod_{k=0}^n X_{\imath_k,k}$. The map $\pi_{\overline{\imath},n}$ is open, by definition, and it is a split surjection.\footnote{A right inverse is given by sending $(x_{\imath_0,0},\ldots,x_{\imath_n,n})$ to the point of $\widehat{X}$ with all the same $(\imath_0,0),\ldots,(\imath_n,n))$ coordinates, and all other coordinates equal to 0.} We see that $\ker(\pi_{\overline{\imath},n})\subset\ker(f)$ and $\widehat{X}/\ker(\pi_{\overline{\imath},n})\cong\prod_{k=0}^n X_{\imath_k,k}$. Thus, $f$ factors through the projection $\pi_{\overline{\imath},n}$, and therefore lies in the image of the injection \eqref{procts}.
\end{proof}
\begin{remark}\mbox{}
    \begin{enumerate}
        \item The above argument breaks down for more general objects in $\FM^\vee(P_f(R))$. Indeed, if $X\colon I\to P_f(R)$ is an admissible Pro-diagram with non-vanishing $\lim^1$,\footnote{For $I$ countable, the classical argument for Mittag-Leffler systems shows that $\lim^n_I X_i=0$ for $n>0$. However, this argument fails for $|I|>\aleph_0$. In general, the best one can show is that $\lim^{n+1}_I X_i=0$ if $|I|<\aleph_n$ (see e.g. \cite[Chapters 3 and 9]{Jen:72}).} then for any $i\in I$, we have a long exact sequence
            \begin{equation*}
                0\to\tau(\ker(\pi_i))\to\tau(\widehat{X})\to^{\pi_i} X_i\to^\partial\lim_I{}^1(\ker(X_j\to X_i))\to\lim_I{}^1 X_j\to 0.
            \end{equation*}
            In general the connecting map $\partial$ will be non-zero, and it should be possible to construct continuous maps from $\tau(\widehat{X})$ to a discrete module $Z$ which do not factor through any $\pi_i$.
        \item As the previous remark makes clear, even when $\Cc$ is complete, admissible Pro-objects carry more information than their underlying limit. The discrepancy between derived limits and classical limits is one instance of this. However, even were we to work in a derived setting, it is rarely the case that objects in categories of interest are finitely co-presentable; one passes to the category of Pro-objects in order to fix this.
    \end{enumerate}
\end{remark}

\begin{definition}\label{def:topdual}
    Let $R$ be a ring with the discrete topology. Let $M$ be a topological (right) $R$-module. The \emph{topological dual} $M^{\vee}$ is the (left) module of continuous homomorphisms $\hom_{cts}(M,R)$, endowed with the compact-open topology (i.e. as a sub-space of the space of continuous maps from the space underlying $M$ to the set underlying $R$).
\end{definition}

\begin{example}
    Suppose $M$ is a discrete (right) $R$-module. A basis of open neighborhoods for $M^\vee$ is given by the sets
    \begin{equation*}
        U_{N,I}:=\{f\in M^{\vee}|f(x)\in I\text{ for all }x\in N.\}
    \end{equation*}
    where $I$ is any ideal of $R$, and $N$ is a finitely generated (right) sub-module of $M$.
\end{example}

\begin{corollary}\label{cor:projdualinproic}
    Denote by $P(R^\circ)$ the category of projective (right) $R$-modules. The topological dual gives an equivalence of categories
    \begin{equation*}
        \begin{xy}
            \morphism<1000,0>[P(R^{\circ})^{\op}`P^\vee(P_f(R));(-)^{\vee}]
            \morphism|b|<1000,0>[P(R^{\circ})^{\op}`P^\vee(P_f(R));\simeq]
        \end{xy}.
    \end{equation*}
    and thus a fully faithful embedding
    \begin{equation*}
        P(R^{\circ})^{\op}\into\FM^\vee(P_f(R)).
    \end{equation*}
\end{corollary}
\begin{proof}
    The category $P(R)$ is the idempotent completion of the category $F(R)$ of free $R$-modules. The universal property of colimits shows that
    \begin{equation*}
        (\bigoplus_I R)^\vee\cong\prod_I R\in P^\vee(P_f(R)).
    \end{equation*}
    Along with Proposition \ref{prop:tprointop}, this implies that the topological dual gives a fully faithful embedding
    \begin{equation*}
        (-)^\vee\colon F(R^\circ)^{\op}\to\Mod(R)_{top}
    \end{equation*}
    factors through the embedding
    \begin{equation}\label{tautpro}
        \tau\colon P^\vee(P_f(R))\to\Mod(R)_{top}.
    \end{equation}
    The universal property of idempotent completion ensures that this extends to a fully faithful embedding
    \begin{equation*}
        (-)^\vee\colon P(R^\circ)^{\op}\into P^\vee(P_f(R)).
    \end{equation*}
    To show that this is essentially surjective, it suffices to show that the essential image of \eqref{tautpro} is contained in the essential image of
    \begin{equation*}
        (-)^\vee\colon P(R^\circ)^{\op}\to\Mod(R)_{top}
    \end{equation*}
    Theorem \ref{thm:prox} shows that every object $\widehat{M}\in P^\vee(P_f(R))$ is a direct summand of $\prod_I M_i$ with $M_i\in P_f(R)$ for all $i$. We show $\tau(\prod_I M_i)$ is the topological dual of some $N\in P(R^\circ)$.

    Indeed, because $((M_i)^\vee)^\vee\cong M_i$ for $M_i\in P_f(R)$, we have
    \begin{align*}
        \widehat{M}&\cong\prod_I M_i\\
        &\cong\prod_I (M_i^\vee)^\vee\\
        &\cong(\bigoplus_I M_i^\vee)^\vee
    \end{align*}
    with $\bigoplus_I M_i^\vee\in P(R^\circ)$.
\end{proof}

\section{Tate Objects}\label{sec:tate}
We are now ready to introduce the category $\elTate(\Cc)$ of \emph{elementary Tate objects} and its idempotent completion $\Tate(\Cc)$. Elementary Tate objects sit in relation to objects of $\Cc$ as the topological $R$-module $R((t))$ sits in relation to finitely generated free $R$-modules.

\subsection{The Category of Elementary Tate Objects}
\begin{definition}
    Let $\Cc$ be an exact category and let $\kappa$ be an infinite cardinal. An \emph{admissible Ind-Pro object in $\Cc$ of size at most $\kappa$} is an object in the category $\Indk(\Prok(\Cc))$.
\end{definition}

\begin{definition}\label{defi:eltate}
    Let $\Cc$ be an exact category. An \emph{elementary Tate diagram in $\Cc$ of size at most $\kappa$} is an admissible Ind-diagram
    \begin{equation*}
        \begin{xy}
            \morphism[I`\Prok(\Cc);X]
        \end{xy}
    \end{equation*}
    of cardinality at most $\kappa$ such that, for all $i\le i'$ in $I$, the object $X_{i'}/X_i$ is in $\Cc$. Denote by $\mathbb{T}_{\kappa}(\Cc)\subset\Dirk(\Prok(\Cc))$ the category of elementary Tate diagrams in $\Cc$ of size at most $\kappa$.
\end{definition}

By definition, we have a canonical functor
\begin{equation}\label{tatecolim}
    \widehat{(-)}\colon\mathbb{T}_\kappa(\Cc)\to\Indk(\Prok(\Cc)).
\end{equation}

\begin{definition}
    Define the category $\elTatek(\Cc)$ of \emph{elementary Tate objects in $\Cc$ of size at most $\kappa$} to be the full sub-category of $\Indk(\Prok(\Cc))$ consisting of objects in the essential image of \eqref{tatecolim}. Denote by $\elTate(\Cc)$ the analogous full sub-category of $\Ind(\Pro(\Cc))$.
\end{definition}

\begin{theorem}\label{thm:eltatechar}
    Let $\Cc$ be an exact category. Let $\kappa$ be an infinite cardinal. The category $\elTatek(\Cc)$ is the smallest full sub-category of $\Indk(\Prok(\Cc))$ which
    \begin{enumerate}
        \item contains the sub-category $\Indk(\Cc)\subset\Indk(\Prok(\Cc))$,
        \item contains the sub-category $\Prok(\Cc)\subset\Indk(\Prok(\Cc))$, and
        \item is closed under extensions.
    \end{enumerate}
    In particular, $\elTatek(\Cc)$ admits a canonical structure as an exact category.
\end{theorem}

The theorem allows us to quickly produce examples of elementary Tate objects.
\begin{example}
    Let $X$ be an integral curve over a field $k$. Denote the set of closed points by $|X|$.  For each closed point $x\in |X|$, let $\Oc_{X,x}$ denote the local ring at $x$, let $\widehat{\Oc_{X,x}}$ denote its completion with respect to the maximal ideal, and let $\Frac(\widehat{\Oc_{X,x}})$ denote the field of fractions of the completed local ring. The \emph{ring of ad\`{e}les} $\Ab(X)$ is the restricted product
    \begin{equation*}
        \Ab(X):=\sideset{}{'}\prod_{x\in |X|}\Frac(\widehat{\Oc_{X,x}})
    \end{equation*}
    where, for any $f\in\Ab(X)$, the factor $f(x)$ lies in $\widehat{\Oc_{X,x}}$ for all but finitely many $x\in |X|$.

    If the set of closed points of $X$ has cardinality $\kappa$, then $\Ab(X)$ is an elementary Tate vector space over $k$ of size $\kappa$. Indeed, $\Ab(X)$ is isomorphic as a $k$-vector space to the direct sum
    \begin{align*}
        \prod_{x\in |X|}\widehat{\Oc_{X,x}}\oplus\bigoplus_{x\in |X|}\Frac(\widehat{\Oc_{X,x}})/\widehat{\Oc_{X,x}}
    \end{align*}
    The product $\prod_{x\in |X|}\widehat{\Oc_{X,x}}$ is an admissible Pro-vector space of size $\kappa$, while the coproduct $\bigoplus_{x\in |X|}\Frac(\widehat{\Oc_{X,x}})/\widehat{\Oc_{X,x}}$ is an admissible Ind-vector space of size $\kappa$. The category $\elTatek(\Vect_k)$ is closed under extensions in $\Indk(\Prok(\Vect_k))$, so $\Ab(X)$ is an object in $\elTatek(\Vect_k)$.
\end{example}

\begin{proof}[Proof of Theorem \ref{thm:eltatechar}]
    The definition of an elementary Tate object immediately implies that the embeddings $\Indk(\Cc)\into\Indk(\Prok(\Cc))$ and $\Prok(\Cc)\into\Indk(\Prok(\Cc))$ factor through the inclusion $\elTatek(\Cc)\subset\Indk(\Prok(\Cc))$. We show that the sub-category $\elTatek(\Cc)$ is closed under extensions, and that every elementary Tate object arises as an extension of an admissible Ind-object by an admissible Pro-object.

    Let
    \begin{equation}\label{eq:pftate}
        \begin{xy}
            \morphism/^{ (}->/[\widehat{X}`F;]
            \morphism(500,0)/->>/[F`\widehat{Z};]
        \end{xy}
    \end{equation}
    be an exact sequence of admissible Ind-Pro objects such that $\widehat{X}$ and $\widehat{Z}$ are elementary Tate objects.

    Observe that for any elementary Tate diagram
    \begin{equation*}
        \begin{xy}
            \morphism[I`\Prok(\Cc);W]
        \end{xy}
    \end{equation*}
    and any final map $J\to I$, the restriction of $W$ to $J$ is also an elementary Tate diagram. Lemma \ref{lemma:indleftspecial} and the proof of Theorem \ref{thm:indleftspecial} imply that we can lift the short exact sequence \ref{eq:pftate} to a sequence of admissible Ind-diagrams of admissible Pro-objects
    \begin{equation*}
        \begin{xy}
            \Vtrianglepair[I`I`I`\Cc;``X`F`Z]
            \place(375,250)[\twoar(1,0)]
            \place(375,350)[_{\alpha}]
            \place(650,250)[\twoar(1,0)]
            \place(650,350)[_{\beta}]
        \end{xy}
    \end{equation*}
    such that $X$ and $Z$ are elementary Tate diagrams of size at most $\kappa$, such that the components of $\beta$ are admissible epics, and such that the components of $\alpha$ are admissible monics.

    For each $i\le j$ in $I$, we have a commuting diagram of admissible Pro-objects with exact rows
    \begin{equation*}
        \begin{xy}
            \square/^{ (}->`^{ (}->`^{ (}->`^{ (}->/[X_i`F_i`X_{j}`F_{j};```]
            \square(500,0)/->>`^{ (}->`^{ (}->`->>/[F_i`Z_i`F_{j}`Z_{j};```]
        \end{xy}
    \end{equation*}
    All vertical maps are admissible monics in $\Prok(\Cc)$. The $(3\times 3)$-Lemma \cite[Corollary 3.6]{Buh:10} shows that taking the cokernels of the vertical maps gives an exact sequence of admissible Pro-objects
    \begin{equation*}
        \begin{xy}
            \morphism/^{ (}->/[X_j/X_i`F_j/F_i;]
            \morphism(500,0)/->>/[F_j/F_i`Z_j/Z_i;]
        \end{xy}
    \end{equation*}
    The first and last terms are in $\Cc$, and $\Cc$ is closed under extensions in $\Prok(\Cc)$ (Theorem \ref{thm:prox}). We conclude that $F$ is an elementary Tate diagram.

    It remains to show that every elementary Tate object is an extension of an admissible Ind-object by an admissible Pro-object. Lemma \ref{lemma:indsubob} shows that, given an elementary Tate diagram $X\colon I\to\Prok(\Cc)$, for any $i\in I$, the map $X_i\into\widehat{X}$ is an admissible monic. The sub-object $X_i$ is an admissible Pro-object, and, because $X_j/X_i$ is in $\Cc$ for all $j\ge i$, the quotient $\widehat{X}/X_i$ is an admissible Ind-object.
\end{proof}

\subsection{Properties of Elementary Tate Objects}
The properties of admissible Ind-objects established in Section \ref{sec:indprop} have their counterparts for elementary Tate objects. We develop these here.

\subsubsection{\texorpdfstring{$\Prok(\Cc)$}{Pro(C)}, \texorpdfstring{$\Indk(\Cc)$}{Ind(C)} and \texorpdfstring{$\Cc$}{C} as Exact, Full Sub-Categories of \texorpdfstring{$\elTatek(\Cc)$}{Tate-el(C)}}
\begin{proposition}\label{prop:prointatesfilt}
    The sub-category $\Prok(\Cc)\subset\elTatek(\Cc)$ is left s-filtering.
\end{proposition}
\begin{proof}
    The embedding $\Prok(\Cc)\into\Indk(\Prok(\Cc))$ is left s-filtering by Proposition \ref{prop:cleftsfiltinind}. Elementary Tate objects form a full sub-category of $\Indk(\Prok(\Cc))$, so the result follows from Lemma \ref{lemma:leftstransitive}.
\end{proof}

\begin{proposition}\label{prop:proandindintate}
    $\Cc\simeq\Ind(\Cc)\cap\Pro(\Cc)\subset\elTate(\Cc)$.
\end{proposition}
\begin{proof}
    Let $\widehat{X}\in\elTate(\Cc)$ be an object which is both an admissible Ind-object and an admissible Pro-object. Let $X\colon I\to\Cc$ be an admissible Ind-diagram representing $\widehat{X}$. Because $\widehat{X}$ is also an admissible Pro-object, the isomorphism $\widehat{X}\cong \colim_I X_i$ factors through the inclusion of $X_i$ for some $i$
    \begin{equation*}
        \begin{xy}
            \qtriangle/>`>`<-_{) }/<750,500>[\widehat{X}`\colim_I X_i`X_i;\cong``]
        \end{xy}
    \end{equation*}
    The inclusion is therefore an epic admissible monic, i.e. an isomorphism. We conclude that $\widehat{X}$ is in $\Cc$.
\end{proof}

\begin{proposition}\label{prop:indintate}\mbox{}
    \begin{enumerate}
        \item For any exact sequence in $\Indk(\Prok(\Cc))$
            \begin{equation*}
                \widehat{X}\into\widehat{Y}\onto\widehat{Z},
            \end{equation*}
            $\widehat{Y}$ is in $\Indk(\Cc)$ if and only if $\widehat{X}$ and $\widehat{Z}$ are in $\Indk(\Cc)$.
        \item The sub-category $\Indk(\Cc)\subset\elTatek(\Cc)$ is right filtering.
    \end{enumerate}
\end{proposition}
\begin{proof}
    Because $\Cc\subset\Prok(\Cc)$ is right special (Theorem \ref{thm:prox}), it is closed under extensions (Lemma \ref{lemma:lsextclosed}). Accordingly, by straightening short exact sequences in $\Indk(\Prok(\Cc))$ (Proposition \ref{prop:inde=eind}), we see that the central term is an admissible Ind-object in $\Cc$ if the outer two terms are.

    We now show the converse. Let
    \begin{equation*}
        \begin{xy}
            \Vtrianglepair/=`=`>`>`>/[I`I`I`\Prok(\Cc);``X`Y`Z]
            \place(375,250)[\twoar(1,0)]
            \place(375,350)[_{\alpha}]
            \place(650,250)[\twoar(1,0)]
            \place(650,350)[_{\beta}]
        \end{xy}
    \end{equation*}
    be an admissible diagram of exact sequences in $\Prok(\Cc)$, and suppose $\widehat{Y}\in\Ind(\Cc)$. We first show that $Y$ factors through $\Cc$. Let $Y'\colon J\to \Cc$ be an admissible Ind-diagram in $\Cc$ representing $\widehat{Y}$. For each $i\in I$, the admissible monic $Y_i\into\widehat{Y}$ factors through a map $Y_i\to Y'_j$ for some $j$; this factoring shows that the map $Y_i\to Y'_j$ is a monic. Because $\Cc$ is right filtering in $\Prok(\Cc)$, this map factors through an admissible epic $Y_i\onto W$ with $W\in\Cc$. Because the map $Y_i\to Y'_j$ is monic, the admissible epic $Y_i\onto W$ is also monic.  It is therefore an isomorphism, and we see that $Y\colon I\to\Cc$ is an admissible Ind-diagram in $\Cc$. Because $\Cc$ is right s-filtering in $\Prok(\Cc)$ (Theorem \ref{thm:prox}), that $Y_i\in\Cc$ for all $i$ combines with Proposition \ref{prop:buh} to imply that $X_i$ and $Z_i$ are in $\Cc$ for all $i$ as well. We conclude that $\widehat{X}$ and $\widehat{Z}$ are in $\Indk(\Cc)$ if $\widehat{Y}$ is.

    We now show that $\Indk(\Cc)$ is right filtering in $\elTatek(\Cc)$. We must show that, for any $f\colon\widehat{X}\to\widehat{Y}$ in $\elTatek(\Cc)$ with $\widehat{Y}\in\Indk(\Cc)$, there exists $\widehat{Z}\in\Indk(\Cc)$ such that $f$ factors through an admissible epic $\widehat{X}\onto\widehat{Z}$.

    Given $f$, let
    \begin{equation*}
        \widehat{L}\into\widehat{X}\onto{X}/\widehat{L}
    \end{equation*}
    be any exact sequence with $\widehat{L}\in\Prok(\Cc)$ and $\widehat{X}/\widehat{L}\in\Indk(\Cc)$ (the proof of Theorem \ref{thm:eltatechar} shows this exists). Note that we assume no relation between $\widehat{L}$ and $f$. Because $\Prok(\Cc)\subset\elTatek(\Cc)$ is left filtering, the composite
    \begin{equation*}
        \widehat{L}\into\widehat{X}\to^f\widehat{Y}
    \end{equation*}
    factors through an admissible monic $\widehat{P}\into\widehat{Y}$ with $\widehat{P}\in\Prok(\Cc)$. As we observed above, because $\widehat{P}$ is an admissible sub-object of the admissible Ind-object $\widehat{Y}$, $\widehat{P}$ is also in $\Indk(\Cc)$. By Proposition \ref{prop:proandindintate}, we conclude that $\widehat{P}\in\Cc$.

    Because $\Cc\subset\Prok(\Cc)$ is right s-filtering, the map $\widehat{L}\to\widehat{P}$ factors through an admissible epic $\widehat{L}\onto P'$ in $\Prok(\Cc)$ with $P'\in\Cc$. Let $\widehat{L}':=\ker(\widehat{L}\onto P')$.  Noether's Lemma \cite[Lemma 3.5]{Buh:10} guarantees that the sequence
    \begin{equation*}
        \widehat{L}'\into\widehat{L}\into\widehat{X}
    \end{equation*}
     of admissible monics in $\elTatek(\Cc)$ rise to an exact sequence
    \begin{equation*}
        P'\cong \widehat{L}/\widehat{L'} \into \widehat{X}/\widehat{L}'\onto \widehat{X}/\widehat{L}
    \end{equation*}
    Because the outer terms are in $\Indk(\Cc)$, we conclude that $\widehat{X}/\widehat{L}'\in\Indk(\Cc)$. By the universal property of cokernels, our construction implies that $f\colon\widehat{X}\to\widehat{Y}$ factors through the admissible epic $\widehat{X}\onto\widehat{X}/\widehat{L}'$.
\end{proof}

\subsubsection{Exact Sequences of Elementary Tate Objects}
The proof of Proposition \ref{prop:inde=eind} implies the following.
\begin{proposition}
    The exact category $\elTatek(\Ec\Cc)$ is canonically equivalent to $\Ec\elTatek(\Cc)$.
\end{proposition}

\subsubsection{Elementary Tate Objects and the S-Construction}
The proof of Proposition \ref{prop:indsk=skind} implies the following.
\begin{proposition}
    For $k\ge 0$, the exact category $\elTatek(S_k\Cc)$ is canonically equivalent to $S_k\elTatek(\Cc)$.
\end{proposition}

\subsubsection{Elementary Tate Objects as a Localization}
 The proof of Proposition \ref{prop:indloc} implies the following.
\begin{proposition}[See also \cite{Bei:87}, \cite{Pre:11}]\label{prop:tateloc}
    Denote by $W\subset\mathbb{T}_\kappa(\Cc)$ the sub-category consisting of all morphisms of elementary Tate diagrams given by strictly commuting triangles
    \begin{equation*}
        \begin{xy}
            \Vtriangle<350,400>[I`J`\Cc;\varphi`X`Y]
        \end{xy}
    \end{equation*}
    in which the map $\varphi$ is final. The functor $\widehat{(-)}\colon\mathbb{T}_\kappa(\Cc)\to\elTatek(\Cc)$ takes morphisms in $W$ to isomorphisms of elementary Tate objects. The induced functor
    \begin{equation*}
        \mathbb{T}_\kappa(\Cc)[W^{-1}]\to\elTatek(\Cc)
    \end{equation*}
    is an equivalence of categories.
\end{proposition}

We present a slight modification of this for later use.
\begin{proposition}\label{prop:basedtateloc}
    Denote by $\mathbb{T}'_\kappa(\Cc)\subset\mathbb{T}_\kappa(\Cc)$ the full sub-category of \emph{based elementary Tate diagrams}, i.e. elementary Tate diagrams $X\colon I\to\Prok(\Cc)$ for which $I$ has an initial object.\footnote{We emphasize that we do not require that maps of based Tate diagrams map initial objects to initial objects.} Define $W'\subset\mathbb{T}'_\kappa(\Cc)$ to be the sub-category of final maps (i.e. $W':=W\cap\mathbb{T}'_\kappa(\Cc)$).

    The restriction of $\widehat{(-)}$ to $\mathbb{T}'_\kappa(\Cc)$ induces an equivalence of categories
    \begin{equation}\label{basedtateloc}
        \mathbb{T}'_\kappa(\Cc)[W'^{-1}]\to^\simeq\elTatek(\Cc)
    \end{equation}
\end{proposition}
\begin{proof}
    The proof follows from Proposition \ref{prop:indloc} with minor changes. Given an elementary Tate diagram $X\colon I\to\Prok(\Cc)$, and $i\in I$, consider the sub-poset $I_i\subset I$ of all $j\geq i$ in $I$. The inclusion $I_i\into I$ is final. As a result, the diagram $X\colon I_i\to\Prok(\Cc)$ in $\mathbb{T}'_{\kappa}(\Cc)$ also represents $\widehat{X}$, and \eqref{basedtateloc} is essentially surjective.

    Fullness follows by a slight modification of the straightening construction (Lemma \ref{lemma:straight}). Let $X\colon I\to\Prok(\Cc)$ and $Y\colon J\to\Prok(\Cc)$ be based elementary Tate diagrams. Lemma \ref{lemma:straight} shows that any morphism $f\colon\widehat{X}\to\widehat{Y}$ in $\elTatek(\Cc)$ is the colimit of a span
    \begin{equation*}
        \begin{xy}
            \Vtrianglepair/<-`>`>`>`>/[I`I\downarrow_{\Prok(\Cc)}J`J`\Prok(\Cc);``X``Y]
            \place(650,250)[\twoar(1,0)]
        \end{xy}
    \end{equation*}
    where the maps $I\downarrow_{\Prok(\Cc)}J\to I$ and $I\downarrow_{\Prok(\Cc)}J\to J$ are final. For any $(i,j,\gamma_{ij})\in I\downarrow_{\Prok(\Cc)}J$, the sub-poset $(I\downarrow_{\Prok(\Cc)}J)_{(i,j,\gamma_{ij})}\subset I\downarrow_{\Prok(\Cc)}J$ of all elements $(l,k,\gamma_{lk})\geq (i,j,\gamma_{ij})$ is directed, final, and has an initial object. We see that the map $f$ is the image of the morphism
    \begin{equation*}
        \begin{xy}
            \Vtrianglepair/<-`>`>`>`>/<650,500>[I`(I\downarrow_{\Prok(\Cc)}J)_{(i,j,\gamma_{ij})}`J`\Prok(\Cc);``X``Y]
            \place(800,250)[\twoar(1,0)]
        \end{xy}
    \end{equation*}
    in $\mathbb{T}'_{\kappa}(\Cc)[W^{-1}]$.

    Faithfulness follows by a slight modification of the argument for Proposition \ref{prop:indloc}. Suppose that $X\colon I\to\Prok(\Cc)$ and $Y\colon J\to\Prok(\Cc)$ are based elementary Tate diagrams for which there exist morphisms
    \begin{equation*}
        \begin{xy}
            \morphism/{@<3pt>}/[X`Y;(\varphi_0,\alpha_0)]
            \morphism|b|/{@<-3pt>}/[X`Y;(\varphi_1,\alpha_1)]
        \end{xy}
    \end{equation*}
    which induce equal maps of elementary Tate objects. For $a=0,1$, the pair $(\varphi_a,\alpha_a)$ induces a section of the map $I\downarrow_{\Prok(\Cc)}J\to I$. Denote by $i_0\in I$ the initial object, and denote by $K_a\subset I\downarrow_{\Prok(\Cc)}J$ the final sub-category consisting of all $(i,j,\gamma_{ij})\ge(i_0,\varphi_a(i_0),\alpha_{a,i_0})$.

    These sections fit into commuting triangles
    \begin{equation*}
        \begin{xy}
            \Vtriangle/<-`>`>/<350,400>[I`K_a`\Prok(\Cc);`X`]
            \morphism(0,400)|a|/{@<5pt>}/<700,0>[I`K_a;(\varphi_a,\alpha_a)]
        \end{xy}
    \end{equation*}
    The existence of these commuting triangles implies that, for $a=0,1$, the image of the map $(\varphi_a,\alpha_a)\colon X\to Y$ in the localization $\mathbb{T}'_{\kappa}(\Cc)[W^{-1}]$ is equal to the map represented by the zig-zag
    \begin{equation*}
        \begin{xy}
            \Vtrianglepair/<-`>`>`>`>/[I`K_a`J`\Prok(\Cc);``X``Y]
            \place(650,250)[\twoar(1,0)]
        \end{xy}
    \end{equation*}
    Let $b\in I\downarrow_{\Prok(\Cc)}J$ be an element with $b\ge(i_0,\varphi_a(i_0),\alpha_{a,i_0})$ for $a=0,1$. Denote by $K\subset I\downarrow_{\Prok(\Cc)}J$ the final sub-category on all $(i,j,\gamma_{ij})\ge b$. For $a=0,1$, we have $K\subset K_a$ with $K$ final. This implies that the maps represented by the zig-zags above are isomorphic to the zig-zag
    \begin{equation*}
        \begin{xy}
            \Vtrianglepair/<-`>`>`>`>/[I`K`J`\Prok(\Cc);``X``Y]
            \place(650,250)[\twoar(1,0)]
        \end{xy}
    \end{equation*}
    We conclude that \eqref{basedtateloc} is faithful.
\end{proof}

\subsubsection{Functoriality of the Construction}
\begin{proposition}\label{prop:tatefun}
    An exact functor $F\colon\Cc\to\Dc$ extends canonically to an exact functor
    \begin{equation*}
        \begin{xy}
            \morphism<750,0>[\elTatek(\Cc)`\elTatek(\Dc);\widetilde{F}]
        \end{xy}.
    \end{equation*}
    If $F$ is faithful, fully faithful, or an equivalence, then so is $\widetilde{F}$.
\end{proposition}
\begin{proof}
    Proposition \ref{prop:FtoIndF} and the analogous clause in Theorem \ref{thm:prox} show that $F$ extends canonically to an exact functor $\widetilde{F}\colon\Indk(\Prok(\Cc))\to\Indk(\Prok(\Dc)$ such that $\widetilde{F}$ is faithful, fully faithful, or an equivalence if $F$ is. It suffices to show that $\widetilde{F}$ preserves elementary Tate objects.

    Represent $\widehat{X}\in\elTatek(\Cc)$ by an elementary Tate diagram of size at most $\kappa$
    \begin{equation*}
        \begin{xy}
            \morphism[I`\Prok(\Cc);X]
        \end{xy}.
    \end{equation*}
    Because $F$ is exact, the diagram
    \begin{equation*}
        \begin{xy}
            \morphism[I`\Prok(\Dc);FX]
        \end{xy}
    \end{equation*}
    is also an elementary Tate diagram. We conclude that $\widetilde{F}(\widehat{X})\in\elTatek(\Dc)$.
\end{proof}

\subsubsection{Countable Elementary Tate Objects}
Countable elementary Tate objects were introduced by Beilinson \cite[A.3]{Bei:87}. We show here that our approach is compatible with his.

\begin{definition}(Beilinson \cite[A.3]{Bei:87})
    Let $\Cc$ be an exact category. Let $\Pi\subset\mathbb{Z}\times\mathbb{Z}$ be the full sub-poset consisting of all $(i,j)$ with $i\leq j$. An \emph{admissible $\Pi$-diagram in $\Cc$} is a functor $X\colon\Pi\to\Cc$ such that for all $i\leq j\leq k$, the sequence
    \begin{equation*}
        \begin{xy}
            \morphism/^{ (}->/[X_{i,j}`X_{i,k};]
            \morphism(500,0)/->>/[X_{i,k}`X_{j,k};]
        \end{xy}
    \end{equation*}
    is short exact. Denote by $\Pi^a(\Cc)$ the category of admissible $\Pi$-diagrams in $\Cc$ and natural transformations between them.
\end{definition}

\begin{definition}
    A functor $\varphi\colon\mathbb{Z}\rightarrow\mathbb{Z}$ is \emph{bifinal} if $\varphi(n)\rightarrow\pm\infty$ as $n\rightarrow\pm\infty$. Let $\varphi_0$ and $\varphi_1$ be two bifinal maps. We say $\varphi_0\leq \varphi_1$ if, for all $n\in\mathbb{Z}$, $\varphi_0(n)\leq\varphi_1(n)$.
\end{definition}

A functor $\varphi\colon\mathbb{Z}\rightarrow\mathbb{Z}$ induces a functor $\varphi\colon\Pi\rightarrow\Pi$ by applying $\varphi$ in each factor.

\begin{definition}
    Denote by $U\subset\Pi^a(\Cc)$ the sub-category consisting of all morphisms of the form $X\varphi_0\rightarrow X\varphi_1$ for bifinal maps $\varphi_0\leq \varphi_1$. The \emph{Beilinson category} $\lim_{\leftrightarrow}\Cc$ is the localization $\Pi^a(\Cc)[U^{-1}]$.
\end{definition}

\begin{proposition}[\rm{Previdi \cite[Theorem 5.8, 6.1]{Pre:11}}]
    Let $\Cc$ be an exact category. The Beilinson category $\lim_{\leftrightarrow}\Cc$ embeds as a full sub-category of $\Indc(\Proc(\Cc))$ which is closed under extensions.
\end{proposition}
\begin{remark}
     Let $X\colon\Pi\to\Cc$ be an admissible $\Pi$-diagram in $\Cc$. The assignment
     \begin{equation*}
         \{X_{i,j}\}\mapsto\colim_j\lim_i X_{i,j}
     \end{equation*}
     extends to a functor $\lim_{\leftrightarrow}\Cc\to\Indc(\Proc(\Cc))$. Previdi \cite{Pre:11} shows that this is a fully faithful embedding into the countable envelope of the dual of the countable envelope of $\Cc^{\op}$ (Previdi denotes this by $\textsf{IP}^a(\Cc)$). Proposition \ref{prop:countableind} and its analogue for countable admissible Pro-objects imply that $\textsf{IP}^a(\Cc)$ is equivalent to the category $\Indc(\Proc(\Cc))$.
\end{remark}

\begin{proposition}\label{prop:eltatec=bei}
    The Beilinson category $\lim_{\leftrightarrow}\Cc$ is equivalent to $\elTatec(\Cc)$ as an exact category.
\end{proposition}
\begin{proof}
    Both are fully exact categories of $\Indc(\Proc(\Cc))$, so it suffices to show that their essential images in Ind-Pro objects agree.

    Let $X\colon\Pi\to\Cc$ be an admissible $\Pi$-diagram in $\Cc$ representing $\widehat{X}\in\lim_{\leftrightarrow}(\Cc)\subset\Indc(\Proc(\Cc))$. The definition of admissible $\Pi$-diagram ensures that the assignment
    \begin{align*}
        n\mapsto X_{-n,0}
    \end{align*}
    defines a countable admissible Pro-diagram $X_{*,0}\colon\mathbb{N}\to\Cc^{\op}$. Denote by $\widehat{X}_0$ the associated admissible Pro-object. The canonical map $\widehat{X}_0\into\widehat{X}$ is an admissible monic (Lemma \ref{lemma:indsubob}). The quotient $\widehat{X}/\widehat{X}_0$ is an admissible Ind-object. Indeed, the quotient is represented by the admissible $\Pi$-diagram $X/X_0\colon\Pi\to\Cc$ which, for $j<0$, sends $(i,j)$ to $0\in\Cc$ and, for $j\geq 0$, sends $(i,j)$ to $X_{i,j}/X_{i,0}$. For any $i\leq 0\leq j$, we have that $X_{i,j}/X_{i,0}\cong X_{0,j}$ because $X$ is an admissible $\Pi$-diagram. In particular, we see that $X/X_0$ is constant in the Pro-direction (the first factor of $\Pi$). We conclude that $\widehat{X}/\widehat{X}_0\in\Indc(\Cc)$. Using Theorem \ref{thm:eltatechar}, we conclude that $\widehat{X}\in\elTatec(\Cc)$.

    Conversely, every countable elementary Tate object is an extension of a countable admissible Ind-object by a countable admissible Pro-object. Proposition \ref{prop:countableind} shows that the categories $\Indc(\Cc)$ and $\Proc(\Cc)$ are contained in $\lim_{\leftrightarrow}(\Cc)$. The proof of Previdi's \cite[Theorem 6.1]{Pre:11} shows that $\lim_{\leftrightarrow}(\Cc)$ is closed under extensions in $\Indc(\Proc(\Cc))$. We conclude that every elementary Tate object is isomorphic to an object in $\lim_{\leftrightarrow}(\Cc)$.
\end{proof}

\begin{proposition}\label{prop:tatesplit}
    The category $\elTatec(\Cc)$ is split exact if $\Cc$ is.
\end{proposition}
\begin{proof}
    Proposition \ref{prop:indcsplit} implies that $\Proc(\Cc)$ and $\Indc(\Proc(\Cc))$ are split exact if $\Cc$ is. As a fully exact sub-category of $\Indc(\Proc(\Cc))$, $\elTatec(\Cc)$ is split exact as well.
\end{proof}

\begin{proposition}\label{prop:bigtateobject}
    Let $\Cc$ be a split exact category for which there exists a collection of objects $\{S_i\}_{i\in\Nb}\subset\Cc$ such that every object $Y\in\Cc$ is a direct summand of $\bigoplus_{i=0}^n S_i$ for some $n$. Denote by $\widehat{\prod_{\Nb}S}$ and $\widehat{\bigoplus_{\Nb} S}$ the admissible Pro and Ind-objects
    \begin{align*}
        \widehat{\prod_{\Nb} S}&:=\prod_{\Nb}(\prod_{i\in\Nb} S_i)\text{, and}\\
        \widehat{\bigoplus_{\Nb} S}&:=\bigoplus_{\Nb}(\bigoplus_{i\in\Nb} S_i).
    \end{align*}
    Then every countable elementary Tate object in $\Cc$ is a direct summand of
    \begin{equation*}
        \widehat{\prod_{\Nb}S}\oplus\widehat{\bigoplus_{\Nb} S}.
    \end{equation*}
\end{proposition}
\begin{proof}
    The proof of Theorem \ref{thm:eltatechar} shows that every elementary Tate object $\widehat{X}$ fits into an exact sequence
    \begin{equation*}
        \begin{xy}
            \morphism/^{ (}->/[\widehat{L}`\widehat{X};]
            \morphism(500,0)/->>/[\widehat{X}`\widehat{X/L};]
        \end{xy}
    \end{equation*}
    where $\widehat{L}\in\Proc(\Cc)$ and $\widehat{X/L}\in\Indc(\Cc)$. Proposition \ref{prop:tatesplit} shows that $\widehat{X}\cong\widehat{L}\oplus\widehat{X/L}$. By Proposition \ref{prop:bigindobject}, $\widehat{X/L}$ is a direct summand of $\widehat{\bigoplus_{\Nb} S}$. The analogous result for Pro-objects (in Theorem \ref{thm:prox}) shows that $\widehat{L}$ is a direct summand of $\widehat{\prod_{\Nb} S}$.
\end{proof}

\begin{example}
    Let $R$ be a ring. The category $\elTatec(P_f(R))$ is split exact, and every countable elementary Tate module is a direct summand of $R((t))$.
\end{example}

\subsection{The Category of Tate Objects}
\begin{definition}
    Let $\Cc$ be an idempotent complete exact category. Define the category $\Tatek(\Cc)$ of \emph{Tate objects in $\Cc$ of size at most $\kappa$} to be the idempotent completion of $\elTatek(\Cc)$.
\end{definition}

The discussion of Section \ref{sec:notic} shows that $\elTatek(\Cc)$ does not coincide with $\Tatek(\Cc)$ in general. Drinfeld \cite[Example 3.2.2.2]{Dri:06} provides another example of a Tate module over a commutative ring which is not elementary.

An analogue of Proposition \ref{prop:proandindintate} holds.
\begin{proposition}\label{prop:fmandfmveeintate}
    Let $\Cc$ be idempotent complete. Then
    \begin{equation*}
        \Cc\simeq\FM(\Cc)\cap\FM^\vee(\Cc)\subset\Tate(\Cc).
    \end{equation*}
\end{proposition}
\begin{proof}
    Given $X\in\FM(\Cc)\cap\FM^\vee(\Cc)$, there exist objects $P,Q\in\Tate(\Cc)$ such that $X\oplus P\in\Pro(\Cc)$ and $X\oplus Q\in\Ind(\Cc)$.

    The composition
    \begin{equation}\label{eq:pfgrproj}
        \begin{xy}
            \morphism/->>/<500,0>[X\oplus P`X;]
            \morphism(500,0)/^{ (}->/<500,0>[X`X\oplus Q;]
        \end{xy}
    \end{equation}
    is a morphism of elementary Tate objects, because the embedding of an exact category into its idempotent completion is fully faithful. Let $Y\colon I\to \Cc$ be an admissible Ind-diagram representing $X\oplus Q$. Because $X\oplus P\in\Pro(\Cc)$, there exists $i\in I$ for which the map \ref{eq:pfgrproj} factors as
    \begin{equation*}
        \begin{xy}
            \Vtrianglepair/->>`^{ (}->`>`-->`<-_{) }/<750,500>[X\oplus P`X`X\oplus Q`Y_i;``f`\exists\tilde{f}`]
        \end{xy}.
    \end{equation*}
    Note that the map $\tilde{f}$ is induced from the map $f$ by the universal property of cokernels. Because $X$ is a retract of $X\oplus Q$ in $\Ind(\Cc)$, the object $X$ is also a retract of $Y_i$ in $\Cc$. The composite
    \begin{equation*}
        Y_i\to X\to Y_i
    \end{equation*}
    is an idempotent. If $\Cc$ is idempotent complete, this idempotent splits, and we conclude that $X$ is an object of $\Cc$.
\end{proof}

\subsection{Tate R-Modules}
The category of Tate objects in finite dimensional vector spaces over a discrete field dates back at least to Lefschetz \cite[Chapter II.25]{Lef:42}. Lefschetz described this as a category of ``locally linearly compact'' topological vector spaces.\footnote{The terminology appears to be inspired by Pontrjagin duality for locally compact abelian groups.} Recently, Drinfeld \cite{Dri:06} asked, ``What is a family of Tate spaces?'' He proposed a notion of Tate modules over discrete rings $R$. We relate his notion to Tate objects in the category $P_f(R)$ of finitely generated projective (left) $R$-modules.

\begin{definition}[Drinfeld]
    Let $R$ be a ring. An \emph{elementary Tate module \textit{\`{a} la} Drinfeld} is a topological $R$-module isomorphic to $P\oplus Q^{\vee}$ where $P$ is a discrete projective (left) module and $Q^\vee$ is the topological dual of a discrete projective (right) module. A \emph{Tate module \textit{\`{a} la} Drinfeld} is a topological direct summand of an elementary Tate module $P\oplus Q^{\vee}$.

    Denote by $\Drate(R)$ the category of Tate modules \textit{\`{a} la} Drinfeld and continuous homomorphisms. Denote by $\Dratec(R)\subset\Drate(R)$ the full sub-category of direct summands of modules $P\oplus Q^{\vee}$ where $P$ and $Q$ are countably generated.
\end{definition}

\begin{theorem}\label{thm:drincomp}
    The category $\Drate(R)$ of Tate modules \textit{\`{a} la} Drinfeld is equivalent to a full sub-category of $\Tate(P_f(R))$. The categories $\Dratec(R)$ and $\Tatec(P_f(R))$ are equivalent.
\end{theorem}
\begin{proof}
    It suffices to show that the category $\elDrate(R)$ of elementary Tate modules \emph{\`{a} la} Drinfeld is equivalent to a full sub-category of $\Tate(P_f(R))$.

    Restricting \eqref{prointop} to $\Pro(P_f(R))$, we obtain a faithful embedding
    \begin{equation*}
        \begin{xy}
            \morphism<1000,0>[\Pro(P_f(R))`\Mod(R)_{top};\tau]
        \end{xy}.
    \end{equation*}
    This naturally extends to a functor
    \begin{equation*}
        \begin{xy}
            \morphism<1000,0>[\Dir(\Pro(P_f(R)))`\Mod(R)_{top};]
            \morphism(0,-200)|l|/>/<0,-400>[I`\Pro(P_f(R));M]
            \morphism(50,-400)/|->/<950,0>[`\colim_I\tau(M_i);]
        \end{xy}
    \end{equation*}
    which induces
    \begin{equation}\label{indprointop}
        \begin{xy}
            \morphism<1000,0>[\FM(\Pro(P_f(R)))`\Mod(R)_{top};\tilde{\tau}]
        \end{xy}.
    \end{equation}
    Unpacking the definition, we see that $\tilde{\tau}$ is faithful. As with $\tau$, $\tilde{\tau}$ will fail, in general, to be full. However this failure does not affect several cases of interest.

    First, for any $M\colon I\to P_f(R)$ representing $\widehat{M}\in\Ind(P_f(R))\subset\Ind(\Pro(P_f(R)))$, and for any $\widehat{N}\in\Ind(\Pro(P_f(R)))$, we have
    \begin{align*}
        \hom_{cts}(\tilde{\tau}(\widehat{M}),\tilde{\tau}(\widehat{N}))&\cong\lim_I\hom_{cts}(M_i,\tilde{\tau}(\widehat{N}))\\
        &\cong\lim_I\hom_R(M_i,\widehat{N})\\
        &\cong\lim_I\colim_J\lim_{K_j}\hom_R(M_i,N_{j,k})\\
        &\cong\hom_{\Ind(\Pro(P_f(R)))}(\widehat{M},\widehat{N}).
    \end{align*}
    where $N\colon J\to\Pro(P_f(R))$ is an admissible Ind-diagram representing $\widehat{N}$, and where, for each $j$, $N_j\colon K_j\to P_f(R)$ is an admissible Pro-diagram representing $N_j$. This implies that for any $\widehat{M}\in\FM(P_f(R))$ and for any $\widehat{N}\in\FM(\Pro(P_f(R))$, we have
    \begin{align*}
        \hom_{cts}(\tilde{\tau}(\widehat{M}),\tilde{\tau}(\widehat{N}))\cong\hom_{\FM(\Pro(P_f(R)))}(\widehat{M},\widehat{N}).
    \end{align*}

    Second, $\tilde{\tau}$ restricts to a full embedding $\tau\colon P^\vee(P_f(R))\to\Mod(R)_{top}$ by Proposition \ref{prop:tprointop}.

    Third, given $\widehat{N}\in\FM(P_f(R))$, $\tilde{\tau}(\widehat{N})$ is discrete. Therefore, the same argument as in the proof of Proposition \ref{prop:tprointop} shows that for $\widehat{M}\in P^\vee(P_f(R))$, we have
    \begin{align*}
        \hom_{cts}(\tilde{\tau}(\widehat{M}),\tilde{\tau}(\widehat{N}))&\cong\hom_{\FM(\Pro(P_f(R)))}(\widehat{M},\widehat{N}).
    \end{align*}

    We can now show that the essential image of $\Tate(P_f(R))$ in $\Mod(R)_{top}$ contains $\elDrate(R)$ as a full sub-category. We first observe that every elementary Tate module \textit{\`{a} la} Drinfeld $P\oplus Q^{\vee}$ is in $\Tate(P_f(R))$. The module $P\oplus Q^{\vee}$ is trivially an extension of the discrete projective (left) module $P$ by the topological dual of a discrete projective (right) module $Q$. Corollaries \ref{cor:projRgeneral} and \ref{cor:projdualinproic} show that $P$ and $Q^{\vee}$ are objects in $\Tate(P_f(R))$, so $P\oplus Q^{\vee}$ is as well.

    Given $M_0\cong P_0\oplus Q_0^\vee$ and $M_1\cong P_1\oplus Q_1^\vee$, we have
    \begin{align*}
        \hom_{\elDrate(R)}(M_0,M_1):=&\hom_{cts}(M_0,M_1)\\
        \cong&\hom_{cts}(P_0,P_1)\times\hom_{cts}(Q_0^\vee,P_1)\\
        &\times\hom_{cts}(P_0,Q_1^\vee)\times\hom_{cts}(Q_0^\vee,Q_1^\vee).
    \end{align*}
    Our observations above imply that each factor is isomorphic to the analogous hom-set in $\Tate(P_f(R))$, and thus
    \begin{equation*}
        \hom_{\elDrate(R)}(M_0,M_1)\cong\hom_{\Tate(P_f(R))}(M_0,M_1)
    \end{equation*}
    as claimed.

    To prove $\Dratec(R)\simeq\Tatec(P_f(R))$, we show that every countable elementary Tate object in $P_f(R)$ is elementary \textit{\`{a} la} Drinfeld. Indeed, every countable elementary Tate module $V\in\elTatec(P_f(R))$ admits a lattice $L$. The lattice $L$ is isomorphic to the topological dual of a discrete, countably generated, projective (right) module by Corollary \ref{cor:projdualinproic}. The quotient $V/L$ is discrete and projective by Corollary \ref{cor:projRgeneral}.
\end{proof}

\begin{remark}\label{rmk:drate}\mbox{}
    \begin{enumerate}
        \item It is possible to give a purely categorical description of $\Drate(R)$ as follows. For idempotent complete, split exact $\Cc$, define
            \begin{equation*}
                \elDrate(\Cc)\subset\Tate(\Cc)
            \end{equation*}
            to be the smallest full sub-category of $\Tate(\Cc)$ which contains the categories $P(\Cc)$ and $P^\vee(\Cc)$ and which is closed under extensions. Denote by $\Drate(\Cc)$ the idempotent completion of $\elDrate(\Cc)$. The discussion above shows that $\Drate(P_f(R))\simeq\Drate(R)$.
        \item Drinfeld \cite[p. 266]{Dri:06} suggests the possibility of a notion of infinite dimensional vector bundle using flat Mittag-Leffler modules in lieu of projective modules. Some consequences and pathologies of this suggestion have been investigated in \cite{EGPT:12} and \cite{EGT:11}.

            If one attempts to use topological language to formulate the analogous notion of Tate module, one encounters related problems. The most serious of these is that if $M$ and $N$ are flat Mittag-Leffler modules, we can say very little about the image of a map $M^\vee\to N$, whereas a key property of the Tate formalism is that any such map should factor through a finitely generated admissible sub-module. A closely related problem is that the topological dual does not preserve exact sequences of flat Mittag-Leffler modules. Both of these are consequences of the non-vanishing of $\lim^1$ for uncountable Mittag-Leffler systems, and we interpret them as a sign that we should abandon the topological framework, and work categorically instead.

            If one is willing to make this switch, then, in light of Proposition \ref{p:more}, Example \ref{ex:Rindprodual} and Theorem \ref{thm:eltatechar}, we propose $\Tate(P_f(R))$ as the category of Tate modules \emph{\`{a} la} Drinfeld modeled on flat Mittag-Leffler modules.
    \end{enumerate}
\end{remark}

\subsection{Tate Objects and the Calkin Category}
Let $R$ be a ring. Denote by $\pi_{R[[t]]}\colon R((t))\to R((t))$ the projection onto $R[[t]]\subset R((t))$. Tate objects provide a categorical analogue of the $R$-algebra
\begin{equation*}
    \Mat^\pm_\infty(R):=\{A\in\End_R(R((t)))~\vert~ [A,\pi_{R[[t]]}]\text{ has finite rank.}\}
\end{equation*}
The assignment $A\mapsto (1-\pi_{R[[t]]})A(1-\pi_{R[[t]]})$ defines a surjective algebra homomorphism
\begin{equation*}
    \Mat^\pm_\infty(R)\to\Calk(R).
\end{equation*}
This homomorphism admits a categorical analogue.

\subsubsection{The Map \texorpdfstring{$\Tatek(\Cc)\rightarrow\Calkk(\Cc)$}{from Tate to Calk}}
Recall from Proposition \ref{prop:basedtateloc} that $\mathbb{T}'_\kappa(\Cc)$ is the category of based elementary Tate diagrams, $X\colon I\to\Prok(\Cc)$. Denote by $i_0\in I$ the initial object. The assignment
\begin{equation*}
    \begin{xy}
        \morphism<0,-500>[I`\Prok(\Cc);X]
        \morphism(50,-250)/|->/<400,0>[`;]
        \morphism(500,0)|r|<0,-500>[I`\Cc;X/X_{i_0}]
    \end{xy}
\end{equation*}
extends to a functor
\begin{equation*}
    \begin{xy}
        \morphism<750,0>[\mathbb{T}'_\kappa(\Cc)`\Calkk(\Cc);q]
    \end{xy}.
\end{equation*}
This functor sends a morphism
\begin{equation*}
    \begin{xy}
        \Vtriangle<350,400>[I`J`\Prok(\Cc);\psi`X`Y]
        \place(350,200)[\twoar(1,0)]
        \place(350,300)[_{\alpha}]
    \end{xy}
\end{equation*}
to the morphism
\begin{equation*}
    \begin{xy}
        \Vtriangle/{}`>`<-/<1000,0>[\colim_I X_i/X_{i_0}`\colim_J Y_j/Y_{j_0}`\colim_I Y_{\psi(i)}/Y_{\psi(i_0)};``\cong]
    \end{xy}.
\end{equation*}
The isomorphism on the right is the inverse, in $\Calkk(\Cc)$, of the map in $\Calkk(\Cc)$ given by the zig-zag
\begin{equation*}
    \begin{xy}
        \Atriangle/>`>`{}/<1000,0>[\colim_I Y_{\psi(i)}/Y_{j_0}`\colim_I Y_{\psi(i)}/Y_{\psi(i_0)}`\colim_J Y_j/Y_{j_0};``]
    \end{xy}.
\end{equation*}
Recall that $W'\subset\mathbb{T}'_\kappa(\Cc)$ denotes the sub-category of final maps (which are not required to preserve the initial object!). By inspection, $q$ takes maps in $W'$ to isomorphisms. Proposition \ref{prop:basedtateloc} and the universal property of localization guarantee that $q$ induces a unique functor
\begin{equation}\label{eltatetocalk}
    \begin{xy}
        \morphism<750,0>[\elTatek(\Cc)`\Calkk(\Cc);]
    \end{xy}.
\end{equation}
Our construction guarantees that \eqref{eltatetocalk} is exact. By the universal property of idempotent completion, \eqref{eltatetocalk} extends uniquely to an exact functor
\begin{equation*}
    \begin{xy}
        \morphism<750,0>[\Tatek(\Cc)`\Calkk(\Cc);\tilde{q}]
    \end{xy}.
\end{equation*}

\subsubsection{The Kernel of \texorpdfstring{$\Tatek(\Cc)\rightarrow\Calkk(\Cc)$}{the map from Tate to Calk}}
If $A\in\Mat^\pm_\infty(R)$ has its image contained in $R[[t]]\oplus M\subset R((t))$ for some finitely generated sub-module $M$, then $A\in\ker(\Mat^\pm_\infty(R)\to\Calk(R))$. Similarly, the construction of $\tilde{q}$ shows that it takes admissible Pro-objects to the zero object in $\Calkk(\Cc)$. As a result, it takes admissible epics (monics) in $\elTatek(\Cc)$ whose (co)kernels are in $\Prok(\Cc)$ to isomorphisms. It therefore factors through an exact functor
\begin{equation}\label{tatetocalk}
    \begin{xy}
        \morphism<1000,0>[(\elTatek(\Cc)/\Prok(\Cc))^{\ic}`\Calkk(\Cc);\tilde{q}]
    \end{xy}.
\end{equation}

\begin{proposition}[\rm{See also Saito \cite[Lemma 3.3]{Sai:13}}]\label{prop:quot}
    The map \eqref{tatetocalk} is an equivalence of exact categories.
\end{proposition}
\begin{proof}
    The universal property of localization and idempotent completion ensure that the map $\Indk(\Cc)\into\Tatek(\Cc)\to(\elTatek(\Cc)/\Prok(\Cc))^{\ic}$ induces a canonical map
    \begin{equation*}
        \begin{xy}
            \morphism<1000,0>[\Calkk(\Cc)`(\elTatek(\Cc)/\Prok(\Cc))^{\ic};\iota]
        \end{xy}
    \end{equation*}
    We will show that $\iota$ is inverse to $\tilde{q}$. From the construction of idempotent completion, it suffices to show that:
    \begin{enumerate}
        \item if $\widehat{Y}\in\Calkk(\Cc)$ is isomorphic to an admissible Ind-object, then $\tilde{q}\iota(\widehat{Y})$ is naturally isomorphic to $\widehat{Y}$, and
        \item if $\widehat{X}\in(\elTatek(\Cc)/\Prok(\Cc))^{\ic}$ is isomorphic to an elementary Tate object, then $\iota\tilde{q}(\widehat{X})$ is naturally isomorphic to $\widehat{X}$.
    \end{enumerate}
    The first is immediate from the construction of $\iota$ and $\tilde{q}$. Conversely, if $\widehat{X}$ is an elementary Tate object, the construction of $\tilde{q}$ defines a map
    \begin{equation*}
        \begin{xy}
            \morphism[\widehat{Y}`\iota\tilde{q}(\widehat{Y});]
        \end{xy}
    \end{equation*}
    whose kernel is an admissible Pro-object.
\end{proof}

Similarly, we have the following.
\begin{proposition}\label{prop:promodc}
    Let $\Cc$ be idempotent complete. The assignment $X\mapsto X_{i_0}$ extends to a functor
    \begin{equation}\label{preTatequot}
        \mathbb{T}'_\kappa(\Cc)\to\Prok(\Cc)/\Cc.
    \end{equation}
    This induces an equivalence of categories
    \begin{equation*}
        \begin{xy}
            \morphism<1000,0>[\elTatek(\Cc)/\Indk(\Cc)`\Prok(\Cc)/\Cc;\simeq]
        \end{xy}
    \end{equation*}
    Its inverse $\Prok(\Cc)/\Cc\into\elTatek(\Cc)/\Indk(\Cc)$ is the map induced by the inclusion $\Prok(\Cc)\subset\elTatek(\Cc)$.
\end{proposition}
\begin{remark}
    The category $\Indk(\Cc)$ is right filtering in $\elTatek(\Cc)$ by Proposition \ref{prop:indintate}. However, we do not know if it is right special. Therefore, Schlichting's theory of quotients of exact categories is not entirely available. Nonetheless, using results from Section \ref{sec:applat}, we are able to show that this quotient is well behaved. While this result logically follows the results of Section \ref{sec:applat}, we include it here to connect it to the discussion of the Calkin category.
\end{remark}
\begin{proof}
    Just as for quotients by right s-filtering sub-categories, we define the quotient $\elTatek(\Cc)/\Indk(\Cc)$ to be the localization of $\elTatek(\Cc)$ at the class of admissible monics with cokernels in $\Indk(\Cc)$.

    By inspection, the functor \eqref{preTatequot} factors through the localization
    \begin{equation*}
        \mathbb{T}'_\kappa(\Cc)\to\elTatek(\Cc)
    \end{equation*}
    of Proposition \ref{prop:basedtateloc}.  The induced functor
    \begin{equation}\label{prelocquot}
        \elTatek(\Cc)\to\Prok(\Cc)/\Cc
    \end{equation}
    sends an elementary Tate object $\widehat{V}$ to any Pro-object $L$ appearing in an admissible elementary Tate diagram representing $\widehat{V}$. We refer to the admissible monic $L\into\widehat{V}$ as a \emph{lattice} of $\widehat{V}$ (cf. Definition \ref{def:lattice} and Remark \ref{rmk:lattice}). We now show that \eqref{prelocquot} factors through $\elTatek(\Cc)/\Indk(\Cc)$.

    Let
    \begin{equation*}
        \widehat{V}_0\into \widehat{V}_1\onto \widehat{X}
    \end{equation*}
    be a short exact sequence of elementary Tate objects with $\widehat{X}\in\Indk(\Cc)$. By the universal property of localizations, it suffices to show that \eqref{prelocquot} sends the map $V_0\hookrightarrow V_1$ to an isomorphism in $\Prok(\Cc)/\Cc$.

    To check this, we let $L_0\into\widehat{V}_0$ be any lattice of $\widehat{V}_0$. By the definition of morphisms in $\elTatek(\Cc)$, the map
    \begin{equation*}
        L_0\into V_0\into V_1
    \end{equation*}
    factors through a lattice $L_1\into \widehat{V}_1$. Therefore, the functor \eqref{prelocquot} sends the map $V_0\into V_1$ in $\elTatek(\Cc)$ to the map $L_0\to L_1$ in $\Prok(\Cc)/\Cc$. We claim that this map is an isomorphism, i.e. that $L_0\to L_1$ is an admissible monic in $\Prok(\Cc)$ with cokernel in $\Cc$.

    Because $\Cc$ is idempotent complete, by Lemma \ref{lemma:grdirkey}, it suffices to show that the admissible monic $L_0\into V_1$ is also a lattice. This follows from Noether's lemma and Proposition \ref{prop:indintate}. Indeed, we have a short exact sequence in $\elTatek(\Cc)$
    \begin{equation*}
        \widehat{V}_0/L_0\into \widehat{V}_1/L_0\onto \widehat{V}_1/\widehat{V}_0.
    \end{equation*}
    By assumption $\widehat{V}_0/L_0$ and $\widehat{V}_1/\widehat{V}_0$ are both in $\Indk(\Cc)$. Therefore $\widehat{V}_1/L_0$ is as well. By straightening the exact sequence
    \begin{equation*}
        L_0\into \widehat{V}_1\onto \widehat{V}_1/L_0
    \end{equation*}
    we see that this implies that $L_0\into \widehat{V}_1$ is a lattice as claimed. We have therefore shown that \eqref{prelocquot} induces a functor
    \begin{equation*}
        \elTatek(\Cc)/\Indk(\Cc)\to\Prok(\Cc)/\Cc.
    \end{equation*}
    From the definitions, we see that this is an inverse to the map 
    \begin{equation*}
        \Prok(\Cc)/\Cc\to\elTatek(\Cc)/\Indk(\Cc)
    \end{equation*}
    as claimed.
\end{proof}

\begin{remark}\label{rmk:KTate}
    As S. Saito \cite{Sai:13} has observed in the countable case, the propositions above combine with the Eilenberg swindle and Schlichting's localization theorem \cite{Sch:04} to show that $K_i(\Tatek(\Cc))\cong K_{i-1}(\Cc)$, when $\Cc$ is idempotent complete.
\end{remark}

\section{Sato Grassmannians}\label{sec:applat}
Let $k$ be a field and consider the Tate vector space $k((t))$. Sato and Sato \cite{SaS:83} introduced an infinite dimensional Grassmannian $\Gr(k((t)))$ whose points correspond to \emph{lattices}, i.e. members of a certain class of subspaces $L\subset k((t))$. They then constructed a \emph{determinant line bundle} $\mathcal{L}\to\Gr(k((t)))$ and employed this to great effect in applications. The key properties required for the construction of the determinant line are:
\begin{enumerate}
    \item for any nested pair of lattices $L_0\subset L_1\subset k((t))$, the quotient $L_1/L_0$ is finite dimensional, and
    \item for any pair of lattices $L_0$ and $L_1$, there exists a common enveloping lattice $N$ with $L_i\subset N$ for $i=0,1$.
\end{enumerate}

In this section, we recall the definition of lattices and Sato Grassmannians for elementary Tate objects in general exact categories (see also Previdi \cite{Pre:12}). We show that, for any $\Cc$, the construction of the Grassmannian is natural with respect to exact functors, and that the analogue of the first property holds. If $\Cc$ is idempotent complete, we show that the analogue of the second holds as well.

\begin{definition}\label{def:lattice}
    Let $V\in\elTatek(\Cc)$.
    \begin{enumerate}
        \item A \emph{lattice} $L$ of $V$ consists of an admissible monic
            \begin{equation*}
                \begin{xy}
                    \morphism/^{ (}->/[L`V;]
                \end{xy}
            \end{equation*}
            such that $L\in\Prok(\Cc)$ and $V/L\in\Indk(\Cc)$.
        \item A \emph{co-lattice} $L^\perp$ of $V$ consists of an admissible epic
            \begin{equation*}
                \begin{xy}
                    \morphism/->>/[V`L^\perp;q]
                \end{xy}
            \end{equation*}
            such that $L^\perp\in\Indk(\Cc)$ and $\ker(q)\in\Prok(\Cc)$.
    \end{enumerate}
\end{definition}

\begin{remark}\label{rmk:lattice}\mbox{}
    \begin{enumerate}
        \item We use the term ``lattice'' to refer to what Drinfeld \cite{Dri:06} calls a ``co-projective lattice.''
        \item The proof of Theorem \ref{thm:eltatechar} shows that any object in an elementary Tate diagram is a lattice of the associated Tate object.
    \end{enumerate}
\end{remark}

\begin{definition}
    Let $V$ be an elementary Tate object in $\Cc$. The \emph{Sato Grassmannian} $\Gr(V)$ is the poset of lattices of $V$, where
    \begin{equation*}
        (L_0\into V)\le (L_1\into V)
    \end{equation*}
    if and only if there exists a commuting triangle
    \begin{equation*}
        \begin{xy}
            \Vtriangle/^{ (}->`^{ (}->`^{ (}->/<350,400>[L_0`L_1`V;f``]
        \end{xy}
    \end{equation*}
    in $\elTatek(\Cc)$ with $f$ an admissible monic.
\end{definition}

\begin{proposition}\label{prop:grfun}
    Let $V$ be an elementary Tate object in $\Cc$. An exact functor $F\colon\Cc\to\Dc$ induces an order preserving map
    \begin{equation*}
        \begin{xy}
            \morphism<750,0>[\Gr(V)`\Gr(\widetilde{F}(V));F_V]
            \morphism(0,-200)/|->/<750,0>[L`\widetilde{F}(L);]
        \end{xy}.
    \end{equation*}
\end{proposition}
\begin{proof}
    Because $F$ extends to an exact functor $\widetilde{F}\colon\elTatek(\Cc)\to\elTatek(\Dc)$ (Proposition \ref{prop:tatefun}), order preserving will follow once we show that $F_V$ is well defined. Let
    \begin{equation*}
        \begin{xy}
            \morphism/^{ (}->/[L`V;]
        \end{xy}
    \end{equation*}
    be a lattice in $V$. It defines an exact sequence
    \begin{equation*}
        \begin{xy}
            \morphism/^{ (}->/[L`V;]
            \morphism(500,0)/->>/[V`V/L;]
        \end{xy}
    \end{equation*}
    of elementary Tate objects in $\Cc$, where $V/L$ is an admissible Ind-object. The functor $\widetilde{F}$ preserves exact sequences as well as admissible Pro and Ind-objects, so
    \begin{equation*}
        \begin{xy}
            \morphism/^{ (}->/[\widetilde{F}(L)`\widetilde{F}(V);]
            \morphism(500,0)/->>/[\widetilde{F}(V)`\widetilde{F}(V/L);]
        \end{xy}
    \end{equation*}
    is an exact sequence of elementary Tate objects in $\Dc$, such that $\widetilde{F}(L)$ is an admissible Pro-object, and such that $\widetilde{F}(V/L)$ is an admissible Ind-object.
\end{proof}

\begin{remark}
    Let $R$ be a commutative ring. Let $V$ be an elementary Tate $R$-module. Proposition \ref{prop:grfun} shows that the Sato Grassmannian $\Gr(V)$ defines a presheaf over $\text{Spec}(R)$. The Sato Grassmannian $\Gr(V)$ can in fact be viewed as an Ind-projective Ind-scheme which is Ind-proper over $\text{Spec}(R)$ \cite[Proposition 3.8; Remark b)]{Dri:06}.
\end{remark}

\begin{proposition}\label{prop:latcomm}
    Let $L_0\into L_1\into V$ be a nested pair of lattices in $V\in\elTatek(\Cc)$. The quotient $L_1/L_0$ is an object of $\Cc$.
\end{proposition}
\begin{proof}
    By Proposition \ref{prop:proandindintate}, it suffices to show that $L_1/L_0$ lies in $\Indk(\Cc)\cap\Prok(\Cc)$.

    By Noether's Lemma \cite[Lemma 3.5]{Buh:10}, a nested pair of lattices gives rise to an exact sequence
    \begin{equation*}
        \begin{xy}
            \morphism/^{ (}->/[L_1/L_0`V/L_0;]
            \morphism(500,0)/->>/[V/L_0`V/L_1;]
        \end{xy}.
    \end{equation*}
    The object $L_1/L_0$ is an admissible Pro-object. Indeed, it is an admissible quotient in $\elTatek(\Cc)$ of an admissible Pro-object, and $\Prok(\Cc)\subset\elTatek(\Cc)$ is left s-filtering. Because $V/L_0$ is an admissible Ind-object, Proposition \ref{prop:indintate} implies that $L_1/L_0$ is also an admissible Ind-object.
\end{proof}

\begin{theorem}\label{thm:grdir}
    Let $\Cc$ be idempotent complete. Let $V\in\elTatek(\Cc)$. The Sato Grassmannian $\Gr(V)$ is a directed and co-directed poset.
\end{theorem}
\begin{remark}
    As we remarked in the introduction, we view this as the most important theorem of this paper. It implies that, given a pair of lattices $L_0,~L_1$ in a Tate module $V$, one can define the \emph{index bundle} of $L_0$ and $L_1$ to be the finitely generated $\mathbb{Z}/2$-graded projective $R$-module
    \begin{equation*}
        L_0/N\oplus L_1/N,
    \end{equation*}
    where $N$ is any common sub-lattice. This definition recalls Atiyah's construction of the index of a continuous family of Fredholm operators \cite[Appendix A]{Ati:67}. In \cite{BGW:13}, we show that the assignment of an index bundle to a pair of lattices extends, independent of the choice of sub-lattice, to a natural map from $\Gr(V)\times\Gr(V)$ to the algebraic $K$-theory space of $R$.
\end{remark}

The following lemma contains the core of the proof of the theorem.
\begin{lemma}\label{lemma:grdirkey}
    Let $\Cc$ be idempotent complete.
    \begin{enumerate}
        \item Let $\imath_0\colon L_0\into V$ and $\imath_1\colon L_1\into V$ be lattices of an elementary Tate object $V\in\elTatek(\Cc)$, and let
            \begin{equation}\label{latmon}
                \begin{xy}
                    \Vtriangle/>`^{ (}->`^{ (}->/<350,450>[L_0`L_1`V;f`\imath_0`\imath_1]
                \end{xy}
            \end{equation}
            be a commuting triangle in $\elTatek(\Cc)$. Then $f\colon L_0\to L_1$ is an admissible monic in $\Prok(\Cc)$.
        \item Let $q_0\colon V\to L_0^\perp$ and $q_1\colon V\to L_1^\perp$ be co-lattices of $V$, and let
            \begin{equation}\label{colatep}
                \begin{xy}
                    \Atriangle/->>`->>`>/<350,450>[V`L_0^\perp`L_1^\perp;q_0`q_1`g]
                \end{xy}
            \end{equation}
            be a commuting triangle in $\elTatek(\Cc)$. Then $g\colon L_0\to L_1$ is an admissible epic in $\Indk(\Cc)$.
    \end{enumerate}
\end{lemma}
\begin{remark}\mbox{}
    \begin{enumerate}
        \item When $\elTatek(\Cc)$ is idempotent complete, this lemma is an immediate consequence of \cite[Proposition 7.6]{Buh:10}. However, this is not generally the case.
        \item If $\Cc$ is split exact, the category $\elDrate(\Cc)$ of elementary Tate objects \emph{\`{a} la} Drinfeld (see Remark \ref{rmk:drate}) is split exact as well. This simplifies the proof of the lemma considerably, since for split exact categories, every admissible sub-object or quotient is in fact a direct summand. Hence, once we adapt the definition of lattices and co-lattices to $\elDrate(\Cc)$ in the natural fashion, it immediately follows that $L_1/L_0$ lies in $P^\vee(\Cc)$ and, by Noether's lemma, also in $P(\Cc)$, and similarly for $\ker(g)$.
    \end{enumerate}
\end{remark}
\begin{proof}[Proof of Lemma \ref{lemma:grdirkey}]
    We begin by proving that for any triangle of the form \eqref{latmon}, $f$ is an admissible monic in $\Prok(\Cc)$. Because the category $\Prok(\Cc)\subset\elTatek(\Cc)$ is closed under extensions, it suffices to show that $f$ fits into an exact sequence
    \begin{equation*}
        0\to L_0\to^f L_1\to X\to 0
    \end{equation*}
    in $\elTatek(\Cc)$ with $X\in\Cc$.

    Because $\imath_1f=\imath_0$ is an admissible monic in $\elTatek(\Cc)$, the dual of \cite[Proposition 7.6]{Buh:10} shows that $f$ is an admissible monic in the idempotent complete category $\Tatek(\Cc)$. We can construct the cokernel of $f$ as follows. Consider the pushout square
    \begin{equation*}
        \begin{xy}
            \square/>`^{ (}->`^{ (}->`>/[L_0`L_1`V`V\cup_{L_0} L_1;f`\imath_0``f']
        \end{xy}
    \end{equation*}
    in $\elTatek(\Cc)$. The map $\imath_1$ determines a map
    \begin{equation*}
        1\oplus \imath_1\colon V\cup_{L_0} L_1\to V.
    \end{equation*}
    By construction, $(1\oplus\imath_1)f'=1_V$. Therefore $f'(1\oplus\imath_1)$ gives an idempotent of the elementary Tate object $V\cup_{L_0} L_1$. In $\Tatek(\Cc)$, this idempotent splits, and we obtain a diagram
    \begin{equation*}
        \begin{xy}
            \square/^{ (}->`^{ (}->`^{ (}->`^{ (}->/[L_0`L_1`V`V\cup_{L_0} L_1;f`\imath_0``f']
            \square(500,0)/>`^{ (}->`=`>/[L_1`X`V\cup_{L_0}L_1`X;```]
        \end{xy}
    \end{equation*}
    whose bottom row is split exact and whose left square is a pushout in which all maps are admissible monics. We conclude by \cite[Proposition 2.12]{Buh:10} that $\coker(f)\cong\coker(f')=X$.

    We will now show that $X$ is isomorphic both to the image of an idempotent of an admissible Pro-object, and also to the image of an idempotent of an admissible Ind-object. As a result, $X\in \FM^\vee(\Cc)\cap\FM(\Cc)$. When $\Cc$ is idempotent complete, this shows that $X\in\Cc$ (Proposition \ref{prop:fmandfmveeintate}).

    Denote by $X\to^s V\cup_{L_0}L_1$ the inclusion of the summand, and denote by $q$ the projection back onto $X$. Consider the map of elementary Tate objects $\alpha\colon L_1\to V\cup_{L_0}L_1$ given in $\Tatek(\Cc)$ by the composition
    \begin{equation*}
        L_1\onto X\into^s V\cup_{L_0} L_1.
    \end{equation*}
    (we are using that a category embeds fully faithfully into its idempotent completion). Because $\Prok(\Cc)\subset\elTatek(\Cc)$ is left filtering (Proposition \ref{prop:prointatesfilt}), the map $\alpha$ factors through an admissible monic
    \begin{equation*}
        \begin{xy}
            \Vtriangle/>`>`^{ (}->/<350,400>[L_1`A`V\cup_{L_0}L_1;`\alpha`]
        \end{xy}
    \end{equation*}
    with $A\in\Prok(\Cc)$. By the universal property of cokernels, this triangle determines a commuting triangle
    \begin{equation*}
        \begin{xy}
            \Vtriangle/>`>`^{ (}->/<350,400>[X`A`V\cup_{L_0}L_1;`s`]
        \end{xy}.
    \end{equation*}
    This triangle guarantees that the map $A\to(V\cup_{L_0}L_1)\to^q X$ is left inverse to the map $X\to A$. Accordingly, the map
    \begin{equation*}
        A\to (V\cup_{L_0}L_1)\to^q X\to A
    \end{equation*}
    is an idempotent, and $X\in\FM^\vee(\Cc)$.

    Now consider the map of elementary Tate objects $\beta\colon V\cup_{L_0}L_1\to V/L_0$ given in $\Tatek(\Cc)$ by the composition
    \begin{equation*}
        V\cup_{L_0} L_1\onto^q X\cong L_1/L_0\into V/L_0.
    \end{equation*}
    Because the category $\Indk(\Cc)\subset\elTatek(\Cc)$ is right filtering (Proposition \ref{prop:indintate}), the map $\beta$ factors through an admissible epic
    \begin{equation*}
        \begin{xy}
            \Vtriangle/->>`>`>/<350,400>[V\cup_{L_0}L_1`B`V/L_0;`\beta`]
        \end{xy}
    \end{equation*}
    with $B\in\Indk(\Cc)$. This triangle shows that the map $B\to\coker(\beta)$ is $0$. Because $X\cong\ker(V/L_0\to\coker(\beta))$, the universal property of kernels determines a commuting triangle
    \begin{equation*}
        \begin{xy}
            \Vtriangle/>`>`>/<350,400>[V\cup_{L_0}L_1`B`X;``]
        \end{xy}.
    \end{equation*}
    As above, this triangle shows that the map $B\to X$ is left inverse to the map $X\to^s V\cup_{L_0}L_1\to B$. Accordingly, the map
    \begin{equation*}
        B\to X\to^s V\cup_{L_0}L_1\to B
    \end{equation*}
    is an idempotent, and $X\in\FM(\Cc)$ as well.

    We have shown that $X\in\FM^\vee(\Cc)\cap\FM(\Cc)$. When $\Cc$ is idempotent complete, we conclude that $X\in\Cc$ by Proposition \ref{prop:fmandfmveeintate}. Because $\elTatek(\Cc)\subset\Tatek(\Cc)$ is a fully exact sub-category, and because $\Prok(\Cc)\subset\elTatek(\Cc)$ is closed under extensions, we conclude that $f$ is an admissible monic in $\Prok(\Cc)$.

    Now suppose we are given a triangle of the form \eqref{colatep}. As above, it suffices to show that $\coker(g)\in\Indk(\Cc)$. By the universal property of kernels, \eqref{colatep} determines a commuting triangle
    \begin{equation*}
        \begin{xy}
            \Vtriangle/>`^{ (}->`^{ (}->/<350,450>[\ker(q_0)`\ker(q_1)`V;g'``]
        \end{xy}
    \end{equation*}
    where both maps to $V$ are inclusions of lattices. Our argument above shows that $g'$ is an admissible monic in $\Prok(\Cc)$, and Noether's Lemma \cite[Lemma 3.5]{Buh:10} shows that $\ker(g)\cong\ker(q_1)/\ker(q_0)$. By Proposition \ref{prop:latcomm}, we have $\ker(q_1)/\ker(q_0)\in\Cc$. We conclude that $g$ is an admissible epic in $\Indk(\Cc)$.
\end{proof}

\begin{proof}[Proof of Theorem \ref{thm:grdir}]
    Let $L_0$ and $L_1$ be two lattices of $V$. We begin by showing that there exists a lattice with both $L_0$ and $L_1$ as admissible sub-objects. Let $V\colon I\to\Prok(\Cc)$ be an elementary Tate diagram representing $V$. Because $\Prok(\Cc)\subset\elTatek(\Cc)$ is left s-filtering (Proposition \ref{prop:prointatesfilt}), and because $L_0\oplus L_1$ is an admissible Pro-object, there exists $i\in I$, such that the morphism $L_0\oplus L_1\to V$ factors through the admissible monic $V_i\into V$. By Lemma \ref{lemma:grdirkey}, both maps $L_0\to V_i$ and $L_1\to V_i$ are admissible monics in $\Prok(\Cc)$. We conclude that $\Gr(V)$ is directed.

    We now show that there exists a common sub-lattice of $L_0$ and $L_1$. Because $\Indk(\Cc)\subset\elTatek(\Cc)$ is right filtering (Proposition \ref{prop:indintate}), the map
   \begin{equation*}
        V\to V/L_0\oplus V/L_1
    \end{equation*}
    factors through an admissible epic $V\onto V/N$ for some $V/N\in\Indk(\Cc)$. Define
    \begin{equation*}
        N:=\ker(V\onto V/N).
    \end{equation*}
    The universal property of kernels ensures that the admissible monic
    \begin{equation*}
        N\into V
    \end{equation*}
    factors through the admissible monics $L_i\into V$ for $i=0,~1$. By Lemma \ref{lemma:grdirkey}, it suffices to show that $N$ is a lattice, or, equivalently, that $N\in\Prok(\Cc)$.

    Because $\Prok(\Cc)\subset\elTatek(\Cc)$ is left s-filtering, it is closed under admissible sub-objects (Proposition \ref{prop:buh}). Therefore, it suffices to show that $N\to L_0$ is an admissible monic in $\elTatek(\Cc)$. Further, because $\elTatek(\Cc)\subset\lex(\Prok(\Cc))$ is closed under extensions, it suffices to show that the left exact presheaf $L_0/N$ is actually an object in $\elTatek(\Cc)$. The same argument as in the proof of Lemma \ref{lemma:grdirkey} can be used to show that, because $\Cc$ is idempotent complete, $L_0/N$ is actually an object in $\Cc$. We conclude that $N\into V$ is a common sub-lattice of $L_0$ and $L_1$.
\end{proof}

\section{n-Tate Objects}\label{sec:ntate}
In this section we consider $n$-Tate objects, and record their basic properties. We then recall the Beilinson--Parshin theory of $n$-dimensional ad\`{e}les and we show that the $n$-dimensional ad\`{e}les of an $n$-dimensional scheme are naturally an $n$-Tate object.

\subsection{The Category of n-Tate Objects and its Properties}
\begin{definition}
    Let $\Cc$ be idempotent complete. Define the category $\nelTatek(\Cc)$ of \emph{elementary $n$-Tate objects of size at most $\kappa$} by
    \begin{equation*}
        \nelTatek(\Cc):=\elTatek((\mathsf{n}-1)\text{-}\Tatek(\Cc))
    \end{equation*}
    The category $\nTatek(\Cc)$ of \emph{$n$-Tate objects of size at most $\kappa$} is the idempotent completion of $\nelTatek(\Cc)$.
\end{definition}

\begin{example}
    Define a \emph{$0$-dimensional local field} to be a finite field. Let $k$ be a finite field. Define an \emph{$n$-dimensional local field over $k$} to be a complete discrete valuation field $F$ with ring of integers $R$ such that the residue field of $R$ is an $(n-1)$-dimensional local field over $k$. Vector spaces over $n$-dimensional local fields are canonically elementary $n$-Tate objects in the category of finitely generated abelian groups.
\end{example}

The results of Sections \ref{sec:tate} and \ref{sec:applat} carry over to $n$-Tate objects.
\begin{theorem}\label{thm:ntatex}
    Let $\Cc$ be an idempotent complete exact category.
    \begin{enumerate}
        \item $\elTatek(\Cc)$ is a the smallest exact sub-category of $\Indk(\Prok((\mathsf{n}-1)\text{-}\Tatek(\Cc)))$ which contains $\Prok((\mathsf{n}-1)\text{-}\Tatek(\Cc))$, and $\Indk((\mathsf{n}-1)\text{-}\Tatek(\Cc))$ and which is closed under extensions.
        \item For $k\ge 0$, the exact categories $\nelTatek(S_k\Cc)$ and $S_k\nelTatek(\Cc)$ are canonically equivalent. The exact categories $\nTatek(S_k\Cc)$ and $S_k\nTatek(\Cc)$ are canonically equivalent as well.
        \item An exact functor $F\colon\Cc\to\Dc$ extends canonically to a pair of exact functors
            \begin{equation*}
                \begin{xy}
                    \square/>`^{ (}->`^{ (}->`>/<1000,500>[\nelTatek(\Cc)`\nelTatek(\Dc)`\nTatek(\Cc)`\nTatek(\Dc);\widetilde{F}```\widetilde{F}]
                \end{xy}
            \end{equation*}
            If $F$ is faithful, fully faithful, or an equivalence, then so are both functors $\widetilde{F}$.
        \item The sub-category $\Prok((\mathsf{n}-1)\text{-}\Tatek(\Cc))\subset\nelTatek(\Cc)$ is left s-filtering.
        \item The category $\nelTatec(\Cc)$ is split exact if $\Cc$ is.
        \item Every elementary $n$-Tate object has a lattice.
        \item Let $V$ be an elementary $n$-Tate object in $\Cc$. An exact functor $F\colon\Cc\to\Dc$ induces an order-preserving map
            \begin{equation*}
                \begin{xy}
                    \morphism<750,0>[\Gr(V)`\Gr(\widetilde{F}(V));F_V]
                    \morphism(0,-200)/|->/<750,0>[L`\widetilde{F}(L);]
                \end{xy}.
            \end{equation*}
        \item The quotient of a lattice by a sub-lattice is an $(n-1)$-Tate object.
        \item The Sato Grassmannian of an elementary $n$-Tate object is a directed and co-directed poset.
    \end{enumerate}
\end{theorem}

\begin{remark}
    The observations behind remarks \ref{rmk:Kcalk} and \ref{rmk:KTate} similarly show that
    \begin{equation*}
        K_i(\nTatek(\Cc))\cong K_{i-n}(\Cc)
    \end{equation*}
    when $\Cc$ is idempotent complete.
\end{remark}

\begin{proposition}
    For any $X$ in an exact category $\Cc$, denote by $X[t]$, $X[[t]]$ and $X((t))$ the admissible Ind, Pro and elementary Tate objects
    \begin{align*}
        X[t]&:=\bigoplus_{\Nb} X\text{, }\\
        X[[t]]&:=\prod_{\Nb} X\text{, and}\\
        X((t))&:=X[[t]]\oplus X[t]\in\elTatec(\Cc).
    \end{align*}
    Similarly, we define $X((t_1))\cdots((t_n))\in\nelTatec(\Cc)$ by
    \begin{equation*}
        X((t_1))\cdots((t_n)):=(X((t_1))\cdots((t_{n-1})))((t_n))
    \end{equation*}
    Now suppose that $\Cc$ is a split exact category for which there exists a collection of objects $\{S_i\}_{i\in\Nb}\subset\Cc$ such that every object $Y\in\Cc$ is a direct summand of $\bigoplus_{i=0}^n S_i$ for some $n$. As above, denote by $\widehat{\prod_{\Nb}S}$ and $\widehat{\bigoplus_{\Nb} S}$ the admissible Pro and Ind-objects
    \begin{align*}
        \widehat{\prod_{\Nb} S}&:=\prod_{\Nb}(\prod_{i\in\Nb} S_i)\text{, and}\\
        \widehat{\bigoplus_{\Nb} S}&:=\bigoplus_{\Nb}(\bigoplus_{\Nb}S_i).
    \end{align*}
    Then every countable $n$-Tate object in $\Cc$ is a direct summand of
    \begin{equation*}
        (\widehat{\prod_{\Nb} S}\oplus\widehat{\bigoplus_{\Nb} S})((t_2))\cdots((t_n)).
    \end{equation*}
\end{proposition}
\begin{proof}
    We induct on $n$. Proposition \ref{prop:bigtateobject} establishes the case $n=1$. Assume the result is true for $n$. Then the category $\nTatec(\Cc)$ is split exact and every object is a direct summand of $(\widehat{\prod_{\Nb} S}\oplus\widehat{\bigoplus_{\Nb} S})((t_2))\cdots((t_n))$. We can therefore apply Proposition \ref{prop:bigtateobject} to $\nTatec(\Cc)$ and conclude that every object in $\elTate(\nTatec(\Cc))$ is a direct summand of
    \begin{equation*}
        (\widehat{\prod_{\Nb} S}\oplus\widehat{\bigoplus_{\Nb} S})((t_2))\cdots((t_n))((t_{n+1})).
    \end{equation*}
    This completes the induction.
\end{proof}

\begin{example}
    Let $R$ be a ring. Every countable $n$-Tate $R$-module is a direct summand of $R((t_1))\cdots((t_n))$.
\end{example}

\subsection{Beilinson--Parshin Ad\`{e}les}\label{sec:bp}
\begin{definition}[Beilinson--Parshin]\label{def:globaladeles}
    Let $X$ be an $n$-dimensional Noetherian scheme. For $0\leq i\leq n$, denote by $|X|_i$ the set of points $p\in X$ such that the closure of $p$ is an $i$-dimensional sub-scheme of $X$. Given $p\in X$, denote the inclusion of its closure by $j_{\overline{p}}\colon\overline{p}\into X$. Denote the inclusion of the $r^{th}$-order formal neighborhood of its closure by $j_{\overline{p^r}}\colon\overline{p^r}\into X$. Define the \emph{$n$-dimensional ad\`{e}les}
    \begin{equation*}
        \begin{xy}
            \morphism<1000,0>[\QCoh(X)`\Mod(\Oc_X);\Ab^n_X(-)]
        \end{xy}
    \end{equation*}
    to be the functor which commutes with direct limits and whose restriction to $\Coh(X)$ is inductively given by,
    \begin{enumerate}
        \item for $n=0$, $\Ab^0_X(\Fc):=\Fc$, and
        \item for $n>0$,
            \begin{enumerate}
                \item if $X$ is irreducible with generic point $\eta$, denote by $j_\eta$ the inclusion of the generic point, and define
                    \begin{equation*}
                        \Ab^n_X(\Fc):=\colim_i \prod_{p\in |X|_{n-1}}\lim_r j_{\overline{p^r},\ast}\Ab^{n-1}_{\overline{p^r}}(j_{\overline{p^r}}^\ast\Fc_i)
                    \end{equation*}
                    where the colimit is over the poset of coherent subsheaves $\Fc_i\subset j_{\eta,\ast}j_\eta^\ast\Fc$ such that $j_\eta^\ast\Fc_i=j_\eta^\ast\Fc$.
                \item if $X$ has irreducible components $\{X_a\}$, then denote by $j_{X_a}$ the inclusion of the component $X_a$ and define
                    \begin{equation*}
                        \Ab^n_X(\Fc):=\bigoplus_{a}j_{X_a,\ast}\Ab^n_{X_a}(j_{X_a}^\ast\Fc).
                    \end{equation*}
            \end{enumerate}
    \end{enumerate}
\end{definition}
\begin{remark}
    This definition corresponds to the \emph{reduced} $n$-dimensional ad\`{e}les of \cite{Hub:91}.
\end{remark}

\begin{example}
    When $X=\Spec(\mathbb{Z})$, $\Ab^1_X(\Oc_X)$ is the finite ad\`{e}les $\mathbb{Q}\otimes(\prod_{p}\mathbb{Z}_p)$.
\end{example}

We can also discuss the ad\`{e}les at a single place, or at a specified collection of places.
\begin{definition}
    Let $X$ be an $n$-dimensional Noetherian scheme.
    \begin{enumerate}
        \item Let $\xi:=(p_0<\ldots<p_n)$ be an increasing sequence of points in $X$, with $\overline{\{p_i\}}$ of dimension $i$. Denote by $d(\xi)$ the sequence $d(\xi):=(p_0<\ldots<p_{n-1})$. Using the notation of Definition \ref{def:globaladeles}, we define the \emph{$n$-dimensional ad\`{e}les at the place $\xi$}
            \begin{equation*}
                \begin{xy}
                    \morphism<1000,0>[\QCoh(X)`\Mod(\Oc_X);\Ab^n_{X,\xi}(-)]
                \end{xy}
            \end{equation*}
            to be the functor which commutes with direct limits and whose restriction to $\Coh(X)$ is inductively given by,
            \begin{enumerate}
                \item for $n=0$, $\Ab^0_{X,\xi}(\Fc):=\Fc$, and
                \item for $n>0$, define
                    \begin{equation*}
                        \Ab^n_{X,\xi}(\Fc):=\colim_i \lim_r j_{\overline{p_{n-1}^r},\ast}\Ab^{n-1}_{\overline{p_{n-1}^r},d(\xi)}(j_{\overline{p_{n-1}^r}}^\ast\Fc_i)
                    \end{equation*}
                    where the colimit is over the poset of coherent subsheaves $\Fc_i\subset j_{{p_n},\ast}j_{p_n}^\ast\Fc$ such that $j_{p_n}^\ast\Fc_i=j_{p_n}^\ast\Fc$.\footnote{Note that dimension considerations imply that $p_n$ is the generic point of the irreducible component of $X$ which contains $\xi$.}
            \end{enumerate}
        \item Let $T:=\{\xi_a\}_{a\in A}$ be a collection of sequences $\xi$ as above. We define the \emph{$n$-dimensional ad\`{e}les at the collection of places $T$}
            \begin{equation*}
                \begin{xy}
                    \morphism<1000,0>[\QCoh(X)`\Mod(\Oc_X);\Ab^n_{X,T}(-)]
                \end{xy}
            \end{equation*}
            to be the functor which commutes with direct limits and whose restriction to $\Coh(X)$ is inductively given by,
            \begin{enumerate}
                \item for $n=0$, $\Ab^0_{X,T}(\Fc):=\Fc$, and
                \item for $n>0$, define
                    \begin{equation*}
                        \Ab^n_{X,T}(\Fc):=\colim_i \prod_{\xi\in T}\Ab^n_{X,\xi}(\Fc_i)
                    \end{equation*}
                    where the colimit is over the poset of coherent sub-sheaves $\Fc_i$ of $\bigoplus_{p_n\in X} j_{{p_n},\ast}j_{p_n}^\ast\Fc$ such that for each $n$-dimensional point $p_n$, which is contained in some $\xi\in T$, we have $j_{p_n}^\ast\Fc_i=j_{p_n}^\ast\Fc$.
            \end{enumerate}
    \end{enumerate}
\end{definition}

Beilinson \cite{Bei:80} formulated the $n$-dimensional ad\`{e}les as the top degree piece of a functorial flasque resolution $\Fc\to\Ab^\bullet_X(\Fc)$ of a coherent sheaf $\Fc$.\footnote{For a detailed description of the full ad\`{e}lic resolution, see Huber \cite{Hub:91}.} In particular, $\Ab^n_X(-)$ is an exact functor with a canonical natural surjection
\begin{equation*}
    \Ab^n_X(-)\to H^n(X;-).
\end{equation*}
Following Parshin, Beilinson used this to express the Grothendieck trace map via a sum of residues, in analogy with Tate's work \cite{Tat:68} on curves.

Denote by $\Coh_0(X)\subset\Coh(X)$ the full sub-category consisting of sheaves with $0$-dimensional support, and define $0-\elTate(\Coh_0(X)):=\Coh_0(X)$.

\begin{theorem}\label{thm:adeles}
    Let $X$ be an $n$-dimensional Noetherian scheme.
    \begin{enumerate}
        \item The $n$-dimensional ad\`{e}les factor through an exact functor
            \begin{equation*}
                \begin{xy}
                    \morphism<1000,0>[\Coh(X)`\nelTate(\Coh_0(X));\Ab^n_X(-)]
                \end{xy}.
            \end{equation*}
        \item Let $\xi:=(p_0<\ldots<p_n)$ be an increasing sequence of points in $X$, with $\overline{\{p_i\}}$ of dimension $i$, and let $T=\{\xi_a\}_{a\in A}$ be a collection of such sequences. The $n$-dimensional ad\`{e}les at $\xi$ and $T$ factor through exact functors
            \begin{align*}
                &\begin{xy}
                    \morphism<1000,0>[\Coh(X)`\nelTate(\Coh_0(X));\Ab^n_{X,\xi}(-)]
                \end{xy},\intertext{and}
                &\begin{xy}
                    \morphism<1000,0>[\Coh(X)`\nelTate(\Coh_0(X));\Ab^n_{X,T}(-)]
                \end{xy}.
            \end{align*}
    \end{enumerate}
\end{theorem}

\begin{remark}\mbox{}
    \begin{enumerate}
        \item This result should not be surprising to experts, but we have been unable to find it recorded in the literature. Yekutieli \cite[Theorem 3.3.2]{Yek:92} shows that for a complete flag $\xi$ in an integral, excellent, Noetherian $n$-dimensional scheme $X$, the ring $\Ab^n_{X,\xi}(\Oc_X)$ is a finite product of $n$-dimensional local fields; in particular, if we forget the ring structure and take global sections, this shows that $\Ab^n_{X,\xi}(\Oc_X)$ is an $n$-Tate abelian group (in the mixed characteristic case) or an $n$-Tate vector space (in the equi-characteristic case).
        \item More recently, for schemes over a base field $k$, Osipov \cite{Osi:07} has given a similar characterization of higher ad\`{e}les using a formalism of ``categories $C_n$''. It seems likely that there is a fully faithful embedding of $\nTate(\Vect_f(k))$ into $C_n$, but we do not pursue this here.
    \end{enumerate}
\end{remark}

\begin{proof}
    We prove the result for the global ad\`{e}les $\Ab_X^n$ by induction on $n$; the proofs for $\Ab^n_{X,\xi}$ and $\Ab^n_{X,T}$ follow by similar reasoning.

    For $n=0$, there is nothing to show. Suppose that we have shown the result for $0\le m<n$. From the definition, it is enough to prove the result for $X$ irreducible. Let $\Fc$ be a coherent sheaf on $X$. To show the factorization exists, we use the inductive hypothesis to show that $\Ab^n_X(\Fc)$ is naturally an object in $\Ind(\Pro((n-1)\text{-}\Tate(\Coh_0(X))))$. We then exhibit a lattice of $\Ab^n_X(\Fc)$ to show
    \begin{equation*}
        \Ab^n_X(\Fc)\in\nelTate(\Coh_0(X)).
    \end{equation*}

    Denote the generic point of $X$ by $\eta$, and the inclusion of the generic point by $j_\eta$. By definition
    \begin{equation*}
        \Ab^n_X(\Fc):=\colim_i \prod_{p\in |X|_{n-1}}\lim_r j_{\overline{p^r},\ast}\Ab^{n-1}_{\overline{p^r}}(j_{\overline{p^r}}^\ast\Fc_i)
    \end{equation*}
    where the colimit ranges over coherent subsheaves $\Fc_i\subset j_{\eta,\ast}j_\eta^\ast\Fc$ such that $j_\eta^\ast\Fc_i=j_\eta^\ast\Fc$. Our inductive hypothesis states that
    \begin{align*}
        \Ab^{n-1}_{\overline{p^r}}(j_{\overline{p^r}}^\ast\Fc_i)\in(n-1)\text{-}\Tate(\Coh_0(\overline{p^r})),\intertext{so }
        j_{\overline{p^r},\ast}\Ab^{n-1}_{\overline{p^r}}(j_{\overline{p^r}}^\ast\Fc_i)\in(n-1)\text{-}\Tate(\Coh_0(X)),\intertext{and we see that}
        \Ab^n_X(\Fc)\in\Ind(\Pro((n-1)\text{-}\Tate(\Coh_0(X)))).
    \end{align*}
    To show that $\Ab^n_X(\Fc)$ is an $n$-Tate object, it suffices to produce a lattice (Theorem \ref{thm:eltatechar}). Denote by $I$ the directed poset indexing coherent subsheaves $\Fc_i\subset j_{\eta,\ast}j_\eta^\ast\Fc$ such that $j_\eta^\ast\Fc_i=j_\eta^\ast\Fc$. For any $i\in I$, denote by $I_i$ the final subset of $I$ consisting of all $\Fc_j$ containing $\Fc_i$. Note that $\colim_{j\in I_i}\Fc_j\cong j_{\eta,\ast}j_\eta^\ast\Fc$.

    We claim that the inclusion
    \begin{equation}
        \begin{xy}
            \morphism/^{ (}->/<1250,0>[\prod_{p\in |X|_{n-1}}\lim_r j_{\overline{p^r},\ast}\Ab^{n-1}_{\overline{p^r}}(j_{\overline{p^r}}^\ast\Fc_i)`\Ab^n_X(\Fc);]
        \end{xy}
    \end{equation}
    is a lattice. Our inductive hypothesis guarantees that
    \begin{equation*}
        \prod_{p\in |X|_{n-1}}\lim_r j_{\overline{p^r},\ast}\Ab^{n-1}_{\overline{p^r}}(j_{\overline{p^r}}^\ast\Fc_i)\in\Pro((n-1)\text{-}\Tate(\Coh_0(X))).
    \end{equation*}
    For each $j\in I_i$, we have a short exact sequence
    \begin{equation}
        0\to\Fc_i\to\Fc_j\to Q_j\to 0
    \end{equation}
    Our inductive hypothesis ensures that for each $p\in|X|_{n-1}$ and each $r$, we have an exact sequence
    \begin{equation}
        0\to j_{\overline{p^r},\ast}\Ab^{n-1}_{\overline{p^r}}(j_{\overline{p^r}}^\ast\Fc_i)\to j_{\overline{p^r},\ast}\Ab^{n-1}_{\overline{p^r}}(j_{\overline{p^r}}^\ast\Fc_j)\to j_{\overline{p^r},\ast}\Ab^{n-1}_{\overline{p^r}}(j_{\overline{p^r}}^\ast Q_j)\to 0.
    \end{equation}
    These exact sequences fit into an admissible Pro-diagram of exact sequences, and thus an exact sequence of admissible Pro-objects
    \begin{equation}\label{adellat}
        \begin{xy}
            \morphism(0,300)<1000,0>[0`\prod_{p\in|X|_{n-1}}\lim_r j_{\overline{p^r},\ast}\Ab^{n-1}_{\overline{p^r}}(j_{\overline{p^r}}^\ast\Fc_i);]
            \morphism(1000,300)<1650,0>[\prod_{p\in|X|_{n-1}}\lim_r j_{\overline{p^r},\ast}\Ab^{n-1}_{\overline{p^r}}(j_{\overline{p^r}}^\ast\Fc_i)`\prod_{p\in|X|_{n-1}}\lim_r j_{\overline{p^r},\ast}\Ab^{n-1}_{\overline{p^r}}(j_{\overline{p^r}}^\ast\Fc_j);]
            \morphism(2650,300)/-/<1000,0>[\prod_{p\in|X|_{n-1}}\lim_r j_{\overline{p^r},\ast}\Ab^{n-1}_{\overline{p^r}}(j_{\overline{p^r}}^\ast\Fc_j)`;]
            \morphism(1500,0)<1000,0>[`\prod_{p\in|X|_{n-1}}\lim_r j_{\overline{p^r},\ast}\Ab^{n-1}_{\overline{p^r}}(j_{\overline{p^r}}^\ast Q_j);]
            \morphism(2500,0)<1000,0>[\prod_{p\in|X|_{n-1}}\lim_r j_{\overline{p^r},\ast}\Ab^{n-1}_{\overline{p^r}}(j_{\overline{p^r}}^\ast Q_j)`0;]
        \end{xy}.
    \end{equation}
    We claim that
    \begin{equation*}
        \prod_{p\in|X|_{n-1}}\lim_r j_{\overline{p^r},\ast}\Ab^{n-1}_{\overline{p^r}}(j_{\overline{p^r}}^\ast Q_j)\in(n-1)\text{-}\Tate(\Coh_0(X)).
    \end{equation*}
    Indeed, by definition, $j_\eta^\ast\Fc_i\cong j_\eta^\ast F_j$, so the support of $Q_j$ has dimension at most $n-1$. Since $X$ is Noetherian, the support of $Q_j$ is Noetherian as well. The support therefore contains finitely many irreducible components. In particular, $j_{\overline{p^r}}^\ast Q_j=0$ for all but finitely many $p\in|X|_{n-1}$.  Further, because $Q_j$ is coherent, for each $p$ such that $Q_j$ is non-zero on $\overline{p}$, there exists $r<\infty$ such that $j_{\overline{p^s}}^\ast Q_j\cong j_{\overline{p^r}}^\ast Q_j$ for all $s\ge r$. We conclude that
    \begin{equation*}
        \prod_{p\in|X|_{n-1}}\lim_r j_{\overline{p^r},\ast}\Ab^{n-1}_{\overline{p^r}}(j_{\overline{p^r}}^\ast Q_j)\in(n-1)\text{-}\Tate(\Coh_0(X))
    \end{equation*}
    is a finite direct sum of $(n-1)$-Tate objects, so is an $(n-1)$-Tate object itself.

    We now take the colimit over $I_i$ of the short exact sequences \eqref{adellat} to obtain a short exact sequence
    \begin{equation*}
        \begin{xy}
            \morphism(0,300)<1000,0>[0`\prod_{p\in|X|_{n-1}}\lim_r j_{\overline{p^r},\ast}\Ab^{n-1}_{\overline{p^r}}(j_{\overline{p^r}}^\ast\Fc_i);]
            \morphism(1000,300)<1000,0>[\prod_{p\in|X|_{n-1}}\lim_r j_{\overline{p^r},\ast}\Ab^{n-1}_{\overline{p^r}}(j_{\overline{p^r}}^\ast\Fc_i)`\Ab^n_X(\Fc);]
            \morphism(2000,300)/-/[\Ab^n_X(\Fc)`;]
            \morphism(500,0)<1250,0>[`\colim_{I_i}\prod_{p\in|X|_{n-1}}\lim_r j_{\overline{p^r},\ast}\Ab^{n-1}_{\overline{p^r}}(j_{\overline{p^r}}^\ast Q_j);]
            \morphism(1750,0)<1250,0>[\colim_{I_i}\prod_{p\in|X|_{n-1}}\lim_r j_{\overline{p^r},\ast}\Ab^{n-1}_{\overline{p^r}}(j_{\overline{p^r}}^\ast Q_j)`0;]
        \end{xy}.
    \end{equation*}
    By construction,
    \begin{align*}
        \prod_{p\in|X|_{n-1}}\lim_r j_{\overline{p^r},\ast}\Ab^{n-1}_{\overline{p^r}}(j_{\overline{p^r}}^\ast\Fc_i)\in\Pro((n-1)\text{-}\Tate(\Coh_0(X)))\intertext{while }
        \colim_{I_i}\prod_{p\in|X|_{n-1}}\lim_r j_{\overline{p^r},\ast}\Ab^{n-1}_{\overline{p^r}}(j_{\overline{p^r}}^\ast Q_j)\in\Ind((n-1)\text{-}\elTate(\Coh_0(X))).
    \end{align*}
    We conclude that $\Ab^n_X(\Fc)\in\nelTate(\Coh_0(X))$.

    It remains to show that the functor $\Ab^n_X(-)$ is exact. Because $\nelTate(\Coh_0(X))$ is closed under extenstions in $\Ind(\Pro((n-1)\text{-}\Tate(\Coh_0(X))))$ (Theorem \ref{thm:ntatex}), it suffices to show that $\Ab^n_X(-)$ takes exact sequences in $\Coh(X)$ to exact sequences in $\Ind(\Pro((n-1)\text{-}\Tate(\Coh_0(X))))$.

    Let
    \begin{equation*}
        \Fc^0\into \Fc^1\onto \Fc^2
    \end{equation*}
    be a short exact sequence in $\Coh(X)$. Because the inclusion of the generic point is an affine morphism, the sequence of quasi-coherent sheaves
    \begin{equation}\label{adelesexact}
        j_{\eta,\ast}j_\eta^\ast\Fc^0\into j_{\eta,\ast}j_\eta^\ast\Fc^1\onto j_{\eta,\ast}j_\eta^\ast\Fc^2
    \end{equation}
    is exact. Just as above, for $a=0$, 1 or 2, we can write
    \begin{equation*}
        j_{\eta,\ast}j_\eta^\ast\Fc^a\cong\colim_i \Fc_i^a
    \end{equation*}
    where the colimit ranges over the directed set of coherent sub-sheaves $\Fc_i^a\subset j_{\eta,\ast}j_\eta^\ast\Fc^a$ such that $j_\eta^\ast\Fc_i^a=j_\eta^\ast\Fc^a$. Because $X$ is Noetherian, $\QCoh(X)\simeq\lex(\Coh(X))$. Theorem \ref{thm:indleftspecial} shows that $\Ind(\Coh(X))\subset\QCoh(X)$ is closed under extensions; therefore \eqref{adelesexact} is an exact sequence in $\Ind(\Coh(X))$. By Proposition \ref{prop:inde=eind}, we can straighten this to obtain an admissible Ind-diagram of exact sequences of coherent sheaves indexed by a directed poset $I$. Our inductive hypothesis guarantees that for each point $p\in X$ of codimension 1, each $r\ge 0$ and each $i\in I$, the sequence
    \begin{equation*}
         j_{\overline{p^r},\ast}\Ab^{n-1}_{\overline{p^r}}(j_{\overline{p^r}}^\ast\Fc_i^0)\into j_{\overline{p^r},\ast}\Ab^{n-1}_{\overline{p^r}}(j_{\overline{p^r}}^\ast\Fc_i^1)\onto j_{\overline{p^r},\ast}\Ab^{n-1}_{\overline{p^r}}(j_{\overline{p^r}}^\ast\Fc_i^2)
    \end{equation*}
    is exact in $(n-1)\text{-}\elTate(\Coh_0(X))$. Taking the appropriate limits and colimits, we conclude that the sequence
    \begin{equation*}
        \Ab^n_X(\Fc^0)\into\Ab^n_X(\Fc^1)\onto\Ab^n_X(\Fc^2)
    \end{equation*}
    is exact in $\Ind(\Pro((n-1)\text{-}\elTate(\Coh_0(X))))$.
\end{proof}

\appendix
\section{Remarks on the left s-filtering condition, after T. B\"{u}hler}\label{sec:buhler}
In this appendix, we recall Schlichting's definition of left s-filtering \cite{Sch:04}, and we reproduce a proof, due to T. B\"{u}hler, that this definition is equivalent to Definition \ref{def:lsfilt}.

\begin{definition}[Schlichting]\label{def:schlsfilt}
    Let $\Dc$ be an exact category. An exact, full sub-category $\Cc\subset\Dc$ is \emph{Schlichting left s-filtering} if
    \begin{enumerate}
        \item for any exact sequence $X\into Y\onto Z$ in $\Dc$, $X$ and $Z$ are in $\Cc$ if and only if $Y$ is,
        \item $\Cc$ is left filtering in $\Dc$ (in the sense of Definition \ref{def:lfilt}), and
        \item for every admissible epic $F\onto X$ in $\Dc$, with $X\in\Cc$, there exists an admissible monic $Y\into F$, with $Y\in\Cc$, such that the composite map $Y\to X$ is an admissible epic in $\Cc$.
    \end{enumerate}
\end{definition}

\begin{proposition}[B\"{u}hler]\label{prop:buh}
    Definitions \ref{def:lsfilt} and \ref{def:schlsfilt} are equivalent, i.e. an exact, full sub-category $\Cc\subset\Dc$ satisfies Definition \ref{def:schlsfilt} if and only if $\Cc$ is left filtering and left special in $\Dc$.
\end{proposition}
\begin{proof}
    For the ``only if'', it suffices to show that $\Cc\subset\Dc$ is left special if it satisfies the third condition of Definition \ref{def:schlsfilt}. By assumption, given an admissible epic $F\onto X$ in $\Dc$ with $X$ in $\Cc$, there exists a commuting square
    \begin{equation*}
        \begin{xy}
            \square/->>`>`=`->>/[Y`X`F`X;```]
        \end{xy}
    \end{equation*}
    with $Y\onto X$ an admissible epic in $\Cc$. The universal property of kernels implies that this square extends to a commuting diagram
    \begin{equation*}
        \begin{xy}
            \square/^{ (}->`>`>`^{ (}->/[Z`Y`G`F;```]
            \square(500,0)/->>`>`=`->>/[Y`X`F`X;```]
        \end{xy}
    \end{equation*}
    in which the top row is an exact sequence in $\Cc$, while the bottom row is an exact sequence in $\Dc$. We conclude that $\Cc$ is left special in $\Dc$.

    For the ``if'', suppose $\Cc\subset\Dc$ is left filtering and left special. We begin by showing that this implies that for every admissible epic $F\onto X$ in $\Dc$ with $X\in\Cc$, there exists a commuting diagram
    \begin{equation*}
        \begin{xy}
            \square/^{ (}->`^{ (}->`^{ (}->`^{ (}->/[Z`Y`G`F;```]
            \square(500,0)/->>`^{ (}->`=`->>/[Y`X`F`X;```]
        \end{xy}
    \end{equation*}
    in which the top row is an exact sequence in $\Cc$, the bottom row is an exact sequence in $\Dc$, and the vertical maps are admissible monics in $\Dc$; note that if such a diagram always exists, then $\Cc$ satisfies the third condition of Definition \ref{def:schlsfilt}, because we can take $Y$ to be as in the diagram above.

    To see that the diagram above exists, observe that, because $\Cc$ is left special, there exists a diagram
    \begin{equation*}
        \begin{xy}
            \square/^{ (}->`>`>`^{ (}->/[Z`Y`G`F;```]
            \square(500,0)/->>`>`=`->>/[Y`X`F`X;```]
        \end{xy}
    \end{equation*}
    with no assumptions on the vertical maps, in which the top row is exact in $\Cc$ and the bottom row is exact in $\Dc$. Because $\Cc$ is left filtering in $\Dc$, the map $Z\to G$ factors through an admissible monic $Z'\into G$ with $Z'\in\Cc$. Because pushouts of admissible monics along arbitrary maps exist and are admissible monics, we denote $Y':=Z'\cup_Z Y$ and observe that we have a commuting diagram
    \begin{equation*}
        \begin{xy}
            \square(0,500)/^{ (}->`>`>`^{ (}->/[Z`Y`Z'`Y';```]
            \square(500,500)/->>`>`=`>/[Y`X`Y'`X;```]
            \square(0,0)/^{ (}->`^{ (}->`>`^{ (}->/[Z'`Y'`G`F;```]
            \square(500,0)/>`>`=`->>/[Y'`X`F`X;```]
        \end{xy}.
    \end{equation*}
    Because the top sequence is exact, and the upper left square is a pushout in which the horizontal maps are admissible monics, the middle row is exact by \cite[Proposition 2.12]{Buh:10}. Because the lower left vertical map is an admissible monic (by assumption), the 5-lemma \cite[Corollary 3.2]{Buh:10} implies that lower middle vertical map is also an admissible monic \cite[Corollary 3.2]{Buh:10}. We conclude that the bottom rectangle of this diagram is of the desired form.

    It remains to show that if $\Cc\subset\Dc$ is left filtering and left special, then the first condition in Definition \ref{def:schlsfilt} is satisfied. Let
    \begin{equation*}
        X\into Y\onto Z
    \end{equation*}
    be a short exact sequence in $\Dc$. Lemma \ref{lemma:lsextclosed} shows that $Y$ is in $\Cc$ if $X$ and $Z$ are. Conversely, suppose $Y\in\Cc$. Because $\Cc$ is left filtering in $\Dc$, there exists a commuting triangle
    \begin{equation*}
        \begin{xy}
            \qtriangle/->>`>`<-^{) }/[Y`Z`W;``]
        \end{xy}
    \end{equation*}
    with $W$ in $\Cc$. The commuting triangle implies that the map $W\to Z$ is an epic admissible monic, i.e. an isomorphism. This shows that $Z\in\Cc$.

    To show that $X$ is in $\Cc$, we observe that there exists a commuting diagram
    \begin{equation*}
        \begin{xy}
            \square/^{ (}->`^{ (}->`^{ (}->`^{ (}->/[W`Y'`X`Y;```]
            \square(500,0)/->>`^{ (}->`=`->>/[Y'`Z`Y`Z;```]
        \end{xy}
    \end{equation*}
    in which the vertical maps are admissible monics in $\Dc$, and the top row is an exact sequence in $\Cc$. By \cite[Proposition 2.12]{Buh:10}, the left hand square is a pushout, and
    \begin{equation*}
        \coker(W\into X)\cong\coker(Y'\into Y).
    \end{equation*}
    We showed above that $\coker(Y'\into Y)$ is in $\Cc$. Because $\Cc$ is closed under extensions (Lemma \ref{lemma:lsextclosed}), we conclude that $X$ is in $\Cc$ as well.
\end{proof}

\section{The Structure of \texorpdfstring{$\Ind(P_f(R))$}{Ind(Pf(R))}\\ by J. \v S\v tov\'\i\v cek and J. Trlifaj}\label{sec:stovicektrlifaj}
Denote by $\mathcal {FM}(R)$ the category of all flat Mittag-Leffler modules over a ring $R$ (see Definition \ref{def:fml}). By Proposition \ref{prop:indfml}, $\Ind(P_f(R))$ is equivalent to the full sub-category $\Cc$ of $\mathcal {FM}(R)$ consisting of all $M \in \mathcal {FM}(R)$ such that $M$ is the direct limit of a direct system consisting of split monomorphisms between finitely generated projective modules. In this appendix, we will describe the structure and properties of $\Cc$ in more detail.

The close relation between $\Cc$ and $\mathcal {FM}(R)$ is clear from the following lemma:

\begin{lemma} \label{l:rel} Let $R$ be a ring and $M$ a module.
    \begin{enumerate}
        \item $M \in \mathcal {FM}(R)$, if and only if each finite (or countable) subset of $M$ is contained in a countably generated projective and pure sub-module of $M$.
        \item $M \in \Cc$, if and only if each finite subset of $M$ is contained in a finitely generated projective and pure sub-module of $M$.
        \item The following conditions are equivalent:
            \begin{enumerate}
                \item $\Cc = \mathcal {FM}(R)$,
                \item each (countably generated) projective module is a direct sum of finitely generated modules,
                \item $\Cc$ is closed under direct summands.
            \end{enumerate}
    \end{enumerate}
\end{lemma}
\begin{proof}
    (1) is well known (see \cite{RG:71} or \cite[Corollary 3.19]{GT:12}).

    (2) Assume that $M$ is a directed union of a direct system of finitely generated projective modules $(P_i \mid i \in I)$ such that for each $i \leq j \in I$, the inclusion of $P_i$ into $P_j$ splits. Then each $P_i$ is a pure sub-module of $M$ (because each $R$-linear system solvable in $M$ is solvable in some $P_j$ with $i \leq j \in I$, and the split projection provides for a solution in $P_i$). This proves the only-if part.

    Conversely, the assumption on $M$ makes it possible to express $M$ as a directed union of a direct system of finitely generated projective modules $(P_i \mid i \in I)$ such that each $P_i$ is pure in $M$. Then each embedding $P_i \subseteq P_j$ ($i \leq j \in I$) has a flat and finitely presented cokernel, hence it splits.

    (3) Clearly (a) implies (c), and (b) implies (a) by the characterizations in (1) and (2).

    Assume (b) fails, and let $P$ be a projective module which is not a direct sum of finitely generated modules. By Kaplansky's structure theorem for projective modules, we can assume that $P$ is countably generated. Then $P \notin \Cc$, since otherwise $P = \bigcup_{i < \omega} P_i$ where $P_i \subseteq P_{i+1}$ are split inclusions for all $i < \omega$, a contradiction. However, $P$ is a direct summand in a countably generated free module $F \in \Cc$, so (c) fails.
\end{proof}

\begin{remark}\label{r:semiher}
    Condition (b) of Lemma \ref{l:rel}(3) is known to hold in a number of cases, for example, when $R$ is a right semi-hereditary ring, or $R$ is semi-perfect (e.g., local), or $R$ is a commutative Noetherian domain (by Bass's Theorem).

    However, the class $\Cc$ is not closed under direct summands in general (that is, $\Ind(P_f(R))$ is not idempotent complete). A simple example of this phenomenon goes back to Kaplansky: there is a commutative ring $R$ and a countably generated projective ideal $P$ in $R$, such that $P$ is not a direct sum of finitely generated ideals in $R$ (see e.g.\ \cite[(2.12D)]{Lam:99}).
\end{remark}

$\mathcal {FM}(R)$ is always the closure of $\Cc$ under direct summands; in fact, we have

\begin{proposition} \label{p:more}
    Let $R$ be a ring, and $M \in \mathcal {FM}(R)$. Then $M \oplus R^{(\omega)} \in \Cc$.
\end{proposition}
\begin{proof}
    We will use Eilenberg's Trick: for each countably generated projective module $P$, there is an isomorphism $P \oplus R^{(\omega)} \cong R^{(\omega)}$. Take $M \in \mathcal {FM}(R)$ and consider the set $\mathcal P$ of all countably generated projective and pure sub-modules in $M$.  Then $\mathcal F = \{ P \oplus R^{(\omega)} \mid P \in \mathcal P \}$ is a set of countably generated free and pure sub-modules in $M \oplus R^{(\omega)}$. Since each finite subset of $M$ is contained in some $P \in \mathcal P$ by Lemma \ref{l:rel}(i), each finite subset of $M \oplus R^{(\omega)}$ is contained in a finitely generated, free and pure sub-module in $M \oplus R^{(\omega)}$. By Lemma \ref{l:rel}(ii), we conclude that $M \oplus R^{(\omega)} \in \Cc$.
\end{proof}

The matter is easy when $R$ is a perfect ring (for a right Noetherian ring, this just means that $R$ is right Artinian, so for commutative Noetherian rings, these are just the ones of Krull dimension $0$). Then $\mathcal {FM}(R) = \Cc$ is the class of all projective modules, and each projective module is a direct sum of the cyclic ones. In particular, all short exact sequences in $\Ind(P_f(R))$ split.

Things completely change in the case when $R$ is not right perfect (which is the case relevant for our context). Then $\mathcal {FM}(R)$ coincides with the class of all $\aleph_1$-projective modules, see \cite{HT:12} or \cite[Corollary 3.19]{GT:12}. This class was studied by the methods of set-theoretic homological algebra in \cite{EM:02}. In order to indicate its complexity, we only note that $\mathcal {FM}(R)$ is closed under transfinite extensions, but it cannot be obtained by transfinite extensions from any set of its elements, see \cite{HT:12} or \cite[Theorem 10.13]{GT:12}. Moreover, $\mathcal {FM}(R)$ is not a pre-covering class, \cite{AST}.

For each infinite cardinal $\kappa$, we denote by $\Cc _\kappa$ the image of $\Indk(P_f(R))$ in the equivalence of Proposition \ref{prop:indfml}. Since $\Cc _{\aleph_0}$ is the class of all countably generated projective modules, all short exact sequences in $\Indc(P_f(R))$ split (see also Proposition \ref{prop:indcsplit}). This completely fails for uncountable cardinalities:

\begin{proposition}\label{r:split}
    Assume that $R$ is not right perfect. Then for each infinite cardinal $\kappa$, there is a non-projective $\kappa ^+$-generated module\footnote{where $\kappa^+$ denotes the successor cardinal to $\kappa$.} $M_{\kappa^+} \in \Cc _{\kappa^+}$ whose canonical presentation in $\Cc _{\kappa^+}$ does not split.
\end{proposition}
\begin{proof}
    Let $\kappa$ be an infinite cardinal and $M \in \Cc _\kappa$, so $M = \varinjlim_{i \in I} P_i$ is the direct limit of a direct system consisting of split monomorphisms between finitely generated projective modules $P_i$, and $I$ is a directed set of cardinality $\leq \kappa$. Then there is a pure exact sequence $0 \to K \to P \to M \to 0$ coming from the presentation of $M$ as the pure-epimorphic image of the direct sum $P = \bigoplus_{i \in I}  P_i$.
    Since $K$ is a pure sub-module in $P$, also $K \in \mathcal {FM}(R)$, and $K$ is $\leq \kappa$-generated. By adding a countably generated free module $F$ (see Proposition \ref{p:more}), we obtain the canonical presentation of $M$ in $\Cc _\kappa$:
    $$0 \to (K \oplus F) \to (P \oplus F) \to M \to 0.$$
    The latter sequence only splits if $M$ is projective.

    The existence of a non-projective $\kappa ^+$-generated module $M_{\kappa^+} \in \mathcal {FM}(R)$ follows immediately from \cite[Theorem 10.13]{GT:12}. By Proposition \ref{p:more}, we can assume that $M_{\kappa^+} \in \Cc _{\kappa^+}$. By the above, the canonical presentation of $M_{\kappa^+}$ in $\Cc _{\kappa^+}$ does not split.
\end{proof}

\begin{remark} \label{r:gamma}
    In the case of $\kappa = \aleph_0$, the construction of the non-projective $\aleph_1$-generated module $M_{\aleph_1} \in \mathcal {FM}(R)$ in \cite[Section 10.2]{GT:12} can be modified so that $M_{\aleph_1}$ has an arbitrary prescribed non-zero $\Gamma$-invariant (in the sense of \cite[p.118]{EM:02}). This $\Gamma$-invariant measures the distance of $M_{\aleph_1}$ from being projective. It takes values in the Boolean algebra of all subsets of $\aleph_1$ modulo the ideal of all thin (= non-stationary) subsets.
\end{remark}

\bibliographystyle{amsplain}
\bibliography{masterthesis}
\end{document}